\documentclass[reqno, a4paper, 11pt]{amsart}
\usepackage{amsfonts, amsmath, amssymb, amsthm, color}
\usepackage[hmargin=2.5cm, vmargin=2.75cm]{geometry}
\usepackage{hyperref}

\newtheorem{thm}{Theorem}[section]
\newtheorem{cor}[thm]{Corollary}
\newtheorem{lem}[thm]{Lemma}
\newtheorem{prop}[thm]{Proposition}

\newtheorem{thmx}{Theorem}

\theoremstyle{definition}

\theoremstyle{remark}
\newtheorem{rem}[thm]{Remark}

\numberwithin{equation}{section}
\allowdisplaybreaks

\newcommand{\pa}{\partial}
\newcommand{\ep}{\epsilon}
\newcommand{\vep}{\varepsilon}

\newcommand{\bx}{\bar{x}}
\newcommand{\bw}{\bar{w}}

\newcommand{\N}{\mathbb{N}}
\newcommand{\R}{\mathbb{R}}
\renewcommand{\S}{\mathbb{S}}
\newcommand{\mfd}{\mathbf{d}}
\newcommand{\mca}{\mathcal{A}}
\newcommand{\mce}{\mathcal{E}}
\newcommand{\mcf}{\mathcal{F}}
\newcommand{\mch}{\mathcal{H}}
\newcommand{\mci}{\mathcal{I}}
\newcommand{\mcj}{\mathcal{J}}
\newcommand{\mck}{\mathcal{K}}
\newcommand{\mcl}{\mathcal{L}}
\newcommand{\mcr}{\mathcal{R}}
\newcommand{\mcv}{\mathcal{V}}
\newcommand{\mcw}{\mathcal{W}}
\newcommand{\mcx}{\mathcal{X}}
\newcommand{\mcy}{\mathcal{Y}}

\newcommand{\ty}{\tilde{y}}
\newcommand{\tta}{\tilde{\tau}}
\newcommand{\wtg}{\widetilde{G}}
\newcommand{\wth}{\widetilde{H}}
\newcommand{\wtom}{\widetilde{\Omega}}

\newcommand{\supp}{\textnormal{supp }}

\renewcommand{\(}{\left(}
\renewcommand{\)}{\right)}

\begin{document}
\title[]{Asymptotic behavior of least energy solutions to the Lane-Emden system near the critical hyperbola}

\author{Woocheol Choi}
\address[Woocheol Choi]{Department of Mathematics Education, Incheon National University, 12 Gaetbeol-ro Yeonsu-gu, Incheon 22012, Republic of Korea}
\email{choiwc@inu.ac.kr}

\author{Seunghyeok Kim}
\address[Seunghyeok Kim]{Department of Mathematics and Research Institute for Natural Sciences, Hanyang University, 222 Wangsimni-ro Seongdong-gu, Seoul 04763, Republic of Korea}
\email{shkim0401@gmail.com}

\begin{abstract}
The Lane-Emden system is written as
\begin{equation*}
\begin{cases}
-\Delta u = v^p &\text{in } \Omega,\\
-\Delta v = u^q &\text{in } \Omega,\\
u, v > 0 &\text{in } \Omega,\\
u = v = 0 &\text{on } \pa \Omega
\end{cases}
\end{equation*}
where $\Omega$ is a smooth bounded domain in the Euclidean space $\R^n$ for $n \ge 3$ and $0< p < q <\infty$.
The asymptotic behavior of least energy solutions near the critical hyperbola was studied by Guerra \cite{G} when $p \geq 1$ and the domain is convex.
In this paper, we cover all the remaining cases $p < 1$ and extend the results to any smooth bounded domain.
\end{abstract}

\date{\today}
\maketitle

\section{introduction}
In this paper, we consider the following elliptic system
\begin{equation}\label{eq-main}
\begin{cases}
-\Delta u = v^p &\text{in } \Omega,\\
-\Delta v = u^q &\text{in } \Omega,\\
u, v > 0 &\text{in } \Omega,\\
u = v = 0 &\text{on } \pa \Omega
\end{cases}
\end{equation}
where $\Omega$ is a smooth bounded domain in the Euclidean space $\R^n$ for $n \ge 3$ and $p, q \in (0, \infty)$.
This problem, often referred to the Lane-Emden system, has been a subject of strong interest to many researchers,
because it is one of the simplest Hamiltonian-type strongly coupled elliptic systems but yet has rich structure.

\subsection{Brief history and motivation}
The existence theory for system \eqref{eq-main} is associated with so-called the critical hyperbola
\begin{equation}\label{eq-cr-hy}
\frac{1}{p+1} + \frac{1}{q+1} = \frac{n-2}{n}
\end{equation}
introduced by Cl\'ement et al. \cite{CFM} and van der Vorst \cite{V}.
Thanks to the works of Hulshof and van der Vorst \cite{HV}, Figueiredo and Felmer \cite{FF} and Bonheure et al. \cite{BMR}, it is known that if $pq \ne 1$ and
\begin{equation}\label{pq-cr-1}
\frac{1}{p+1} + \frac{1}{q+1} > \frac{n-2}{n},
\end{equation}
then \eqref{eq-main} has a solution. In contrast, as shown by Mitidieri \cite{M1}, if the domain $\Omega$ is star-shaped and
\[\frac{1}{p+1} + \frac{1}{q+1} \le \frac{n-2}{n},\]
then \eqref{eq-main} has no solution.

To deduce the above results, the authors used the fact that system \eqref{eq-main} has a variational structure.
Indeed, a solution of \eqref{eq-main} can be characterized as a positive critical point of the energy functional
\[E(u,v) = \int_{\Omega} \nabla u \cdot \nabla v \, dx - \frac{1}{p+1} \int_{\Omega} |u|^{p+1} dx - \frac{1}{q+1} \int_{\Omega} |v|^{q+1} dx\]
defined for $(u,v) \in (H^1_0(\Omega))^2$.
We say that $(u, v)$ is a least energy solution to \eqref{eq-main} if it solves \eqref{eq-main} and attains the minimal value of $E$ among all nontrivial solutions.
It is well-known that there exists a least energy solution whenever $pq > 1$ and \eqref{pq-cr-1} is valid.

Once the existence theory is established, one of the next natural questions is to examine the shape of solutions.
A well-known method related to this issue is the moving plane method, which shows that symmetries of solutions are inherited from those of the equation and the domain,
and works also for \eqref{eq-main}.
A further important progress to this direction was achieved by Guerra \cite{G} where he investigated the precise profile of least energy solutions to \eqref{eq-main} on convex domains.
His result can be described in the following way:
Fix any number $p \ge 1$ that belongs to the interval $(\frac{2}{n-2}, \frac{n+2}{n-2}]$, and for each small $\ep > 0$, determine $q_{\ep}$ by
\begin{equation}\label{eq-pq-e}
\frac{1}{p+1} + \frac{1}{q_{\ep}+1} = \frac{n-2}{n} + \ep.
\end{equation}
Then $(p,q_{\ep})$ satisfies the subcriticality condition \eqref{pq-cr-1} and approaches the critical hyperbola as $\ep \to 0$.
Let $q_0$ be the limit of $q_{\ep}$ as $\ep \to 0$ so that $(p,q_0)$ satisfies \eqref{eq-cr-hy} and $p \le q_0$.
For a least energy solution $(u_{\ep}, v_{\ep})$ to \eqref{eq-main} with $q=q_{\ep}$, it holds that
\begin{equation}\label{eq-energy}
S_{\ep}(\Omega) = \frac{ \int_{\Omega} |\Delta u_{\ep}|^{p+1 \over p} dx} {\|u_{\ep}\|_{L^{q_{\ep}+1}(\Omega)}^{p+1 \over p}} \to S \quad \text{as } \ep \to 0
\end{equation}
where $S > 0$ is the best constant of the Sobolev inequality
\begin{equation}\label{eq-Sob}
\|u\|_{L^{q_0+1}(\R^n)} \le S^{-\frac{p}{p+1}} \| \Delta u\|_{L^{\frac{p+1}{p}}(\R^n)} \quad \text{for all } u \in C^{\infty}_c(\R^n).
\end{equation}
Let $G$ and $\tau$ be the Green's function and the Robin function of the Dirichlet Laplacian in $\Omega$, respectively.
Also, for $p \in (\frac{2}{n-2}, \frac{n}{n-2}]$, set by $\wtg$ the unique solution of
\begin{equation}\label{eq-h2-2}
\begin{cases}
-\Delta_x \wtg(x,y) = G^p(x,y) &\text{for } x \in \Omega,\\
\wtg(x,y) = 0 &\text{for } x \in \pa\Omega,
\end{cases}
\end{equation}
by $\wth$ the $C^1$-regular part of $\wtg$ (see \eqref{eq-h2-1} for its precise definition) and $\tta(x) = \wth(x,x)$ for each $x \in \Omega$.
Granted these notions, we have
\begin{thmx}[Theorem 1.1 of \cite{G}]\label{gue}
Suppose that $\Omega$ is a convex smooth bounded domain in $\R^n$, $p \ge 1$, $p \in (\frac{2}{n-2}, \frac{n+2}{n-2}]$ and $\ep > 0$ is sufficiently small.
Let $\{(u_{\ep}, v_{\ep})\}_{\ep > 0}$ be a family of least energy solutions to \eqref{eq-main} with $q = q_{\ep}$.
Then, along a subsequence, $(u_{\ep}, v_{\ep})$ blows-up at a point $x_0 \in \Omega$ as $\ep \to 0$,
which means that for any $\{\ep_k\}_{k \in \N}$ of small positive numbers such that $\ep_k \to 0$, we have
\[\max_{x \in \Omega} u_{\ep_k}(x) = u_{\ep_k}(x_{\ep_k}) \to \infty, \quad x_{\ep_k} \to x_0 \in \Omega
\quad \text{and} \quad
u_{\ep_k} \to 0 \quad \text{in } C_{\textnormal{loc}}(\overline{\Omega} \setminus \{x_0\})\]
as $k \to \infty$, up to subsequence. In addition, the followings are true:

\medskip \noindent \textnormal{(1)} If $p \in [\frac{n}{n-2}, \frac{n+2}{n-2}]$, then $x_0$ is a critical point of the Robin function $\tau$.
If $p \in (\frac{2}{n-2}, \frac{n}{n-2})$, then $x_0$ is a critical point of the function $x \in \Omega \mapsto \wth(x,x_0)$.

\medskip \noindent \textnormal{(2)} It holds that
\[\begin{cases}
\displaystyle \lim_{\ep \to 0} \ep \|u_{\ep}\|_{L^{\infty}(\Omega)}^{\frac{n}{(n-2)p-2}+1}
= S^{\frac{1-pq_0}{p(q_0+1)}} \|U_0\|_{L^{q_0}(\R^n)}^{q_0} \|V_0\|_{L^p(\R^n)}^p  |\tau(x_0)|
&\text{if } p \in (\frac{n}{n-2}, \frac{n+2}{n-2}], \\
\displaystyle \lim_{\ep \to 0} \frac{\ep \|u_{\ep}\|_{L^{\infty}(\Omega)}^{\frac{n}{n-2}+1}} {\log \|u_{\ep}\|_{L^{\infty}(\Omega)}}
= \frac{p+1}{n-2} L^{\frac{n}{n-2}} S^{\frac{1-pq_0}{p(q_0+1)}} \|U_0\|_{L^{q_0}(\R^n)}^{q_0} |\tau(x_0)|
&\text{if } p = \frac{n}{n-2}, \\
\displaystyle \lim_{\ep \to 0} \ep \|u_{\ep}\|_{L^{\infty}(\Omega)}^{p+1} = S^{\frac{1-pq_0}{p(q_0+1)}} \|U_0\|_{L^{q_0}(\R^n)}^{q_0(p+1)} |\tta(x_0)|
&\text{if } p \in (\frac{2}{n-2}, \frac{n}{n-2})
\end{cases}\]
where $(U_0, V_0)$ is a solution of the system
\begin{equation}\label{eq-entire}
\begin{cases}
-\Delta U_0 = V_0^p,\ -\Delta V_0 = U_0^{q_0} &\text{in } \R^n, \\
U_0, V_0 > 0 &\text{in } \R^n, \\
U_0(x), V_0(x) \to 0 &\text{as } |x| \to \infty,\\
U_0(0) = 1 = \max_{x \in \R^n} U_0(x)
\end{cases}
\end{equation}
and $L := \lim_{|x| \to \infty} |x|^{n-2} V_0(x) \in (0, \infty)$.

\medskip \noindent \textnormal{(3)} It holds that
\[\lim_{\ep \to 0} \|u_{\ep}\|_{L^{\infty}(\Omega)} v_{\ep}(x) = \|U_0\|_{L^{q_0}(\R^n)}^{q_0} G(x,x_0)\]
and
\[\begin{cases}
\displaystyle \lim_{\ep \to 0} \|u_{\ep}\|_{L^{\infty}(\Omega)}^{n \over (n-2)p-2} u_{\ep}(x) = \|V_0\|_{L^p(\R^n)}^p G(x,x_0)
&\text{if } p \in (\frac{n}{n-2}, \frac{n+2}{n-2}], \\
\displaystyle \lim_{\ep \to 0} \frac{\|u_{\ep}\|_{L^{\infty}(\Omega)}^{n \over n-2}}
{\log \|u_{\ep}\|_{L^{\infty}(\Omega)}} u_{\ep}(x) = \frac{p+1}{n-2} L^{n \over n-2} G(x,x_0)
&\text{if } p = \frac{n}{n-2}, \\
\displaystyle \lim_{\ep \to 0} \|u_{\ep}\|_{L^{\infty}(\Omega)}^p u_{\ep}(x) = \|U_0\|_{L^{q_0}(\R^n)}^{pq_0} \wtg(x,x_0)
&\text{if } p \in (\frac{2}{n-2}, \frac{n}{n-2})
\end{cases}\]
in $C^1_{\textnormal{loc}}(\Omega \setminus \{x_0\})$-sense.
\end{thmx}
\noindent The works of Chen et al. \cite{CLO} and Hulshof and Van der Vorst \cite{HV2} guarantee that the number $L$ is well-defined.

The condition $p \ge 1$ was used in \cite{G} when a decay estimate on suitably rescaled least energy solutions to \eqref{eq-main} was derived.
This kind of uniform estimate is one of the essential steps in asymptotic analysis of nonlinear elliptic problems, as well-known in the literatures.
On the basis of numerical tests, Guerra \cite{G} conjectured that the assumption on $p$ is just technical,
and Theorem \ref{gue} should hold even if $p \in (\frac{2}{n-2}, 1)$.
The first contribution of this paper is to give an affirmative answer to this conjecture.

The second contribution of this paper is to remove the convexity assumption in the above theorem.
As can be seen in \cite{G}, the convexity of the domain allows one to apply the moving plane method
in obtaining uniform boundedness of least energy solutions $(u_{\ep}, v_{\ep})$ near the boundary $\pa \Omega$ with respect to $\ep > 0$.
When it comes to the Lane-Emden equation, a special case of \eqref{eq-main} with the choice $p = q$ and $w = u = v$,
\begin{equation}\label{eq-LEF}
\begin{cases}
-\Delta w = w^{\frac{n+2}{n-2}-\ep} &\text{in } \Omega, \\
w > 0 &\text{in } \Omega, \\
w = 0 &\text{on } \pa \Omega,
\end{cases}
\end{equation}
one can treat general domains by employing the Kelvin transform to \eqref{eq-LEF} on small balls that touch $\pa \Omega$; refer to \cite{H}.
Unfortunately, this idea does not work well for \eqref{eq-main} if $p < \frac{n+2}{n-2}$; see e.g. \cite{Sou}.

To obtain the above results, we introduce two new ideas:
Firstly, to cover the case when $p$ is sub-linear, we perform a decay estimate by writing system \eqref{eq-main}
as a single non-local equation \eqref{eq-ext-sol} and applying a Brezis-Kato type argument.
Secondly, we obtain uniform boundedness of least energy solutions near $\pa \Omega$ from local Pohozaev-type identities and sharp pointwise estimates of the solutions,
not exploiting the Kelvin transform and the moving plane method.
See Subsection \ref{subsec-ideas} for more detailed explanations.

\subsection{Statement of the main theorems}
The main result of this paper is the following.
\begin{thmx} \label{thm-main}
Suppose that $p \in (\frac{2}{n-2}, \frac{n+2}{n-2})$ and $\Omega$ is any smooth bounded domain in $\R^n$. Then all the assertions in Theorem \ref{gue} remain true.
\end{thmx}
\begin{rem}
We have two remarks on the above theorem.

\medskip \noindent (1) For $p \in (\frac{2}{n-2}, \frac{n}{n-2})$, the proof of the above theorem uses Proposition \ref{prop-h2},
whose validity is reduced to that of \eqref{eq-as-0} or \eqref{eq-as-1} according to the value of $p$.
In Appendix \ref{subsec-ver}, we shall give an analytic derivation of \eqref{eq-as-0} for all $n \ge 5$ and $p \in (\frac{2}{n-2}, \frac{n-1}{n-2})$,
and that of \eqref{eq-as-1} for all $n \ge 100$ and $p \in [\frac{n-1}{n-2}, \frac{n}{n-2})$.
To derive \eqref{eq-as-1} given that $5 \le n \le 99$ and $p \in [\frac{n-1}{n-2}, \frac{n}{n-2})$, we further reduce it into an inequality involving the Gauss hypergeometric function $_2F_1$.
However, the complexity of the resulting inequality compels us to use a computer software for its verification; see Case 3 (ii) of Appendix \ref{subsec-ver} and the supplement \cite{CK2}.

\medskip \noindent (2) In the statement, it is enough to assume that $\pa \Omega$ is of class $C^2$
so that the principal curvatures of $\pa \Omega$ are well-defined and uniformly bounded.
\end{rem}

In order to prove the above theorem, we first need an adequate decay estimate for least energy solutions to \eqref{eq-main}.
\begin{thm}\label{thm-0}
Let us introduce two parameters
\[\alpha_{\ep} = \frac{2(p+1)}{pq_{\ep}-1} \quad \text{and} \quad \beta_{\ep} = \frac{2(q_{\ep}+1)}{pq_{\ep}-1},\]
and then choose a number $\lambda_{\ep} > 0$ and a point $x_{\ep} \in \Omega$ by
\begin{equation}\label{eq-lambda}
\lambda_{\ep}^{\alpha_{\ep}} = \max_{x \in \Omega} u_{\ep}(x) = u_{\ep} (x_{\ep}).
\end{equation}
Moreover, we normalize the solutions $(u_{\ep}, v_{\ep})$ to \eqref{eq-main} as
\begin{equation}\label{eq-tuv}
U_{\ep}(x) = \lambda_{\ep}^{-\alpha_{\ep}} u_{\ep} (\lambda_{\ep}^{-1}x + x_{\ep})
\quad \text{and} \quad
V_{\ep}(x) = \lambda_{\ep}^{-\beta_{\ep}} v_{\ep} (\lambda_{\ep}^{-1}x + x_{\ep})
\end{equation}
for $x \in \Omega_{\ep} := \lambda_{\ep} (\Omega - x_{\ep})$. If $(U_0, V_0)$ is a pair of functions in $C^{\infty}(\R^n,[0,1])$ satisfying
\[\lim_{r \to \infty} r^{n-2} V_0(r) = 1
\quad \text{and} \quad
\begin{cases}
\lim\limits_{r \to \infty} r^{n-2} U_0(r) = 1 &\text{if } p \in (\frac{n}{n-2}, \frac{n+2}{n-2}], \\
\lim\limits_{r \to \infty} \dfrac{r^{n-2}}{\log r} U_0(r) = 1 &\text{if } p = \frac{n}{n-2}, \\
\lim\limits_{r \to \infty} r^{(n-2)p-2} U_0(r) = 1 &\text{if } p \in (\frac{2}{n-2}, \frac{n}{n-2}),
\end{cases}\]
then there exists a constant $C > 0$ depending only on $n$, $s$, $p$ and $\Omega$ such that
\[U_{\ep} \le C U_0 \quad \text{and} \quad V_{\ep} \le C V_0 \quad \text{in } \Omega_{\ep}\]
for all $\ep > 0$ small.
\end{thm}

Once Theorem \ref{thm-0} is obtained, the most nontrivial part in the proof of Theorem \ref{thm-main} will be
to deduce that solutions $(u_{\ep}, v_{\ep})$ to \eqref{eq-main} are uniformly bounded near $\pa \Omega$ with respect to $\ep > 0$.
The cases $p \in [\frac{n}{n-2}, \frac{n+2}{n-2}]$ and $p \in  (\frac{2}{n-2}, \frac{n}{n-2})$ have to be treated separately.
\begin{thm}\label{thm-1}
Let $\Omega$ be a smooth bounded domain in $\R^n$ and $p \in [\frac{n}{n-2}, \frac{n+2}{n-2}]$.
Consider a family of solutions $\{(u_{\ep}, v_{\ep})\}_{\ep >0}$ to \eqref{eq-main} with $q = q_{\ep}$ for which \eqref{eq-energy} holds.
Then $(u_{\ep}, v_{\ep})$ are uniformly bounded near $\pa \Omega$ with respect to $\ep > 0$ small and blow-up at an interior point of $\Omega$.
\end{thm}

\begin{thm}\label{thm-2}
Let $\Omega$ be a smooth bounded domain in $\R^n$ and $p \in (\frac{2}{n-2}, \frac{n}{n-2})$.
Consider a family of solutions $\{( u_{\ep}, v_{\ep})\}_{\ep >0}$ to \eqref{eq-main} with $q = q_{\ep}$ for which \eqref{eq-energy} holds.
Then $(u_{\ep}, v_{\ep})$ are uniformly bounded near $\pa \Omega$ with respect to $\ep > 0$ small and blow-up at an interior point of $\Omega$.
\end{thm}
\noindent We emphasize that the above proposition is valid not only for $p \ge 1$ but for any $p > \frac{2}{n-2}$, which is a nontrivial fact when $n \ge 5$.

Under the validity of Theorems \ref{thm-0}-\ref{thm-2}, we can adapt the arguments in Guerra \cite{G} to conclude that Theorem \ref{thm-main} is indeed true.
A more detailed account will be given in the next subsection.

\subsection{Ideas behind the main theorems} \label{subsec-ideas}
In this subsection, we explain the key ideas on our proof of Theorem \ref{thm-main}.

\medskip
From the fact that the energy of a solution to \eqref{eq-main} with $q = q_0$ cannot be equal to the best constant $S$ of the Sobolev inequality \eqref{eq-Sob},
we see that if $(u_{\ep}, v_{\ep})$ is a least energy solution to \eqref{eq-main} with $q = q_{\ep}$ for $\ep > 0$ small, then
\[\max_{x \in \Omega} u_{\ep}(x) = u_{\ep}(x_{\ep}) \to \infty, \quad x_{\ep} \to x_0 \in \overline{\Omega}
\quad \text{and} \quad
u_{\ep} \to 0 \quad \text{in } C_{\text{loc}}(\overline{\Omega} \setminus \{x_0\})\]
as $\ep \to 0$.
As the next step, we obtain a decay estimate on a suitable rescaling of $(u_{\ep}, v_{\ep})$, which is the first main contribution of the paper.

\medskip \noindent
\textbf{$\star$ Proof of Theorem \ref{thm-0}: Decay estimate of rescaled least energy solutions}

\medskip
Here, we explain the technical difficulty related to the condition $p \ge 1$ imposed in \cite{G} and describe our strategy to overcome it.
To facilitate the reader's understanding, let us first recall the analysis of Han \cite{H} concerning a least energy solution $w_{\ep}$ to the single problem \eqref{eq-LEF}; refer also to de Figueiredo et al. \cite{FLN}.
We select a parameter $\mu_{\ep}$ and a point $y_{\ep} \in \Omega$ by the relation
\[\mu_{\ep}^{2(n-2) \over 4-(n-2)\ep} = \max_{y \in \Omega} w_{\ep}(y) = w_{\ep} (y_{\ep}).\]
Then one has that $\mu_{\ep} \to \infty$ and $\mu_{\ep}^{\ep} \to 1$ as $\ep \to 0$. We rescale the solution $w_{\ep}$ by
\[W_{\ep}(y) = \mu_{\ep}^{-{2(n-2) \over 4-(n-2)\ep}} w_{\ep}(\mu_{\ep}^{-1} y + y_{\ep}) \quad \text{for } y \in \wtom_{\ep} := \mu_{\ep} (\Omega - y_{\ep}).\]
It is a solution of
\[\begin{cases}
-\Delta W_{\ep} = W_{\ep}^{\frac{n+2}{n-2}-\ep} &\text{in } \wtom_{\ep}, \\
W_{\ep} > 0 &\text{in } \wtom_{\ep}, \\
W_{\ep} = 0 &\text{on } \pa \wtom_{\ep}.
\end{cases}\]
Now the Kelvin transform $W_{\ep}^*$ of $W_{\ep}$ defined by
\[W_{\ep}^*(y) = {1 \over |y|^{n-2}} W_{\ep}\({y \over |y|^2}\) \quad \text{in } \wtom_{\ep}^* := \left\{y \in \R^n: {y \over |y|^2} \in \wtom_{\ep} \right\}\]
satisfies
\[-\Delta W_{\ep}^* = \frac{1}{|y|^{(n-2)\ep}} (W_{\ep}^*)^{\frac{n+2}{n-2}-\ep}
= \left[ \frac{1}{|y|^{(n-2)\ep}} (W_{\ep}^*)^{\frac{4}{n-2}-\ep} \right] W_{\ep}^* \quad \text{in } \wtom_{\ep}^*.\]
By the least energy condition, it can be derived that
\[\sup_{\ep > 0} \left\| (W_{\ep}^*)^{{4 \over n-2}-\ep} \right\|_{L^{n/2}(\wtom_{\ep}^*)} \le C
\quad \text{and} \quad
\frac{1}{|y|^{(n-2)\ep}} \le 2 \text{ in } \wtom_{\ep}^*.\]
Then the Moser iteration argument yields
\[\sup_{\ep > 0} \|W_{\ep}^*\|_{L^{\infty}(B^n(0,1) \cap \wtom_{\ep}^*)} \le C,\]
from which one concludes that
\[W_{\ep}(y) \le {C \over |y|^{n-2}} \quad \text{for all } y \in \wtom_{\ep} \cap (\R^n \setminus B^n(0,1)).\]

Remarkably, it was discovered by Guerra \cite{G} that the above approach can be pursued to derive a decay estimate for a least energy solution $(u_{\ep}, v_{\ep})$ to \eqref{eq-main} with $p \ge 1$ and $q = q_{\ep}$ determined by \eqref{eq-pq-e}. If $(U_{\ep}, V_{\ep})$ is a pair of the functions defined by \eqref{eq-tuv}, then it holds that
\begin{equation}\label{eq-1-2}
\begin{cases}
-\Delta U_{\ep} = V_{\ep}^p,\, -\Delta V_{\ep} = U_{\ep}^{q_{\ep}} &\text{in } \Omega_{\ep},\\
U_{\ep}, V_{\ep} > 0 &\text{in } \Omega_{\ep},\\
U_{\ep} = V_{\ep} = 0 &\text{on } \pa \Omega_{\ep},\\
U_{\ep}(0) = 1 = \max_{x \in \Omega_{\ep}} U_{\ep}(x).
\end{cases}
\end{equation}
It follows that the Kelvin transform $(U_{\ep}^*, V_{\ep}^*)$ of $(U_{\ep}, V_{\ep})$ satisfies
\begin{equation}\label{eq-1-5}
\begin{cases}
-\Delta U_{\ep}^* = \left[ \dfrac{1}{|x|^{n+2-(n-2)p}} (V_{\ep}^*)^{p-1} \right] V_{\ep}^*, \\
-\Delta V_{\ep}^* = \left[ \dfrac{1}{|x|^{n+2-(n-2)q_{\ep}}} (U_{\ep}^*)^{q_{\ep}-1} \right] U_{\ep}^*
\end{cases}
\text{in } \Omega_{\ep}^* := \left\{x \in \R^n: {x \over |x|^2} \in \Omega_{\ep} \right\}.
\end{equation}
Then the least energy condition of the solutions $(u_{\ep}, v_{\ep})$ and a Brezis-Kato type estimate involving the weighted Hardy-Littlewood-Sobolev inequality are combined to show the uniform $L^{\infty}(B^n(0,1) \cap \Omega_{\ep}^*)$-bound of $V_{\ep}^*$,
which allows one to find the optimal decay of $U_{\ep}$ and $V_{\ep}$.
This procedure, however, breaks down if the exponent $p-1$ of $V_{\ep}^*$ in \eqref{eq-1-5} is negative.

To bypass this technical issue when dealing with the case $p < 1$, we write the system as
\begin{equation}\label{eq-ext-sol}
\begin{aligned}
V_{\ep} &= (-\Delta)^{-1} U_{\ep}^{q_{\ep}} = (-\Delta)^{-1} \((-\Delta)^{-1} V_{\ep}^p\)^{q_{\ep}} \\
&= (-\Delta)^{-1} \left[ \((-\Delta)^{-1} V_{\ep}^p\)^{q_{\ep}-\frac{1}{p}} \((-\Delta)^{-1} V_{\ep}^p\)^{\frac{1}{p}} \right]
\end{aligned}
\quad \text{in } \Omega_{\ep}
\end{equation}
instead of introducing the Kelvin transform. The crucial fact here is that $q_{\ep} - \frac{1}{p} > 0$ for any small $\ep >0$,
which enables us to apply H\"older's inequality on the corresponding term.
Then a Brezis-Kato type estimate on this integral equation with the aid of the weighted Hardy-Littlewood-Sobolev inequality shows
\[\sup_{\ep > 0}\left\| |x|^{(n-2) - \frac{n+\delta}{b}} V_{\ep}(x)\right\|_{L^b(\Omega_{\ep} \cap (\R^n \setminus B^n(0,1)))} \le C\]
for a small fixed $\delta > 0$ and any large $b > 1$. By inserting it into \eqref{eq-ext-sol}, we derive
\[V_{\ep}(x) \le {C \over |x|^{n-2}} \quad \text{for all } x \in \Omega_{\ep} \cap (\R^n \setminus B^n(0,1)).\]
Putting this estimate into \eqref{eq-1-2}, we also obtain the sharp estimate of $U_{\ep}$.

\medskip
Next, we explain our strategy to verify that $(u_{\ep}, v_{\ep})$ is uniformly bounded near $\pa \Omega$, i.e., the blow-up point $x_0$ belongs to $\Omega$, which is the second main contribution of the paper.

\medskip \noindent
\textbf{$\star$ Proof of Theorem \ref{thm-1}: Uniform estimate near the boundary for $p \in [\frac{n}{n-2}, \frac{n+2}{n-2}]$}

\medskip
In \cite{G}, the moving plane method was used as a crucial tool in the proof of uniform boundedness of solutions $(u_{\ep}, v_{\ep})$ near $\pa \Omega$, which requires the convexity of the domain.
Here we will use a local Pohozaev-type identity near the blow-up point and the boundary behavior of the Green's function $G$ of the Dirichlet Laplacian $-\Delta$ in $\Omega$.
Our approach is applicable for any smooth bounded domain and can be divided into three steps.

\medskip
\noindent \textbf{Step 1}. We assume the contrary and formulate a local Pohozaev-type identity which we will use to derive a contradiction; see \eqref{eq-poho-0}.

\medskip
\noindent \textbf{Step 2}. The left-hand side of \eqref{eq-poho-0} is the sum of integrations involving derivatives of $(u_{\ep}, v_{\ep})$
whose domains are small circles centered at $x_{\ep}$.
It is estimated in terms of derivatives of the regular part $H$ of the Green's function $G$; see \eqref{eq-g-decom}.
By applying the gradient estimate \eqref{eq-hn} of $H$ near $\pa \Omega$, we obtain its lower bound.

\medskip
\noindent \textbf{Step 3}.
The right-hand side of \eqref{eq-poho-0} is the sum of integrations involving $(u_{\ep}, v_{\ep})$ themselves
whose domains are small circles centered at $x_{\ep}$.
Employing the decay estimate of rescaled least energy solutions, we get its upper bound.
It turns out that the upper bound does not match with the lower bound obtained in the previous step, so we get a contradiction.

\medskip
Because of technical reasons, the cases $p \in (\frac{n}{n-2}, \frac{n+2}{n-2}]$ and $p = \frac{n}{n-2}$ will be dealt with separately.

\medskip \noindent
\textbf{$\star$ Proof of Theorem \ref{thm-2}: Uniform estimate near the boundary for $p \in (\frac{2}{n-2}, \frac{n}{n-2})$}

\medskip
In this range of $p$, we must handle both the function $\wtg:\Omega \times \Omega \to \R$ defined by \eqref{eq-h2-2} and the Green's function $G$ together.

In \eqref{eq-h2-1}, we will define the $C^1$-regular part $\wth$ of the function $\wtg$, which plays a similar role to the regular part $H$ of the Green's function $G$.
However, deducing the information on the boundary behavior of $\wth$ is much more involved than getting that of $H$.
Here we will analyze $\wth$ by dividing the two cases according to the value of $p$ and carefully examining its representation formula in each case.
Once it is done, we can argue as in the case $p \in [\frac{n}{n-2}, \frac{n+2}{n-2}]$, but still a more careful treatment is needed.
Particularly, estimate of $v_{\ep}$ should be sharpened.

\medskip
Once we know that the blow-up should occur at an interior point of $\Omega$, i.e., $x_0 \in \Omega$, deducing Theorem \ref{thm-main} becomes a standard task.
Indeed, given the upper estimate of solutions $(u_{\ep}, v_{\ep})$ to \eqref{eq-main} and the fact that their maximum points are uniformly bounded away from $\pa \Omega$,
one can show that the $L^{\infty}$-normalizations of $(u_{\ep}, v_{\ep})$ converge to constant multiples of the Green's function $G$ or its relative $\wtg$, as stated in Theorem \ref{gue} (3).
Then, putting this information into Pohozaev-type identities, one can characterize the blow-up rate and location in the form of Theorem \ref{gue} (1) and (2).

\subsection{Related literatures.}
As already mentioned, if $p = q$ and $u = v$, system \eqref{eq-main} is reduced to a single equation \eqref{eq-LEF}.
For this problem, the asymptotic behavior as $\ep \to 0^+$ has been thoroughly studied in a series of papers.
Han \cite{H} and Rey \cite{R} studied asymptotic behavior of least energy solutions.
The papers of Bahri et al. \cite{BLR} and Rey \cite{R3} were devoted to asymptotic behavior of finite energy solutions.
Applying the Lyapunov-Schmidt reduction method or theory of critical points at infinity, Rey \cite{R, R3}, Bahri et al. \cite{BLR} and Musso and Pistoia \cite{MP} constructed multi-peak solutions.
We remark that many techniques developed for the study of \eqref{eq-LEF} do not work well for system \eqref{eq-main}.

If $p = 1$, problem \eqref{eq-main} is reduced to the biharmonic equation
\[\begin{cases}
(-\Delta)^2 u = u^q &\text{in } \Omega, \\
u >0 &\text{in } \Omega, \\
u = \Delta u = 0 &\text{on } \pa \Omega.
\end{cases}\]
Asymptotic behavior of least energy solutions as $q \to (\frac{n+4}{n-4})^-$ was studied by Chou and Geng \cite{CG}, Geng \cite{Ge}, Ben Ayed and El Mehdi \cite{BE} and El Mehdi \cite{El}.
Our argument is close to that in \cite{Ge}, but depends on the Pohozaev-type identity and \eqref{eq-h2} below more directly.

Before finishing this subsection, we mention the Brezis-Nirenberg type problem
\begin{equation}\label{eq-BN}
\begin{cases}
-\Delta u = v^p + \mu_1 v &\text{in } \Omega, \\
-\Delta v = u^q + \mu_2 u &\text{in } \Omega, \\
u,\, v > 0 &\text{in } \Omega, \\
u = v = 0 &\text{on } \pa \Omega,
\end{cases}
\end{equation}
where $(p,q)$ satisfies \eqref{eq-cr-hy} and $\mu_1, \mu_2 > 0$.
Hulshof et al. \cite{HMV} obtained nontrivial solutions to \eqref{eq-BN} for $0 < \mu_1 \mu_2 < \lambda_1 (\Omega)^2$
where $\lambda_1(\Omega)$ is the first eigenvalue of the Dirichlet Laplacian $-\Delta$ in $\Omega$.
Asymptotic behavior of least energy solutions as $(\mu_1, \mu_2) \to (0,0)$ was studied in \cite{G} provided that $\Omega$ is a convex smooth bounded domain.
We believe that the arguments presented in this paper can be used to remove the convexity assumption for problem \eqref{eq-BN}.

\subsection{Organization of the paper}
This paper is organized as follows.

In Section \ref{sec-pre}, we examine properties of the Green's function $G$, its relative $\wtg$ and their regular parts $H$ and $\wth$.

In Section \ref{sec-pre-2}, we show that least energy solutions $(u_{\ep}, v_{\ep})$ to \eqref{eq-main} with $q = q_{\ep}$ should blow-up as $\ep \to 0$.
We also study behavior of the blow-up rates and the blow-up points,
and derive a local Pohozaev-type identity which will be used as an indispensable tool throughout the paper.

In Section \ref{sec-dec}, we obtain a sharp decay estimate on rescaled functions of $(u_{\ep}, v_{\ep})$ for all $p \in (\frac{2}{n-2}, \frac{n+2}{n-2}]$, which is the content of Theorem \ref{thm-0}.

In Section \ref{sec-vep}, under the assumption that the blow-up point $x_{\ep}$ tends to $\pa \Omega$,
we express the asymptotic behavior of $v_{\ep}$ near $x_{\ep}$ as $\ep \to 0$ in terms of the Green's function $G$.

In Section \ref{sec-thm-11}, we prove Theorem \ref{thm-1} for the case that $p \in (\frac{n}{n-2}, \frac{n+2}{n-2}]$.
Under the assumption that $x_{\ep}$ tends to $\pa \Omega$, we describe the asymptotic behavior of $u_{\ep}$ near $x_{\ep}$ as $\ep \to 0$ in terms of the function $G$.
Then we derive a contradiction using the local Pohozaev-type identity.

In Section \ref{sec-thm-12}, we modify this argument to cover the case $p = \frac{n}{n-2}$, completing the proof of Theorem \ref{thm-1}.

In Section \ref{sec-thm-2}, we prove Theorem \ref{thm-2} which concerns when $p \in (\frac{2}{n-2}, \frac{n}{n-2})$.
To handle the case, we analyze the asymptotic behavior of $u_{\ep}$ near $x_{\ep}$ as $\ep \to 0$ more carefully.
Then a desired contradiction will be derived from the local Pohozaev-type identity.

In Appendices \ref{sec-app-1} and \ref{sec-app-2}, we deduce regularity and pointwise estimate of the regular part $\wth$ of the function $\wtg$ defined by \eqref{eq-h2-1} and \eqref{eq-h2-2}.

\subsection{Notations}
We list some notational conventions which will be used throughout the paper.

\medskip \noindent - $\{(u_{\ep}, v_{\ep})\}_{\ep >0}$ always represents a family of solutions to \eqref{eq-main} with $(p,q_{\ep})$ satisfying \eqref{eq-pq-e} and the least energy condition \eqref{eq-energy}.

\medskip \noindent - For $n \in \N$, let $\R^n_+ = \R^{n-1} \times (0, \infty)$, $\R^n_- = \R^{n-1} \times (-\infty, 0)$
and $B^n(x_0, r) = \{x \in \R^n: |x-x_0| < r\}$ for each $x_0 \in \R^n$ and $r > 0$.

\medskip \noindent - For $x \in \Omega$, we denote the distance from $x$ to $\pa \Omega$ by $\text{dist}(x,\pa \Omega)$ or $\mfd(x)$.

\medskip \noindent - For $x, y, z \in \mathbb{C}$, $\Gamma(z)$ is the Gamma function and $\text{B}(x,y) = \frac{\Gamma(x)\Gamma(y)}{\Gamma(x+y)}$ is the Beta function.

\medskip \noindent - $|\S^{n-1}| = 2\pi^{n/2}/\Gamma(\frac{n}{2})$ is the Lebesgue measure of the $(n-1)$-dimensional unit sphere $\S^{n-1}$.

\medskip \noindent - The surface measure is denoted as $dS_x$ where $x$ is the variable of the integrand.

\medskip \noindent - $C > 0$ is a generic constant that may vary from line to line.

\section{Green's function and its relatives}\label{sec-pre}
In this section, we are concerned with the Green's function $G$, its regular part $H$, the function $\wtg$ defined by \eqref{eq-h2-2} and its $C^1$-regular part $\wth$.
More precisely, we will obtain pointwise estimates of $G$ and $H$ that will be used throughout the paper, and those of $\wtg$ and $\wth$ that will be crucial when we consider the case $p \in (\frac{2}{n-2}, \frac{n}{n-2})$.

\subsection{Green's function and its regular part}\label{subsec-Green}
Let $G$ be the Green's function of the Laplacian $-\Delta$ in $\Omega$ with the Dirichlet boundary condition.
If $H: \Omega \times \Omega \to \R$ is the function satisfying
\[\begin{cases}
-\Delta_x H(x,y) = 0 &\text{for } x \in \Omega, \\
H(x,y) = \dfrac{c_n}{|x-y|^{n-2}} &\text{for } x \in \pa \Omega
\end{cases}\]
where $c_n := (n-2)^{-1}|\S^{n-1}|^{-1}$. Then $G$ can be divided into
\begin{equation}\label{eq-g-decom}
G(x,y) = G(y,x) = \frac{c_n}{|x-y|^{n-2}} - H(x,y) \quad \text{for all } (x,y) \in \Omega \times \Omega,\, x \ne y.
\end{equation}
Take a sufficiently small constant $\delta > 0$. Then, for any $x \in \Omega$ such that $\mfd(x) < \delta$, there exists the unique unit vector $\nu_x \in \S^{n-1}$ such that $x + \mfd(x)\nu_x \in \pa \Omega$.
If $x^* := x + 2\mfd(x)\nu_x$, then we have the following result.
\begin{lem}\label{lem-H-est}
There exists a constant $C > 0$ such that
\begin{equation}\label{eq-h1-1}
\left| H(x,y) - \frac{c_n}{|y-x^*|^{n-2}} \right| \le \frac{C\mfd(x)}{|y-x^*|^{n-2}},
\end{equation}
\begin{equation}\label{eq-h1-2}
\left| \nabla_x H(x,y) - \nabla_x \(\frac{c_n}{|y-x^*|^{n-2}}\) \right| \le \frac{C}{|y-x^*|^{n-2}},
\end{equation}
\begin{equation}\label{eq-h1-3}
\left| \nabla_y H(x,y) + \frac{(n-2)c_n(y-x^*)}{|y-x^*|^n} \right| \le \frac{C\mfd(x)}{\mfd(y)|y-x^*|^{n-2}}
\end{equation}
and
\begin{equation}\label{eq-h1-4}
\left|\nabla_x \nabla_y H(x,y) + (n-2)c_n \nabla_x \(\frac{y-x^*}{|y-x^*|^n}\) \right| \\
\le C\(\frac{1}{\mfd(y)|y-x^*|^{n-2}} + \frac{1}{|y-x^*|^{n-1}}\)
\end{equation}
for all $(x,y) \in \Omega \times \Omega$ such that $\mfd(x) < \delta$. In particular, if we choose $C > 0$ small enough, then it holds that
\begin{equation}\label{eq-hn}
\nu_x \cdot \nabla_x H(x,y)|_{y = x} \ge C\mfd(x)^{-(n-1)}
\end{equation}
for any $x \in \Omega$ with $\mfd(x) < \delta$.
\end{lem}
\begin{proof}
For the derivation of \eqref{eq-h1-1} and \eqref{eq-h1-2}, see the proof of Lemma A.1 of \cite{ACP}.
Estimates \eqref{eq-h1-3} and \eqref{eq-h1-4} can be achieved in the same way.
Putting $y = x$ in \eqref{eq-h1-3}, we obtain \eqref{eq-hn}.
\end{proof}

\begin{cor}\label{cor-H-est}
For all $x \ne y \in \Omega$, there exists a constant $C > 0$ such that
\begin{equation}\label{eq-g1-9}
0 < G(x,y) < \frac{c_n}{|x-y|^{n-2}} \quad \text{and} \quad |\nabla_x G(x,y) | \le \frac{C}{|x-y|^{n-1}}.
\end{equation}
\end{cor}
\begin{proof}
The first estimate easily follows from the strong maximum principle. For the second estimate, refer to the proof of Lemma A.1 of \cite{ACP}.
\end{proof}

\subsection{The function $\wtg$ and its regular part $\wth$}
Recall the function $\wtg: \Omega \times \Omega \to \R$ defined by \eqref{eq-h2-2}.
We set its $C^1$-regular part $\wth: \Omega \times \Omega \to \R$ by
\begin{equation}\label{eq-h2-1}
\wth(x,y) = \begin{cases}
\dfrac{\gamma_1}{|x-y|^{(n-2)p-2}} - \wtg(x,y) &\text{if } p \in (\frac{2}{n-2}, \frac{n-1}{n-2}), \\
\dfrac{\gamma_1}{|x-y|^{(n-2)p-2}} - \dfrac{\gamma_2 H(x,y)}{|x-y|^{(n-2)p-n}} - \wtg(x,y) &\text{if } p \in [\frac{n-1}{n-2}, \frac{n}{n-2})
\end{cases}
\end{equation}
where
\begin{equation}\label{eq-gamma}
\gamma_1 := \frac{c_n^p}{[(n-2)p-2][n-(n-2)p]}
\quad \text{and} \quad
\gamma_2 := \frac{pc_n^{p-1}}{[(n-2)p-2(n-1)][n-(n-2)p]}.
\end{equation}

\begin{lem}\label{lem-wth-reg}
For each $y \in \Omega$, the function $x \in \Omega \to \wth(x,y)$ is contained in $C^1_{\textnormal{loc}}(\Omega)$.
\end{lem}
\begin{proof}
Its proof is deferred to Appendix \ref{sec-app-1}.
\end{proof}

The following result, which is analogous to \eqref{eq-hn}, is turned out to be highly nontrivial in general.
For the special case $p = 1$ in which the function $\wtg$ depends on $G$ linearly (see \eqref{eq-h2-2}), there is a simple proof due to Geng \cite[Proposition 2]{Ge}.
\begin{prop}\label{prop-h2}
For any $n \ge 5$ and $p \in (\frac{2}{n-2}, \frac{n}{n-2})$, there exist small constants $C > 0$ and $\delta > 0$ such that
\begin{equation}\label{eq-h2}
\nu_x \cdot \nabla_x \wth(x,y)|_{y = x} \ge C \mfd(x)^{1-(n-2)p}
\end{equation}
for $x \in \Omega$ with $\mfd(x) = \textnormal{dist}(x,\pa \Omega) < \delta$.
Here $\nu_x \in \S^{n-1}$ is the vector such that $x + \mfd(x)\nu_x \in \pa \Omega$.
\end{prop}
\begin{proof}
We postpone the proof until Appendix \ref{sec-app-2}.
\end{proof}

\section{Preliminary results concerning blow-up}\label{sec-pre-2}
For a family of least energy solutions $\{(u_{\ep}, v_{\ep})\}_{\ep >0}$ to \eqref{eq-main}, we set the blow-up rate $\lambda_{\ep}$ and the blow-up point $x_{\ep}$ as in \eqref{eq-lambda}.
\begin{lem}\label{lem-lam}
It holds that
\[\lim_{\ep \to 0} \lambda_{\ep}\textnormal{dist}(x_{\ep},\pa\Omega) = \infty \quad \text{and} \quad \lim_{\ep \to 0} \lambda_{\ep}^{\ep} = 1.\]
\end{lem}
\begin{proof}
Consult the proof of Lemma 4.3 of \cite{CK}. It works in our case as well, once the order $s$ of the fractional Laplacian $(-\Delta)^s$ is taken to be $1$.
\end{proof}
\noindent We set $d_{\ep} = \frac{1}{4} \text{dist}(x_{\ep},\pa\Omega)$ and $\Lambda_{\ep} = d_{\ep} \lambda_{\ep}$. Then, we see from Lemma \ref{lem-lam} that
\[ \Lambda_{\ep} \to \infty \quad \text{and} \quad \Omega_{\ep} = \lambda_{\ep} (\Omega - x_{\ep}) \to \R^n \quad \text{as } \ep \to 0.\]
By elliptic regularity, $(U_{\ep}, V_{\ep})$ in \eqref{eq-tuv} converges to a solution $(U_0, V_0) \in L^{q_0+1}(\R^n) \times L^{p+1}(\R^n)$ to \eqref{eq-entire} in $C_{\text{loc}}^2(\R^n)$.
By the result of Chen et al. \cite{CLO}, $U_0$ and $V_0$ are radially symmetric for $p \ge 1$.
In addition, Hulshof and Van der Vorst \cite{HV2} showed that if $(U_0, V_0)$ is a ground state to \eqref{eq-entire}, there exist positive numbers $a$, $b_1$, $b_2$ and $b_3$ such that
\[\lim_{r \to \infty} r^{n-2} V_0(r) = a
\quad \text{and} \quad
\begin{cases}
\lim\limits_{r \to \infty} r^{n-2} U_0(r) = b_1 &\text{if } p \in (\frac{n}{n-2}, \frac{n+2}{n-2}], \\
\lim\limits_{r \to \infty} \dfrac{r^{n-2}}{\log r} U_0(r) = b_2 &\text{if } p = \frac{n}{n-2}, \\
\lim\limits_{r \to \infty} r^{(n-2)p-2} U_0(r) = b_3 &\text{if } p \in (\frac{2}{n-2}, \frac{n}{n-2}).
\end{cases}\]
Given that $p \ge 1$ and $p \in (\frac{2}{n-2}, \frac{n+2}{n-2}]$, it was also proved in Lemma 2.3 of \cite{G} that
\[U_{\ep} \le C U_0 \quad \text{and} \quad V_{\ep} \le CV_0 \quad \text{in } \Omega_{\ep}\]
for some constant $C > 0$ and all small $\ep > 0$.

We conclude this section with a local Pohozaev-type identity for problem \eqref{eq-main}.
\begin{lem}\label{lem-poho}
Suppose that $(u,v) \in C^2 (\Omega) \times C^2 (\Omega)$ is a solution of \eqref{eq-main} and $D$ is an arbitrary smooth open subset of $\Omega$.
Then, for $1 \le j \le n$,
\begin{multline}\label{eq-poho-1}
-\int_{\pa D} \( \frac{\pa u}{\pa \nu} \frac{\pa v}{\pa x_j} + \frac{\pa v}{\pa \nu} \frac{\pa u}{\pa x_j} \) dS_x + \int_{\pa D} (\nabla u \cdot \nabla v) \nu_j\, dS_x \\
= \frac{1}{p+1} \int_{\pa D} v^{p+1} \nu_j\, dS_x + \frac{1}{q+1} \int_{\pa D} u^{q+1} \nu_j\, dS_x
\end{multline}
where the map $\nu = (\nu_1, \cdots, \nu_n): \pa D \to \R^n$ is the outward pointing unit normal vector on $\pa D$.
\end{lem}
\begin{proof}
Multiplying the first equation of \eqref{eq-main} by $\frac{\pa v}{\pa x_j}$ and the second equation by $\frac{\pa u}{\pa x_j}$,
integrating the results over the set $D$ and performing integration by parts, we obtain
\begin{equation}\label{eq-poho-2}
-\int_{\pa D} \frac{\pa u}{\pa \nu} \frac{\pa v}{\pa x_j} dS_x + \int_D \nabla u \cdot \frac{\pa \nabla v}{\pa x_j}\, dx = \frac{1}{p+1} \int_{\pa D} v^{p+1} \nu_j\, dS_x
\end{equation}
and
\begin{equation}\label{eq-poho-3}
-\int_{\pa D} \frac{\pa v}{\pa \nu} \frac{\pa u}{\pa x_j} dS_x + \int_D \nabla v \cdot \frac{\pa \nabla u}{\pa x_j}\, dx = \frac{1}{q+1} \int_{\pa D} v^{q+1} \nu_j\, dS_x.
\end{equation}
By combining \eqref{eq-poho-2} and \eqref{eq-poho-3}, and integrating by parts, we deduce \eqref{eq-poho-1}.
\end{proof}

\section{Proof of Theorem \ref{thm-0}}\label{sec-dec}
This section is devoted to the derivation of the next result.
\begin{prop}\label{prop-dec}
Suppose that $n \ge 3$ and $p \in (\frac{2}{n-2}, \frac{n+2}{n-2}]$.
For rescaled solutions $(U_{\ep}, V_{\ep})$ defined in \eqref{eq-tuv}, we have
\begin{equation}\label{eq-4-3}
V_{\ep}(x) \le \frac{C}{1+|x|^{n-2}} \quad \text{and} \quad
\begin{cases}
U_{\ep}(x) \le \dfrac{C}{1+|x|^{n-2}} &\text{if } p \in (\frac{n}{n-2}, \frac{n+2}{n-2}],
\\
U_{\ep}(x) \le \dfrac{C \log (1+|x|)}{1+|x|^{n-2}} &\text{if } p = \frac{n}{n-2},
\\
U_{\ep}(x) \le \dfrac{C}{1+|x|^{(n-2)p-2}} &\text{if } p \in (\frac{2}{n-2}, \frac{n}{n-2})
\end{cases}
\end{equation}
provided $\ep > 0$ sufficiently small.
\end{prop}
\noindent Since the above proposition is already known for $p \in [1, \frac{n+2}{n-2}]$ (see \cite{G}),
we only consider when $p$ is contained in the interval $(\frac{2}{n-2}, 1)$ that is nonempty for $n \ge 5$.
For its proof, we will apply a Brezis-Kato type argument, employing the next lemmas as key tools.
\begin{lem}[Doubly weighted Hardy-Littlewood-Sobolev inequality \cite{SW}]
Suppose that $1 < a, b < \infty$, $0 < \lambda < n$, $\alpha + \beta \ge 0$,
\[1 - \frac{1}{a} - \frac{\lambda}{n} < \frac{\alpha}{n} < 1 - \frac{1}{a}
\quad \text{and} \quad
\frac{1}{a} + \frac{1}{b} + \frac{\lambda + \alpha +\beta}{n} = 2.\]
Then there exists a constant $C > 0$ depending only on $\alpha$, $\beta$, $a$, $\lambda$ and $n$ such that
\[\left| \int_{\R^n} \int_{\R^n} \frac{f(x)g(y)}{|x|^{\alpha} |x-y|^{\lambda} |y|^{\beta}} dx dy \right| \le C \|f\|_{L^a(\R^n)} \|g\|_{L^b(\R^n)}.\]
\end{lem}
\noindent By applying the duality argument and putting $\beta = - \alpha$ and $\lambda = n-2$, we deduce
\begin{lem}
Suppose that $a$ and $b$ obey that $1 < a,b < \infty$,
\[\frac{1}{b} - \frac{n-2}{n} < \frac{\alpha}{n} < \frac{1}{b}
\quad \text{and} \quad
\frac{1}{a} - \frac{1}{b} = \frac{2}{n}.\]
Then we have
\begin{equation}\label{eq-1-1}
\left\| |\cdot|^{-\alpha} \( |\cdot|^{-(n-2)} * g \) \right\|_{L^b(\R^n)} \le C \||\cdot|^{-\alpha} g\|_{L^a(\R^n)}.
\end{equation}
\end{lem}

\medskip
Throughout the proof of Proposition \ref{prop-dec}, we denote the inverse operator of the Dirichlet Laplacian in $\Omega_{\ep}$
and that in $\R^n$ as $(-\Delta_{\ep})^{-1}$ and $(-\Delta_{\R^n})^{-1}$, respectively. It holds that
\begin{equation}\label{eq-inv}
(-\Delta_{\ep})^{-1} g \le (-\Delta_{\R^n})^{-1} g = \frac{1}{(n-2)|\S^{n-1}|} \(|\cdot|^{-(n-2)} * g\)
\end{equation}
for any nonnegative function $g$ such that $\supp g \subset \overline{\Omega}_{\ep}$.

As a starting point of the proof, we concern integrability of $V_{\ep}$.
Given a fixed large number $R > 0$, we decompose $V_{\ep} = V_{\ep i} + V_{\ep o}$ where
\begin{equation}\label{eq-43}
V_{\ep i} = \chi_{B^n(0,R)} V_{\ep}
\quad \text{and} \quad
V_{\ep o} = \chi_{\R^n \setminus B^n(0,R)} V_{\ep} \quad \text{in } \Omega_{\ep}
\end{equation}
and $\chi_D$ is the characteristic function of a set $D \subset \Omega_{\ep}$.
\begin{lem}\label{lem-Fep}
Let
\begin{equation}\label{eq-Fep}
F_{\ep} = \((-\Delta_{\ep})^{-1} V_{\ep o}^p\)^{q_{\ep}-\frac{1}{p}}.
\end{equation}
For any given number $\eta > 0$, we may choose a large number $R > 1$ in \eqref{eq-43} independent of $\ep > 0$ such that
\[\|F_{\ep}\|_{L^{p(q_{\ep}+1) \over pq_{\ep}-1}(\Omega_{\ep})} \le \eta.\]
\end{lem}
\begin{proof}
It holds that
\begin{equation}\label{eq-44}
\begin{aligned}
\(\int_{\Omega_{\ep}} F_{\ep}^{p(q_{\ep}+1) \over pq_{\ep}-1} dx\)^{1 \over q_{\ep}+1} &= \left\|(-\Delta_{\ep})^{-1} V_{\ep o}^p\right\|_{L^{q_{\ep}+1}(\Omega_{\ep})}
\le \left\||\cdot|^{-(n-2)} \ast V_{\ep o}^p\right\|_{L^{q_{\ep}+1}(\Omega_{\ep})} \\
&\le C \lambda_{\ep}^{\ep} \left\|V_{\ep o}^p\right\|_{L^{p+1 \over p}(\Omega_{\ep})} \le C \|V_{\ep o}\|_{L^{p+1}(\Omega_{\ep})}^p
\end{aligned}
\end{equation}
where the last inequality holds due to Lemma \ref{lem-lam}. In addition, by \eqref{eq-tuv}, \eqref{eq-1-2}, Lemma \ref{lem-lam} and \eqref{eq-energy},
\[\int_{\Omega_{\ep}} U_{\ep}^{q_{\ep}+1} dx = \int_{\Omega_{\ep}} \nabla U_{\ep} \cdot \nabla V_{\ep}\, dx
= \int_{\Omega_{\ep}} |\Delta U_{\ep}|^{p+1 \over p} dx = \int_{\Omega_{\ep}} V_{\ep}^{p+1} dx \to S^{{n \over 2} \cdot {p \over p+1}}\]
as $\ep \to 0$. Elliptic regularity tells us that
\begin{equation}\label{eq-Vi}
\sup_{\ep > 0} \|V_{\ep i}\|_{L^{\infty}(\Omega_{\ep})} = \sup_{\ep > 0} \|V_{\ep}\|_{L^{\infty}(B^n(0,R))} < \infty
\end{equation}
and $(U_{\ep}, V_{\ep})$ converges to a solution $(U_0, V_0) \in L^{q_0+1}(\R^n) \times L^{p+1}(\R^n)$ to \eqref{eq-entire} in $C_{\text{loc}}^2(\R^n)$.
Also, $V_{\ep} \rightharpoonup V_0$ in $L^{p+1}(\R^n)$ weakly. Consequently,
\begin{equation}\label{eq-2-5}
\int_{\R^n} U_0^{q_0+1} dx = \int_{\R^n} |\Delta U_0|^{p+1 \over p} dx = \int_{\R^n} V_0^{p+1} dx \le S^{{n \over 2} \cdot {p \over p+1}}.
\end{equation}
As a matter of fact, the inequality in \eqref{eq-2-5} must be the equality and so $V_{\ep} \to V_0$ in $L^{p+1}(\R^n)$ strongly; otherwise the Sobolev inequality \eqref{eq-Sob} would be violated.
Accordingly, if we choose $R > 0$ so large that
\[\int_{B^n(0,R)} V_0^{p+1} dx \ge (1-\eta) \int_{\R^n} V_0^{p+1} dx = (1-\eta) S^{{n \over 2} \cdot {p \over p+1}}\]
holds for a fixed small number $\eta > 0$, then
\begin{align*}
\int_{\Omega_{\ep}} V_{\ep o}^{p+1} dx &= \int_{\Omega_{\ep}} V_{\ep}^{p+1} dx - \int_{\Omega_{\ep}} V_{\ep i}^{p+1} dx
= \(S^{{n \over 2} \cdot {p \over p+1}} + o(1)\) - \(\int_{B^n(0,R)} V_0^{p+1} dx + o(1)\) \\
&\le \eta S^{{n \over 2} \cdot {p \over p+1}} + o(1)
\end{align*}
where $o(1) \to 0$ as $\ep \to 0$. Inserting this estimate in \eqref{eq-44}, we conclude the proof of the lemma.
\end{proof}

\begin{lem}
Suppose $p \in (\frac{2}{n-2}, \frac{n}{n-2})$ and pick numbers $\delta$ and $b$ satisfying
\begin{equation}\label{eq-2-4}
0 < \delta < \frac{n((n-2)p-2)}{(n-2)p} \quad \text{and} \quad \frac{n}{n-2} < b < \frac{np}{2}.
\end{equation}
Then there exists a constant $C > 0$ depending only on $n$, $p$ and $b$ such that
\begin{equation}\label{eq-2-0}
\left\| |\cdot|^{(n-2) - \frac{n+\delta}{b}} V_{\ep o} \right\|_{L^b(\Omega_{\ep})} \le C.
\end{equation}
\end{lem}
\begin{proof}
We infer from \eqref{eq-ext-sol} that
\begin{align*}
V_{\ep o} \le V_{\ep} &\le C (-\Delta_{\ep})^{-1} \left[ ((-\Delta_{\ep})^{-1} V_{\ep o}^p)^{q_{\ep}} + ((-\Delta_{\ep})^{-1} V_{\ep i}^p)^{q_{\ep}} \right] \\
&= C (-\Delta_{\ep})^{-1} \left[ F_{\ep} ((-\Delta_{\ep})^{-1} V_{\ep o}^p)^{1 \over p} + ((-\Delta_{\ep})^{-1} V_{\ep i}^p)^{q_{\ep}} \right].
\end{align*}
Hence \eqref{eq-1-1} and \eqref{eq-inv} imply
\begin{equation}\label{eq-2-1}
\begin{aligned}
\left\| |\cdot|^{-\alpha} V_{\ep o} \right\|_{L^b(\Omega_{\ep})}
&\le C \left\| |\cdot|^{-\alpha} \(|\cdot|^{-(n-2)} \ast \left[ F_{\ep} ((-\Delta_{\ep})^{-1} V_{\ep o}^p)^{1 \over p} \right]\) \right\|_{L^b(\Omega_{\ep})} + C \mca_{\ep} \\
&\le C \left\| |\cdot|^{-\alpha} F_{\ep} ((-\Delta_{\ep})^{-1} V_{\ep o}^p)^{1 \over p} \right\|_{L^{a_1}(\Omega_{\ep})} + C \mca_{\ep}
\end{aligned}
\end{equation}
where
\begin{equation}\label{eq-alpha}
\alpha = \frac{n+\delta}{b} - (n-2), \quad \frac{1}{a_1} - \frac{1}{b} = \frac{2}{n}
\end{equation}
and
\begin{equation}\label{eq-mca}
\mca_{\ep} := \left\| |\cdot|^{-\alpha} \( (|\cdot|^{-(n-2)} \ast \left[ (-\Delta_{\ep})^{-1} V_{\ep i}^p)^{q_{\ep}} \right] \) \right\|_{L^b(\Omega_{\ep})};
\end{equation}
we verify the necessary conditions to apply \eqref{eq-1-1} in Check 1 at the end of the proof.
As shown in Check 2 below, we can select $a_2 > 1$ such that
\[\frac{1}{a_1} = \frac{1}{a_2} + \frac{pq_{\ep}-1}{p(q_{\ep}+1)}.\]
Thus, employing H\"older's inequality and Lemma \ref{lem-Fep} to \eqref{eq-2-1}, we obtain
\begin{equation}\label{eq-2-2}
\begin{aligned}
\left\| |\cdot|^{-\alpha} V_{\ep o} \right\|_{L^b(\Omega_{\ep})}
&\le C \|F_{\ep}\|_{L^{p(q_{\ep}+1) \over pq_{\ep}-1}(\Omega_{\ep})} \left\| |\cdot|^{-\alpha} ((-\Delta_{\ep})^{-1} V_{\ep o}^p)^{1 \over p} \right\|_{L^{a_2}(\Omega_{\ep})} + C \mca_{\ep}\\
&\le \eta \left\| |\cdot|^{-\alpha} ((-\Delta_{\ep})^{-1} V_{\ep o}^p)^{1 \over p} \right\|_{L^{a_2}(\Omega_{\ep})} + C \mca_{\ep}
\end{aligned}
\end{equation}
where $\eta > 0$ can be chosen arbitrarily small. If we set the number $a_3$ by
\[\frac{1}{a_3} - \frac{p}{a_2} = \frac{2}{n},\]
Check 3 below ensures that $a_3 > 1$ and $0 < b - a_3p < C\ep$ for some constant $C > 0$. Also, \eqref{eq-1-1} leads to
\begin{equation}\label{eq-2-3}
\begin{aligned}
\left\| |\cdot|^{-\alpha} ((-\Delta_{\ep})^{-1} V_{\ep o}^p)^{1 \over p} \right\|_{L^{a_2}(\Omega_{\ep})}
&= \left\| |\cdot|^{-\alpha p} ((-\Delta_{\ep})^{-1} V_{\ep o}^p) \right\|^{1 \over p}_{L^{a_2 \over p}(\Omega_{\ep})}
\le C \left\| |\cdot|^{-\alpha p} V_{\ep o}^p \right\|^{1 \over p}_{L^{a_3}(\Omega_{\ep})} \\
&= C \left\| |\cdot|^{-\alpha} V_{\ep o} \right\|_{L^{a_3p}(\Omega_{\ep})}
\le C \left\| |\cdot|^{-\alpha} V_{\ep o} \right\|_{L^b(\Omega_{\ep})}.
\end{aligned}
\end{equation}
Therefore \eqref{eq-2-2} reads as
\[\left\| |\cdot|^{-\alpha} V_{\ep o} \right\|_{L^b(\Omega_{\ep})} \le \frac{1}{2} \left\| |\cdot|^{-\alpha} V_{\ep o} \right\|_{L^b(\Omega_{\ep})} + C\mca_{\ep}.\]
From the above inequality, the relation $\alpha b < n$ and the fact that $V_{\ep} \in L^{\infty}(\Omega_{\ep})$ which holds thanks to standard elliptic regularity theory, we conclude
\begin{equation}\label{eq-2-31}
\left\| |\cdot|^{-\alpha} V_{\ep o} \right\|_{L^b(\Omega_{\ep})} \le C \mca_{\ep}.
\end{equation}

On the other hand, by \eqref{eq-Vi}, there is a constant $C > 0$ independent of $\ep > 0$ such that
\[\(|\cdot|^{-(n-2)} \ast \left[ (-\Delta_{\ep})^{-1} V_{\ep i}^p \right]^{q_{\ep}}\) (x) \le \frac{C}{1+|x|^{n-2}} \quad \text{for every } x \in \Omega_{\ep}.\]
It then follows from \eqref{eq-mca} that
\begin{equation}\label{eq-2-32}
\sup_{\ep > 0} \mca_{\ep} \le C \left[ \int_{\R^n} \frac{1}{|x|^{(n+\delta)-(n-2)b}} \frac{1}{1+|x|^{(n-2)b}} dx \right]^{1 \over b} < \infty.
\end{equation}
Putting \eqref{eq-2-31} and \eqref{eq-2-32} together completes the proof.

\medskip \noindent
\textbf{Check 1.} We have to show
\[\frac{1}{b} - \frac{n-2}{n} < \frac{1}{n} \left[\frac{n+\delta}{b} - (n-2)\right] < \frac{1}{b}.\]
The first inequality holds for all $\delta > 0$. The second one is reduced to $b > \frac{\delta}{n-2}$, which is true whenever $\delta < 1$.

\medskip \noindent
\textbf{Check 2.} It suffices to check that
\[1 < a_1 < \frac{p(q_{\ep}+1)}{pq_{\ep}-1}\]
for small $\ep > 0$. The first inequality is valid because of $b > \frac{n}{n-2}$ and \eqref{eq-alpha}. The second one comes from
\[\frac{1}{a_1} - \frac{2}{n} = \frac{1}{b} > \frac{2}{np} = \frac{pq_0-1}{p(q_0+1)} - \frac{2}{n}.\]

\medskip \noindent
\textbf{Check 3.} We see
\[\frac{1}{a_3} = \frac{p}{a_2} + \frac{2}{n} = p \(\frac{1}{b} + \frac{2}{n}\) - \frac{pq_{\ep}-1}{q_{\ep}+1} + \frac{2}{n} = \frac{p}{b} + (p+1)\ep < 1\]
for small $\ep > 0$. Hence $a_3 > 1$ and $0 < b - a_3p < C\ep$ for some $C > 0$. Moreover,
\[\frac{p}{a_2} - \frac{n-2}{n} < \frac{\alpha p}{n} < \frac{p}{a_2} \quad \text{if} \quad \frac{p}{b} - 1 < \frac{\alpha p}{n} < \frac{p}{b} - \frac{2}{n}.\]
The latter inequalities hold true since
\[b > \frac{n}{n-2} > \frac{\delta p}{(n-2)p-2}.\]
Thus one can use \eqref{eq-1-1} to deduce \eqref{eq-2-3}.
\end{proof}

We next prove that \eqref{eq-2-0} holds for any $b > \frac{n}{n-2}$, relieving the restriction \eqref{eq-2-4} on $b$.
It is notable that the condition $p < 1$ is necessary in the proof.
\begin{lem}\label{lem-decay-2}
Suppose that $p \in (\frac{2}{n-2}, 1)$. For each fixed $b > \frac{n}{n-2}$,
there is a constant $C > 0$ and a small number $\delta > 0$ depending only on $n$, $p$ and $b$ such that \eqref{eq-2-0} is satisfied.
\end{lem}
\begin{proof}
Recall the function $F_{\ep}$ introduced in \eqref{eq-Fep}. We divide the proof into two steps.

\medskip
\noindent \textbf{Step 1}. We claim that for each $s \ge \frac{p(q_0+1)}{pq_0-1} > 1$, there exists a constant $C > 0$ independent of $\ep > 0$ such that
\begin{equation}\label{eq-Fep-2}
\|F_{\ep}\|_{L^s(\Omega_{\ep})} \le C.
\end{equation}

Given any large $\zeta > 1$, let $t > 0$ be the number such that
\[\frac{p}{t} = \frac{2}{n} + \frac{1}{\zeta}.\]
Then $t$ satisfies the second condition in \eqref{eq-2-4}. Hence, if we choose $\delta > 0$ small enough, we obtain from \eqref{eq-1-1} and \eqref{eq-2-0} that
\[\|(-\Delta_{\ep})^{-1} V_{\ep o}^p\|_{L^{\zeta}(\Omega_{\ep})} \le C \|V_{\ep o}^p\|_{L^{t \over p}(\Omega_{\ep})} = C \|V_{\ep o}\|_{L^t(\Omega_{\ep})}^p
\le C \left\| |\cdot|^{(n-2) - \frac{n+\delta}{t}} V_{\ep o} \right\|_{L^t(\Omega_{\ep})}^p \le C.\]
On the other hand, by H\"older's inequality and Lemmas \ref{lem-lam} and \ref{lem-Fep}, we have
\[\|F_{\ep}\|_{L^{p(q_0+1) \over pq_0-1}(\Omega_{\ep})} \le C.\] 
Interpolating the above two estimates, we conclude that \eqref{eq-Fep-2} holds for all $s \ge \frac{p(q_0+1)}{pq_0-1}$.

\medskip
\noindent \textbf{Step 2}. By using \eqref{eq-Fep-2}, we shall prove the lemma.

Fix any large $b > \frac{n}{n-2}$, and let $a_1$, $\alpha$ and $\mca_{\ep}$ be the numbers defined by \eqref{eq-alpha} and \eqref{eq-mca}.
Choosing $a_4 > 1$ so large that \eqref{eq-a4} holds, we set $s$ as the number satisfying
\[\frac{1}{a_4} = \frac{1}{a_1} - \frac{1}{s} \quad \text{and} \quad s \ge \frac{p(q_0+1)}{pq_0-1} = \frac{np}{2(p+1)}.\]
From \eqref{eq-2-1} and H\"older's inequality, we reach
\begin{align*}
\left\| |\cdot|^{-\alpha} V_{\ep o} \right\|_{L^b(\Omega_{\ep})}
&\le C \left\| |\cdot|^{-\alpha} F_{\ep} ((-\Delta_{\ep})^{-1} V_{\ep o}^p)^{1 \over p} \right\|_{L^{a_1}(\Omega_{\ep})} + C \mca_{\ep} \\
&\le C \|F_{\ep}\|_{L^s(\Omega_{\ep})} \left\| |\cdot|^{-\alpha} ((-\Delta_{\ep})^{-1} V_{\ep o}^p)^{1 \over p} \right\|_{L^{a_4}(\Omega_{\ep})} + C \mca_{\ep}.
\end{align*}
Owing to Check 4 below, if $a_5$ is a number satisfying
\[\frac{1}{a_5} - \frac{p}{a_4} = \frac{2}{n},\]
then $a_5p$ satisfies the second condition in \eqref{eq-2-4}, and so one can argue as in \eqref{eq-2-3} to deduce
\begin{align*}
\left\| |\cdot|^{-\alpha} ((-\Delta_{\ep})^{-1} V_{\ep o}^p)^{1 \over p} \right\|_{L^{a_4}(\Omega_{\ep})} \le C \left\| |\cdot|^{-\alpha} V_{\ep o} \right\|_{L^{a_5p}(\Omega_{\ep})}.
\end{align*}
Therefore we see from \eqref{eq-2-0} and \eqref{eq-Fep-2} that
\[\left\| |\cdot|^{-\alpha} V_{\ep o} \right\|_{L^b(\Omega_{\ep})} \le C \|F_{\ep}\|_{L^s(\Omega_{\ep})} \left\| |\cdot|^{-\alpha} V_{\ep o} \right\|_{L^{a_5p}(\Omega_{\ep})} + C \mca_{\ep} \le C,\]
which is the desired result.

\medskip \noindent
\textbf{Check 4.} We want to check that
\begin{equation}\label{eq-a4}
\frac{p}{a_4} + \frac{2}{n} < 1 \quad \text{and} \quad \frac{p}{a_4} - \frac{n-2}{n} < \frac{p}{n} \left[\frac{n+\delta}{b} - (n-2)\right] < \frac{p}{a_4}
\end{equation}
for sufficiently large $a_4 > 1$. It follows from the inequalities
\[\frac{2}{n} < 1 \quad \text{and} \quad - \frac{n-2}{n} < \frac{p}{n} \left[\frac{n+\delta}{b} - (n-2)\right] < 0\]
which hold for every $n \ge 3$ and $p < 1$.
\end{proof}

The above estimate allows us to conclude the proof of Proposition \ref{prop-dec}.
\begin{proof}[Completion of the proof of Proposition \ref{prop-dec}]
Given large $b > 1$ and small $\delta > 0$, it holds that
\begin{equation}\label{eq-4-1}
U_{\ep}(x) = (-\Delta_{\ep})^{-1}(V_{\ep}^p)(x) \le C \int_{\Omega_{\ep}}
\frac{1}{|x-y|^{n-2}} \left[ |y|^{(n-2)-{n+\delta \over b}} V_{\ep}(y) \right]^p |y|^{\left[{n+\delta \over b}-(n-2)\right] p} dy
\end{equation}
for $x \in \Omega_{\ep}$. For any fixed point $x \in \Omega_{\ep}$ such that $|x| \ge 1$, we set
\begin{equation}\label{eq-D1D2}
D_1 = \left\{y \in \Omega_{\ep}: |y| \le \frac{|x|}{2} \right\}, \quad D_2 = \left\{y \in \Omega_{\ep}: |y-x| \le \frac{|x|}{2} \right\}
\end{equation}
and
\begin{equation}\label{eq-D3}
D_3 = \left\{y \in \Omega_{\ep}: |y| > \frac{|x|}{2} \text{ and } |y-x| > \frac{|x|}{2} \right\}.
\end{equation}
Then we divide the integral in the right-hand side of \eqref{eq-4-1} as
\[A_1 + A_2 + A_3 := \(\int_{D_1} + \int_{D_2} + \int_{D_3}\) \frac{1}{|x-y|^{n-2}} \left[ |y|^{(n-2)-{n+\delta \over b}} V_{\ep}(y) \right]^p |y|^{\left[{n+\delta \over b}-(n-2)\right] p} dy;\]
namely, the domain of integration of $A_j$ is $D_j$ for $j = 1,\, 2,\, 3$.
We will estimate each of them.

We note that
\begin{equation}\label{eq-4-7}
\begin{cases}
|x-y| \ge |x|-|y| \ge \dfrac{|x|}{2} &\text{if } y \in D_1, \\
|y| \ge |x|-|y-x| \ge \dfrac{|x|}{2} &\text{if } y \in D_2, \\
|x-y| > \dfrac{|y-x|}{2} + \dfrac{|x|}{4} > \dfrac{|y|}{4} &\text{if } y \in D_3.
\end{cases}
\end{equation}
Throughout the proof, we use $\kappa(b)$ to denote a function of $b \in (1,\infty)$ such that $\kappa(b) \to 0$ as $b \to \infty$, which may vary from line to line.
Then, by employing \eqref{eq-4-7}, H\"older's inequality and Lemma \ref{lem-decay-2}, we discover
\begin{align*}
A_1 &\le \frac{C}{|x|^{n-2}} \int_{D_1} \left[ |y|^{(n-2)-{n+\delta \over b}} V_{\ep}(y) \right]^p
|y|^{\left[{n+\delta \over b}-(n-2)\right] p} dy \\
&\le \frac{C}{|x|^{n-2}} \left\| |\cdot|^{(n-2) - \frac{n+\delta}{b}} V_{\ep} \right\|_{L^b(\Omega_{\ep})}^p
\( \int_{0}^{|x| \over 2} r^{[n+\delta-(n-2)b] {p \over b-p}+(n-1)} dr\)^{\frac{b-p}{b}} \\
&\le \frac{C}{|x|^{n-2}} |x|^{n-(n-2)p + \kappa(b)} = \frac{C}{|x|^{(n-2)p -2+\kappa(b)}}.
\end{align*}
Using the second inequality of \eqref{eq-4-7}, we compute $A_2$ and $A_3$ as
\begin{align*}
A_2 &\le C |x|^{\left[ \frac{n+\delta}{b} - (n-2)\right] p} \int_{D_2} \frac{1}{|x-y|^{n-2}} \left[ |y|^{(n-2)-{n+\delta \over b}} V_{\ep}(y) \right]^p dy \\
&\le C |x|^{\left[\frac{n+\delta}{b} - (n-2)\right]p} \left\||\cdot|^{(n-2) - \frac{n+\delta}{b}} V_{\ep}\right\|_{L^s(\Omega_{\ep})}^p
\( \int_{0}^{|x| \over 2} \frac{r^{n-1}}{r^{(n-2)b \over b-p}} dr \)^{\frac{b-p}{b}} \\
&\le C |x|^{-(n-2)p +\kappa(b)} |x|^{2 + \kappa(b)} \le \frac{C}{|x|^{(n-2)p-2 +\kappa(b)}}
\end{align*}
and
\begin{align*}
A_3 &\le C \int_{D_3}  \frac{1}{|y|^{(n-2)(p+1) +\kappa(b)}} \left[ |y|^{(n-2)-{n+\delta \over b}} V_{\ep}(y) \right]^p dy
\\
&\le C \left\| |\cdot|^{(n-2) - \frac{n+\delta}{b}} V_{\ep} \right\|_{L^b(\Omega_{\ep})}^p \( \int_{|x| \over 2}^{\infty} \frac{r^{n-1}}{r^{(n-2)(p+1)+\kappa(b)}} dr \)^{\frac{b-p}{b}}
\le \frac{C}{|x|^{(n-2)p -2 +\kappa(b)}}.
\end{align*}
Therefore
\begin{equation}\label{eq-4-4}
U_{\ep}(x) \le \frac{C}{|x|^{(n-2)p-2 +\kappa(b)}} \quad \text{for all } x \in \Omega_{\ep} \text{ such that } |x| \ge 1.
\end{equation}
Notice that \eqref{eq-4-4} is almost same as the desired one in \eqref{eq-4-3}, but it contains a small remainder $\kappa(b)$ that should be removed.
To do it, we will first obtain the sharp decay of $V_{\ep}$ by putting \eqref{eq-4-4} into \eqref{eq-ext-sol}.
Then we will be able to derive the desired sharp decay of $U_{\ep}$.

Indeed, \eqref{eq-4-4} and \eqref{eq-ext-sol} give
\[V_{\ep}(x) \le C \int_{\R^n} \frac{1}{|x-y|^{n-2}} \frac{1}{1+|y|^{2p + (n+2) + \kappa(b) + o(1)}} dy\]
where $o(1) \to 0$ as $\ep \to 0$.
We divide the integral of the right-hand side as
\[B_1 + B_2 + B_3 := \(\int_{D_1} + \int_{D_2} + \int_{D_3}\) \frac{1}{|x-y|^{n-2}} \frac{1}{1+|y|^{2p + (n+2) + \kappa(b) + o(1)}} dy\]
where the domains $D_1$, $D_2$ and $D_3$ are set in \eqref{eq-D1D2} and \eqref{eq-D3}. We have
\[B_1 \le \frac{C}{|x|^{n-2}} \int_{0}^{\frac{|x|}{2}} \frac{r^{n-1}}{1+r^{2p +(n+2) +\kappa(b)}} dr \le \frac{C}{|x|^{n-2}}.\]
Furthermore,
\[B_2 \le \frac{C |x|^2}{|x|^{2p+(n+2)+\kappa(b)+o(1)}} = \frac{C}{|x|^{n+2p +\kappa(b)+o(1)}}
\quad \text{and} \quad
B_3 \le \frac{C}{|x|^{n+2p+\kappa(b)+o(1)}}.\]
Gathering the above estimates together, we find
\[V_{\ep}(x) \le \frac{C}{|x|^{n-2}}\quad \text{for all } |x| \ge 1.\]
Finally, we insert this into \eqref{eq-ext-sol} to get
\[U_{\ep} (x) \le C \int_{\Omega_{\ep}} \frac{1}{|x-y|^{n-2}} \frac{1}{1+|y|^{(n-2)p}} dy \le \frac{C}{|x|^{(n-2)p-2}} \quad \text{for all } |x| \ge 1.\]
The proof is finished.
\end{proof}

\section{Estimates for $v_{\ep}$ near the blow-up point}\label{sec-vep}
Let $\{(u_{\ep},v_{\ep})\}_{\ep >0}$ be a family of least energy solutions to \eqref{eq-main} with $q = q_{\ep}$ and $x_{\ep}$ the blow-up point given in \eqref{eq-lambda}.
Assume that $d_{\ep} = \frac{1}{4} \text{dist}(x_{\ep},\pa\Omega) \to 0$.
In this section, we derive sharp estimates for the functions $\{v_{\ep}\}_{\ep > 0}$ and their first-order derivatives on the sphere $\pa B^n(x_{\ep}, 2d_{\ep})$.
With a local Pohozaev-type identity
\begin{multline}\label{eq-poho-0}
-\int_{\pa B^n(x_{\ep}, 2d_{\ep})} \( \frac{\pa u_{\ep}}{\pa \nu} \frac{\pa v_{\ep}}{\pa x_j} + \frac{\pa v_{\ep}}{\pa \nu} \frac{\pa u_{\ep}}{\pa x_j} \) dS_x
+ \int_{\pa B^n(x_{\ep}, 2d_{\ep})} (\nabla u_{\ep} \cdot \nabla v_{\ep}) \nu_j\, dS_x \\
= \frac{1}{p+1} \int_{\pa B^n(x_{\ep}, 2d_{\ep})} v_{\ep}^{p+1} \nu_j\, dS_x + \frac{1}{q_{\ep}+1} \int_{\pa B^n(x_{\ep}, 2d_{\ep})} u_{\ep}^{q_{\ep}+1} \nu_j\, dS_x
\end{multline}
for $1 \le j \le n$ (see Lemma \ref{lem-poho} for its derivation), they will consist of essential tools in the proof of Theorems \ref{thm-1} and \ref{thm-2}.
As expected, our a priori assumption that $d_{\ep} \to 0$ as $\ep \to 0$ makes the analysis delicate.

Here and after, we use the following constants
\begin{equation}\label{eq-AUV}
A_{V_0} := \int_{\R^n} V_0^p(y)\, dy \quad \text{and} \quad A_{U_0} := \int_{\R^n} U_0^{q_0}(y)\, dy.
\end{equation}
Define also
\begin{equation}\label{eq-ab}
\alpha_0 = \frac{2(p+1)}{pq_0-1} \quad \text{and} \quad \beta_0 = \frac{2(q_0+1)}{pq_0-1}.
\end{equation}
Then one may check from \eqref{eq-cr-hy} that
\begin{equation}\label{eq-51}
\alpha_0 (q_0 +1) = n \quad \text{and } \beta_0 (p+1) = n.
\end{equation}

\begin{lem}\label{lem-g10}
Suppose that $p \in (\frac{2}{n-2}, \frac{n+2}{n-2}]$. For each point $x \in \pa B^n(x_{\ep}, 2d_{\ep})$, we have
\begin{equation}\label{eq-g1-0}
v_{\ep}(x) = \lambda_{\ep}^{-\alpha_0} A_{U_0} G(x,x_{\ep}) + o(d_{\ep}^{-(n-2)} \lambda_{\ep}^{-\alpha_0})
\end{equation}
and
\begin{equation}\label{eq-g1-8}
\nabla v_{\ep}(x) = \lambda_{\ep}^{-\alpha_0} A_{U_0} \nabla_x G(x,x_{\ep}) + o(d_{\ep}^{-(n-1)} \lambda_{\ep}^{-\alpha_0}).
\end{equation}
Here $o$ notation is understood as
\begin{equation}\label{eq-g1-18}
\lim_{\ep \to 0} \sup_{x \in \pa B^n(x_{\ep}, 2d_{\ep})} d_{\ep}^k \lambda_{\ep}^{\alpha_0} \cdot \left|o(d_{\ep}^{-k}\lambda_{\ep}^{-\alpha_0})\right| = 0
\quad \text{for } k = n-1 \text{ or } n-2.
\end{equation}
\end{lem}
\begin{proof}
We first derive \eqref{eq-g1-0}. By Green's representation formula, we have
\begin{equation}\label{eq-g1-1}
v_{\ep}(x) = G(x,x_{\ep}) \( \int_{\Omega}u_{\ep}^{q_{\ep}} (y) dy\) + \int_{\Omega} [G(x,y) - G(x,x_{\ep})] u_{\ep}^{q_{\ep}} (y) dy.
\end{equation}

Owing to Lemma \ref{lem-lam} and Proposition \ref{prop-dec}, we can apply the dominated convergence theorem to yield
\[\lim_{\ep \to 0} \lambda_{\ep}^{\alpha_0} \int_{\Omega} u_{\ep}^{q_{\ep}}(y) dy
= \lim_{\ep \to 0} \int_{\Omega_{\ep}} U_{\ep}^{q_{\ep}}(y) dy = A_{U_0}.\]
Therefore, \eqref{eq-g1-9} implies that
\[G(x,x_{\ep}) \(\int_{\Omega} u_{\ep}^{q_{\ep}}(y) dy\) = \lambda_{\ep}^{-\alpha_0} A_{U_0} G(x,x_{\ep}) + o(d_{\ep}^{-(n-2)} \lambda_{\ep}^{-\alpha_0}).\]

To estimate the second integral in the right-hand side of \eqref{eq-g1-1}, we split it into three parts as follows:
\begin{align*}
&\ \int_{\Omega} [G(x,y) - G(x,x_{\ep})] u_{\ep}^{q_{\ep}} (y) dy \\
&= \(\int_{B^n(x_{\ep}, d_{\ep})} + \int_{B^n(x_{\ep},4d_{\ep}) \setminus B^n(x_{\ep},d_{\ep})} + \int_{\Omega \setminus B^n(x_{\ep}, 4d_{\ep})} \) [G(x,y) - G(x,x_{\ep})] u_{\ep}^{q_{\ep}} (y) dy \\
&=: I_1 (x) + I_2 (x) + I_3 (x).
\end{align*}
We assert that
\begin{equation}\label{eq-g1-16}
|I_1 (x)| + |I_2 (x)| + |I_3 (x)| = o(d_{\ep}^{-(n-2)} \lambda_{\ep}^{-\alpha_0}),
\end{equation}
which will lead the validity of \eqref{eq-g1-0}.

\medskip \noindent \textbf{Estimate of $I_1$.}
Because $|x-x_{\ep}|=2d_{\ep}$, it holds that $|x-y| \ge d_{\ep}$ for all $y \in B^n(x_{\ep}, d_{\ep})$. Thus $|\nabla_y G(x,y)| \le Cd_{\ep}^{-(n-1)}$ and so
\begin{equation}\label{eq-g1-22}
|G(x,y) - G(x,x_{\ep})| \le C d_{\ep}^{-(n-1)} |y-x_{\ep}|
\end{equation}
for every $y \in B^n(x_{\ep}, d_{\ep})$. By using Lemma \ref{lem-lam}, Proposition \ref{prop-dec}, \eqref{eq-51} and \eqref{eq-g1-22}, we estimate $I_1$ as
\begin{equation}\label{eq-g1-2}
\begin{aligned}
|I_1 (x)| &\le C d_{\ep}^{-(n-1)} \lambda_{\ep}^{\alpha_{\ep} q_{\ep}} \int_{B^n(x_{\ep}, d_{\ep})} |y-x_{\ep}|\, U_{\ep}^{q_{\ep}}(\lambda_{\ep} (y-x_{\ep})) dy
\\
& \le Cd_{\ep}^{-(n-2)} \lambda_{\ep}^{-\alpha_0} \Lambda_{\ep}^{-1} \int_{B^n(0, \Lambda_{\ep})} |y| U_{\ep}^{q_\ep} (y) dy = o(d_{\ep}^{-(n-2)} \lambda_{\ep}^{-\alpha_0})
\end{aligned}
\end{equation}
where $\Lambda_{\ep} = d_{\ep} \lambda_{\ep} \to \infty$ as $\ep \to 0$ as shown in Lemma \ref{lem-lam}. Indeed, the last equality holds since
\begin{equation}\label{eq-g1-17}
\begin{aligned}
\int_{B^n(0, \Lambda_{\ep})} |y| U_{\ep}^{q_{\ep}} (y) dy &\le \begin{cases}
C &\text{if } p \in (\frac{n+1}{n-2}, \frac{n+2}{n-2}], \\
C\log \Lambda_{\ep} &\text{if } p = \frac{n+1}{n-2}, \\
C\Lambda_{\ep}^{-(n-2)q_{\ep} +n+1} &\text{if } p \in (\frac{n}{n-2}, \frac{n+1}{n-2}), \\
C\Lambda_{\ep}^{-(n-2)q_{\ep} +n+1} \log \Lambda_{\ep} &\text{if } p = \frac{n}{n-2}, \\
C\Lambda_{\ep}^{-((n-2)p-2)q_{\ep} +n+1} &\text{if } p \in (\frac{2}{n-2}, \frac{n}{n-2})
\end{cases} \\
&= o(\Lambda_{\ep}).
\end{aligned}
\end{equation}

\noindent \textbf{Estimate of $I_2$.}
We infer from again Lemma \ref{lem-lam} and Proposition \ref{prop-dec} that
\begin{align*}
u_{\ep}(y) \le C \lambda_{\ep}^{\alpha_0} U_{\ep}(\lambda_{\ep}(y-x_{\ep})) &\le \begin{cases}
C \lambda_{\ep}^{\alpha_0} \Lambda_{\ep}^{-(n-2)} &\text{if } p \in (\frac{n}{n-2}, \frac{n+2}{n-2}],\\
C \lambda_{\ep}^{\alpha_0} \Lambda_{\ep}^{-(n-2)} \log \Lambda_{\ep} &\text{if } p = \frac{n}{n-2}, \\
C \lambda_{\ep}^{\alpha_0} \Lambda_{\ep}^{-((n-2)p-2)} &\text{if } p \in (\frac{2}{n-2}, \frac{n}{n-2})
\end{cases} \\
&= o\big(\lambda_{\ep}^{\alpha_0} \Lambda_{\ep}^{-{n \over q_0}}\big)
\end{align*}
for $y \in B^n(x_{\ep}, 4d_{\ep}) \setminus B^n(x_{\ep}, d_{\ep})$.
Moreover, since
\[|x-y| \le 6d_{\ep} \quad \text{for } x \in \pa B^n(x_{\ep}, 2d_{\ep}) \text{ and } y \in B^n(x_{\ep}, 4d_{\ep}),\]
we have that
\[|G(x,y) - G(x,x_{\ep})| \le |G(x,y)| + |G(x,x_{\ep})| \le \frac{C}{|x-y|^{n-2}} + Cd_{\ep}^{-(n-2)} \le \frac{C}{|x-y|^{n-2}}.\]
As a consequence, we obtain
\begin{equation}\label{eq-g1-3}
\begin{aligned}
|I_2 (x)| &= o\( \lambda_{\ep}^{\alpha_0q_0} \Lambda_{\ep}^{-n} \int_{B^n(x_{\ep},4d_{\ep})\setminus B^n(x_{\ep}, d_{\ep})} \frac{1}{|x-y|^{n-2}} dy \)
= o\(\lambda_{\ep}^{\alpha_0q_0} \Lambda_{\ep}^{-n} d_{\ep}^2\) \\
&= o(d_{\ep}^{-(n-2)} \lambda_{\ep}^{-\alpha_0}).
\end{aligned}
\end{equation}

\noindent \textbf{Estimate of $I_3$.}
There holds that
\[|G(x,y) - G(x,x_{\ep})| \le C d_{\ep}^{-(n-2)} \quad \text{for } x \in \pa B^n(x_{\ep}, 2d_{\ep}) \text{ and } y \in \Omega \setminus B^n(x_{\ep}, 4d_{\ep}).\]
Thus,
\begin{equation}\label{eq-g1-4}
\begin{aligned}
|I_3 (x)| &\le Cd_{\ep}^{-(n-2)} \int_{\Omega \setminus B^n(x_{\ep}, 4d_{\ep})} u_{\ep}^{q_{\ep}}(y) dy
\le C d_{\ep}^{-(n-2)} \lambda_{\ep}^{-\alpha_0} \int_{\R^n \setminus B^n(0,4\Lambda_{\ep})} U_{\ep}^{q_{\ep}}(y) dy \\
&= o(d_{\ep}^{-(n-2)} \lambda_{\ep}^{-\alpha_0})
\end{aligned}
\end{equation}
where the last equality can be justified as in \eqref{eq-g1-17}.

Collecting \eqref{eq-g1-2}, \eqref{eq-g1-3} and \eqref{eq-g1-4}, we obtain \eqref{eq-g1-16}.

\medskip
Similarly, one can deduce the gradient estimate \eqref{eq-g1-8}. In this time, we employ the estimates
\[|\nabla_x G(x,y) - \nabla_x G(x,x_{\ep})| \le \begin{cases}
C d_{\ep}^{-n} |y-x_{\ep}| &\text{for } y \in B^n(x_{\ep}, d_{\ep}),\\
C |x-y|^{-(n-1)} &\text{for } y \in B^n(x_{\ep}, 4d_{\ep}),\\
C d_{\ep}^{-(n-1)} &\text{for } y \in \Omega \setminus B^n(x_{\ep}, 4d_{\ep}),
\end{cases}\]
which is valid whenever $x \in \pa B^n(x_{\ep}, 2d_{\ep})$. Consequently, the lemma is proved.
\end{proof}

\section{Proof of Theorem \ref{thm-1} (Case 1: $p \in (\frac{n}{n-2}, \frac{n+2}{n-2}]$)}\label{sec-thm-11}
This section is devoted to the proof of Theorem \ref{thm-1} under the assumption that $p \in (\frac{n}{n-2}, \frac{n+2}{n-2}]$.
As already explained in the introduction, we suppose that the maximum point $x_{\ep}$ tends to $\pa \Omega$,
and then derive a contradiction from the Pohozaev-type identity \eqref{eq-poho-0} on the sphere $\pa B^n(x_{\ep}, 2d_{\ep})$.

\medskip
We first need estimates for $u_{\ep}$ near the blow-up point.
\begin{lem}
Suppose that $p \in (\frac{n}{n-2}, \frac{n+2}{n-2}]$, $\{(u_{\ep},v_{\ep})\}_{\ep > 0}$ is a family of least energy solutions to \eqref{eq-main} with $q = q_{\ep}$
and $d_{\ep} =  \frac{1}{4} \textnormal{dist}(x_{\ep}, \pa \Omega) \to 0$ as $\ep \to 0$.
Then, for each $x \in \pa B^n(x_{\ep}, 2d_{\ep})$, we have
\begin{equation}\label{eq-g1-20}
u_{\ep}(x) = \lambda_{\ep}^{-\beta_0} A_{V_0} G(x,x_{\ep}) + o(d_{\ep}^{-(n-2)} \lambda_{\ep}^{-\beta_0})
\end{equation}
and
\begin{equation}\label{eq-g1-7}
\nabla u_{\ep}(x) = \lambda_{\ep}^{-\beta_0} A_{V_0} \nabla_x G(x,x_{\ep}) + o(d_{\ep}^{-(n-1)} \lambda_{\ep}^{-\beta_0}).
\end{equation}
Here, the definition of the numbers $A_{V_0}$ and $\beta_0$ can be found in \eqref{eq-AUV} and \eqref{eq-ab},
and $o$ notation is uniform with respect to $x \in \pa B^n(x_{\ep}, 2d_{\ep})$ in the sense that \eqref{eq-g1-18} holds.
\end{lem}
\begin{proof}
In this case, it holds by Proposition \ref{prop-dec} that
\[\int_{B^n(0,\Lambda_{\ep})} |y| V_{\ep}^p(y) dy = o(\Lambda_{\ep});\]
compare with \eqref{eq-g1-17}.
Therefore, arguing as in the proof of Lemma \ref{lem-g10}, we establish \eqref{eq-g1-20} and \eqref{eq-g1-7}.
\end{proof}

We are now ready to prove Theorem \ref{thm-1} for the case that $p \in (\frac{n}{n-2}, \frac{n+2}{n-2}]$.
\begin{proof}[Proof of Theorem \ref{thm-1} (Case 1)]
To the contrary, we assume that $d_{\ep} = \frac{1}{4} \text{dist}(x_{\ep}, \pa \Omega) \to 0$ as $\ep \to 0$ passing to a subsequence.
For the sake of brevity, we keep using $\ep$ as the parameter for the subsequence instead of introducing new notation.

For each $1 \le j \le n$, let $L_j^{\ep}$ and $R_j^{\ep}$ be the left-hand and right-hand sides of the local Pohozaev-type identity \eqref{eq-poho-0}, respectively, so that $L_j^{\ep} = R_j^{\ep}$.
In the following, we shall estimate values of both $L_j^{\ep}$ and $R_j^{\ep}$, which will allow us to reach a contradiction.

\medskip \noindent \textbf{Estimate of $L_j^{\ep}$.}
An application of \eqref{eq-g1-7}, \eqref{eq-g1-8} and \eqref{eq-g1-9} shows
\begin{equation}\label{eq-g1-14}
\begin{aligned}
L_j^{\ep} &=-\lambda_{\ep}^{-(n-2)} A_{U_0} A_{V_0} \int_{\pa B^n(x_{\ep}, 2d_{\ep})}
\( \frac{\pa G}{\pa \nu} (x,x_{\ep}) \frac{\pa G}{\pa x_j} (x,x_{\ep}) + \frac{\pa G}{\pa \nu} (x,x_{\ep}) \frac{\pa G}{\pa x_j} (x,x_{\ep}) \) dS_x
\\
&\ + \lambda_{\ep}^{-(n-2)} A_{U_0} A_{V_0} \int_{\pa B^n(x_{\ep}, 2d_{\ep})} |\nabla_x G(x,x_{\ep})|^2 \nu_j\, dS_x \\
&\ + o \(|\pa B^n(x_{\ep}, 2d_{\ep})| d_{\ep}^{-2(n-1)} \lambda_{\ep}^{-(n-2)} \)
\\
&=-\lambda_{\ep}^{-(n-2)} A_{U_0} A_{V_0} \mci_1(2d_{\ep}) + o(d_{\ep}^{-(n-1)} \lambda_{\ep}^{-(n-2)})
\end{aligned}
\end{equation}
where
\begin{equation}\label{eq-g1-21}
\mci_1(r) := \int_{\pa B^n(x_{\ep},r)}\( 2 \frac{\pa G}{\pa \nu} (x,x_{\ep}) \frac{\pa G}{\pa x_j} (x,x_{\ep}) - |\nabla_x G(x,x_{\ep})|^2 \nu_j \) dS_x \quad \text{for } r > 0 \text{ small}.
\end{equation}
To compute the value of $\mci_1(2d_{\ep})$, we first observe that the value of $\mci_1(r)$ is independent of $r > 0$.
Indeed, testing $\frac{\pa G}{\pa x_j}(\cdot,x_{\ep})$ in
\[-\Delta_x G(\cdot, x_{\ep}) = 0 \quad \text{in } A_r := B^n(x_{\ep}, 2d_{\ep}) \setminus B^n(x_{\ep},r)\]
with $r \in (0, 2d_{\ep})$, we find
\begin{equation}\label{eq-g1-55}
\begin{aligned}
0 & = -\int_{\pa A_r}\frac{\pa G}{\pa \nu}(x,x_{\ep}) \frac{\pa G}{\pa x_j}(x,x_{\ep})\, dS_x + \int_{A_r} \nabla_x G(x,x_{\ep}) \cdot \nabla_x \frac{\pa G(x,x_{\ep})}{\pa x_j}\, dx
\\
& = -\int_{\pa A_r}\frac{\pa G}{\pa \nu}(x,x_{\ep}) \frac{\pa G}{\pa x_j}(x,x_{\ep})\, dS_x + \frac{1}{2} \int_{\pa A_r} |\nabla_x G(x,x_{\ep})|^2 \nu_j\, dS_x,
\end{aligned}
\end{equation}
which implies that $\mci_1(r)$ is constant on $(0,2d_{\ep}]$. By using this fact, we compute
\begin{align*}
\mci_1(2d_{\ep}) &=\lim_{r \to 0} \mci_1(r) \\
& = \lim_{r \to 0} \int_{\pa B^n(x_{\ep},r)} \left[ 2 \( \frac{c_n(n-2)}{|x-x_{\ep}|^{n-1}} + \frac{\pa H}{\pa \nu}(x,x_{\ep})\)
\(\frac{c_n (n-2)(x-x_{\ep})_j}{|x-x_{\ep}|^n} + \frac{\pa H}{\pa x_j}(x,x_{\ep})\) \right. \\
&\hspace{75pt} \left. - \left|\frac{c_n(n-2)(x-x_{\ep})}{|x-x_{\ep}|^n} + \nabla_x H(x,x_{\ep}) \right|^2 \nu_j \right] dS_x.
\end{align*}
On account of the oddness of the integrand, we have that
\[\int_{\pa B^n(x_{\ep},r)} \left[ 2 \( \frac{c_n(n-2)}{|x-x_{\ep}|^n} \)
\( \frac{c_n (n-2)(x-x_{\ep})_j}{|x-x_{\ep}|^n}\) - \left|  \frac{c_n(n-2)(x-x_{\ep})}{|x-x_{\ep}|^n} \right|^2 \nu_j \right]dS_x =0.\]
Furthermore, because $-\Delta_x H(\cdot, x_{\ep}) = 0$ holds in $B^n(x_{\ep}, 2d_{\ep})$, we may proceed as in \eqref{eq-g1-55} to obtain
\[\int_{\pa B^n(x_{\ep},r)} \left[ 2 \frac{\pa H}{\pa \nu}(x,x_{\ep}) \frac{\pa H}{\pa x_j}(x,x_{\ep}) - |\nabla_x H(x,x_{\ep})|^2 \nu_j \right] dS_x =0.\]
By considering the above equalities, we calculate
\begin{align*}
\mci_1(2d_{\ep}) &= 2c_n(n-2) \lim_{r \to 0} \int_{\pa B^n(x_{\ep},r)} \left[ \frac{\pa H}{\pa \nu}(x,x_{\ep}) \frac{(x-x_{\ep})_j}{|x-x_{\ep}|^n}
+ \frac{1}{|x-x_{\ep}|^{n-1}} \frac{\pa H}{\pa x_j}(x,x_{\ep}) \right. \\
&\hspace{130pt} \left. - \frac{(x-x_{\ep})}{|x-x_{\ep}|^n} \cdot \nabla_x H (x,x_{\ep})\nu_j \right] dS_x \\
& = 2c_n(n-2) |\S^{n-1}| \left. \( \frac{1}{n} \frac{\pa H}{\pa x_j}(x, x_{\ep})
+ \frac{\pa H}{\pa x_j}(x, x_{\ep}) - \frac{1}{n} \frac{\pa H}{\pa x_j}(x, x_{\ep}) \) \right|_{x = x_{\ep}}\\
& = 2c_n (n-2)|\S^{n-1}| \left. \frac{\pa H}{\pa x_j}(x, x_{\ep}) \right|_{x = x_{\ep}}.
\end{align*}
Inserting this into \eqref{eq-g1-14}, we get
\begin{equation}\label{eq-g1-45}
L_j^{\ep} = - 2c_n(n-2) A_{U_0} A_{V_0} |\S^{n-1}| \lambda_{\ep}^{-(n-2)}
\left. \frac{\pa H}{\pa x_j}(x, x_{\ep}) \right|_{x = x_{\ep}} + o(d_{\ep}^{-(n-1)} \lambda_{\ep}^{-(n-2)}).
\end{equation}
Now, if we denote by $\nu_{x_{\ep}} = (a_1, \cdots, a_n) \in \S^{n-1}$ the unique unit vector such that $x_{\ep} + \mfd(x_{\ep}) \nu_{x_{\ep}} \in \pa \Omega$, then we deduce with \eqref{eq-hn} that
\begin{equation}\label{eq-g1-5}
\begin{aligned}
- \sum_{j=1}^n a_j L_j^{\ep} & = 2c_n(n-2) A_{U_0} A_{V_0} |\S^{n-1}| \lambda_{\ep}^{-(n-2)}
\left. \frac{\pa H}{\pa \nu_{x_\ep}}(x, x_{\ep}) \right|_{x = x_{\ep}} + o(d_{\ep}^{-(n-1)} \lambda_{\ep}^{-(n-2)}) \\
& \ge C d_{\ep}^{-(n-1)} \lambda_{\ep}^{-(n-2)} = C\lambda_{\ep} \Lambda_{\ep}^{-(n-1)}
\end{aligned}
\end{equation}
for some $C > 0$.

\medskip \noindent \textbf{Estimate of $R_j^{\ep}$.}
From \eqref{eq-4-3}, we see that
\[v_{\ep} (x) \le C \lambda_{\ep}^{\beta_0} V_{\ep} (\lambda_{\ep}(x-x_{\ep})) \le C \lambda_{\ep}^{\beta_0} \Lambda_{\ep}^{-(n-2)} \quad \text{for all } x \in \pa B^n(x_{\ep}, 2d_{\ep}).\]
Hence, by \eqref{eq-51},
\begin{equation}\label{eq-g1-11}
\left| \int_{\pa B^n(x_{\ep}, 2d_{\ep})} v_{\ep}^{p+1} \nu_j\, dS_x \right| \le C|\pa B^n(x_{\ep}, 2d_{\ep})| \lambda_{\ep}^n \Lambda_{\ep}^{-(n-2)(p+1)} = C\lambda_{\ep} \Lambda_{\ep}^{(n-1)-(n-2)(p+1)}.
\end{equation}
Similarly, it holds that
\begin{equation}\label{eq-g1-12}
\left| \int_{\pa B^n(x_{\ep}, 2d_{\ep})} u_{\ep}^{q_{\ep}+1} \nu_j\, dS_x \right| \le C \lambda_{\ep} \Lambda_{\ep}^{(n-1) - (n-2) (q_{\ep}+1)}.
\end{equation}
Therefore, putting \eqref{eq-g1-11}, \eqref{eq-g1-12} and the fact that $p < q_{\ep}$ together, we arrive at
\begin{equation}\label{eq-g1-13}
|R_j^{\ep}| \le C \lambda_{\ep} \Lambda_{\ep}^{(n-1) - (n-2)(p+1)}.
\end{equation}

\medskip
As a result, we combine \eqref{eq-g1-5} and \eqref{eq-g1-13} to derive
\[C \lambda_{\ep} \Lambda_{\ep}^{-(n-1)} \le - \sum_{j=1}^n a_j L_j^{\ep} \le \sum_{j=1}^n |a_j| |R_j^{\ep}| \le C \lambda_{\ep} \Lambda_{\ep}^{(n-1) - (n-2)(p+1)}.\]
Since $\Lambda_{\ep} \to \infty$ as $\ep \to 0$, it holds that $-(n-1) \le (n-1) - (n-2) (p+1)$, which is reduced to $p \le \frac{n}{n-2}$.
This contradicts our assumption on $p$, and so $d_{\ep}$ must be away from 0. The proof is completed.
\end{proof}

\section{Proof of Theorem \ref{thm-1} (Case 2: $p = \frac{n}{n-2}$)}\label{sec-thm-12}
In this section, we prove Theorem \ref{thm-1} for the case that $p = \frac{n}{n-2}$.
Although our strategy is the same as that given in the previous section,
there is a difference due to the fact that the function $V_0^p$ defined over $\R^n$ is not integrable for $p = \frac{n}{n-2}$.

\medskip
As before, we first need estimates for $u_{\ep}$ near the blow-up point.
\begin{lem}
Suppose that $p = \frac{n}{n-2}$, $\{(u_{\ep},v_{\ep})\}_{\ep >0}$ is a family of least energy solutions to \eqref{eq-main} with $q = q_{\ep}$
and $d_{\ep} =  \frac{1}{4} \textnormal{dist}(x_{\ep}, \pa \Omega) \to 0$ as $\ep \to 0$.
Then, for each $x \in \pa B^n(x_{\ep}, 2d_{\ep})$, we have
\begin{equation}\label{eq-g2-20}
u_{\ep}(x) = \lambda_{\ep}^{-\beta_0} K_{\ep} G(x,x_{\ep}) + o(d_{\ep}^{-(n-2)} \lambda_{\ep}^{-\beta_0})
\end{equation}
and
\begin{equation}\label{eq-g2-21}
\nabla u_{\ep}(x) = \lambda_{\ep}^{-\beta_0} K_{\ep} \nabla_x G(x,x_{\ep}) + o(d_{\ep}^{-(n-1)} \lambda_{\ep}^{-\beta_0})
\end{equation}
where $K_{\ep} > 0$ is a constant satisfying
\begin{equation}\label{eq-g2-25}
c_1 \log \Lambda_{\ep} \le K_{\ep} \le c_2 \log \Lambda_{\ep}
\end{equation}
for all $\ep > 0$ small and some $0 < c_1 < c_2$ independent of $\ep > 0$, $\beta_0 > 0$ is the constant defined in \eqref{eq-ab},
and $o$ notation is uniform with respect to $x \in \pa B^n(x_{\ep}, 2d_{\ep})$ in the sense that \eqref{eq-g1-18} holds.
\end{lem}
\begin{proof}
Fix any $x \in \pa B^n(x_{\ep}, 2d_{\ep})$. From \eqref{eq-main}, we have that
\begin{equation}\label{eq-g2-1}
\begin{aligned}
u_{\ep}(x) & = G(x,x_{\ep}) \int_{B^n(x_{\ep}, 4d_{\ep})} v_{\ep}^{p} (y) dy + \int_{B^n(x_{\ep}, 4d_{\ep})} [G(x,y)-G(x,x_{\ep})]\, v_{\ep}^p(y) dy \\
&\ + \int_{\Omega \setminus B^n(x_{\ep}, 4d_{\ep})} G(x,y) v_{\ep}^p(y) dy.
\end{aligned}
\end{equation}
We will derive \eqref{eq-g2-20} by examining each of the integrals in the right-hand side of \eqref{eq-g2-1}.

Firstly, we claim that if we set
\[K_{\ep} = \lambda_{\ep}^{\beta_0} \int_{B^n(x_{\ep}, 4d_{\ep})} v_{\ep}^{p} (y) dy,\]
then it satisfies \eqref{eq-g2-25}.
We infer from \eqref{eq-4-3} that
\[K_{\ep} \le C \int_{B^n(0,4\Lambda_{\ep})} V_{\ep}^p(y) dy \le C \log \Lambda_{\ep},\]
so an upper estimate of $K_{\ep}$ is obtained. In order to deduce its lower estimate, we will find a lower bound of the rescaled function $V_{\ep}$ on $B^n(0,r_0\Lambda_{\ep}) \setminus B^n(0,2)$
where $r_0 > 0$ is a sufficiently small constant independent of $\ep > 0$.
Thanks to \eqref{eq-ext-sol}, we know
\begin{equation}\label{eq-g2-12}
V_{\ep}(z) = \int_{\Omega_{\ep}} G_{\Omega_{\ep}}(z,y) U_{\ep}^{q_{\ep}}(y) dy
\end{equation}
where $G_{\Omega_{\ep}}$ is the Green's function of the Diriclet Laplacian in $\Omega_{\ep} = \lambda_{\ep} (\Omega - x_{\ep})$.
By the scaling property, we have
\begin{align*}
G_{\Omega_\ep} (z,y) &= \lambda_{\ep}^{-(n-2)}G(\lambda_{\ep}^{-1} z + x_{\ep}, \lambda_{\ep}^{-1} y + x_{\ep})
\\
&= \frac{c_n}{|z-y|^{n-2}} - \lambda_n^{-(n-2)}H(\lambda_{\ep}^{-1} z + x_{\ep}, \lambda_{\ep}^{-1} y + x_{\ep}) \quad \text{for } z \ne y \in \Omega_{\ep}.
\end{align*}
Moreover, the relation $\Lambda_{\ep} = \lambda_{\ep} d_{\ep}$ and estimate \eqref{eq-h1-1} imply
\[\sup_{y,z \in B^n(0, \Lambda_{\ep})} H(\lambda_{\ep}^{-1} z + x_{\ep}, \lambda_{\ep}^{-1} y + x_{\ep}) \le C d_{\ep}^{-(n-2)}.\]
Therefore, for $z \in B^n(0, r_0\Lambda_{\ep}) \setminus B^n(0,2)$ and $y \in B^n(0,1)$, it holds that
\[G_{\Omega_\ep}(z,y) = \frac{c_n}{|z-y|^{n-2}} + O(\Lambda_{\ep}^{-(n-2)})
\ge \frac{C}{|z-y|^{n-2}} \ge \frac{C}{|z|^{n-2}}\]
provided $r_0 > 0$ small enough.
Putting this estimate into \eqref{eq-g2-12} and using the fact that $U_{\ep} \to U_0$ in $C(\overline{B^n(0,1)})$ as $\ep \to 0$ reveal that
\[V_{\ep}(z) \ge \frac{C}{|z|^{n-2}} \int_{B^n(0,1)} U_{\ep}^{q_{\ep}}(y) dy \ge \frac{C}{|z|^{n-2}} \quad \text{for } z \in B^n(0,r_0\Lambda_{\ep}) \setminus B^n(0,2).\]
Consequently, we get the estimate
\[K_{\ep} = \int_{B^n(0,4\Lambda_{\ep})} V_{\ep}^{p}(y) dy \ge C \int_{B^n(0,r_0\Lambda_{\ep}) \setminus B^n(0,2)} \frac{1}{|y|^n} dy \ge C \log \Lambda_{\ep},\]
proving the assertion.

Now, it remains to deal with the second and third integrals in the right-hand side of \eqref{eq-g2-1}.
We decompose the second integral as
\begin{multline*}
\int_{B^n(x_{\ep}, 4d_{\ep})} [G(x,y)-G(x,x_{\ep})] v_{\ep}^p(y) dy \\
= \(\int_{B^n(x_{\ep}, d_{\ep})} + \int_{B^n(x_{\ep}, 4d_{\ep}) \setminus B^n(x_{\ep}, d_{\ep})}\) [G(x,y)-G(x,x_{\ep})] v_{\ep}^p(y) dy
\end{multline*}
and compute each integral as in \eqref{eq-g1-2} and \eqref{eq-g1-3}, achieving
\[\int_{B^n(x_{\ep}, 4d_{\ep})} [G(x,y)-G(x,x_{\ep})] v_{\ep}^p(y) dy = O(d_{\ep}^{-(n-2)} \lambda_{\ep}^{-\beta_0}).\]
In addition, applying the inequalities
\[|x-y| \ge |y-x_{\ep}| - |x-x_{\ep}| \ge |y-x_{\ep}| - 2d_{\ep} \ge \frac{|y-x_{\ep}|}{2} \quad \text{ for } y \in \Omega \setminus B^n(x_{\ep}, 4d_{\ep}),\]
we see that the last integral is handled as
\[\int_{\Omega \setminus B^n(x_{\ep}, 4d_{\ep})} G(x,y) v_{\ep}^p(y) dy
\le C \lambda_{\ep}^{\beta_0 p}\int_{\Omega \setminus B^n(0, 4d_{\ep})} \frac{1}{|y|^{n-2}} \frac{1}{|\lambda_{\ep} y|^n} dy = O(d_{\ep}^{-(n-2)} \lambda_{\ep}^{-\beta_0}).\]
This completes the justification of \eqref{eq-g2-20}.

\medskip
Estimate \eqref{eq-g2-21} for $\nabla u_{\ep}$ can be done analogously, so the proof is finished.
\end{proof}
\begin{proof}[Proof of Theorem \ref{thm-1} (Case 2)]
To the contrary, we assume that $d_{\ep} = \frac{1}{4} \textrm{dist}(x_{\ep}, \pa \Omega) \to 0$ as $\ep \to 0$ up to a subsequence. Again, we use $\ep$ as the parameter.

Just as in the previous section, we shall estimate the left-hand side $L_j^{\ep}$ and the right-hand side $R_j^{\ep}$ of the local Pohozaev-type identity \eqref{eq-poho-0} for each $1 \le j \le n$, respectively,
and then induce a contradiction from the identity $L_j^{\ep} = R_j^{\ep}$.

\medskip \noindent \textbf{Estimate of $L_j^{\ep}$.}
Similarly to \eqref{eq-g1-14}, we insert the estimates of $u_{\ep}$ and $v_{\ep}$ given in \eqref{eq-g2-21} and \eqref{eq-g1-8} into $L_j^{\ep}$ to get
\[L_j^{\ep} = -\lambda_{\ep}^{-(n-2)} A_{U_0} K_{\ep} \mci_1(2d_{\ep}) + O(d_{\ep}^{-(n-1)} \lambda_{\ep}^{-(n-2)}).\]
Here $\mci_1(r)$ is the constant function introduced in \eqref{eq-g1-21}.
Then, as in \eqref{eq-g1-45}, we discover
\[L_j^{\ep} = - 2c_n(n-2) A_{U_0} K_{\ep} |\S^{n-1}| \lambda_{\ep}^{-(n-2)} \left. \frac{\pa H}{\pa x_j}(x, x_{\ep}) \right|_{x = x_{\ep}} + O(d_{\ep}^{-(n-1)} \lambda_{\ep}^{-(n-2)})\]
by evaluating $\lim_{r \to 0} \mci_1(r)$. Let $\nu_{x_{\ep}} = (a_1, \cdots, a_n) \in \S^{n-1}$. Then \eqref{eq-g2-25} implies that
\begin{equation}\label{eq-g1-15}
\begin{aligned}
- \sum_{j=1}^n a_j L_j^{\ep} &\ge C (\log \Lambda_{\ep}) \lambda_{\ep}^{-(n-2)}
\left. \frac{\pa H}{\pa \nu_{x_\ep}}(x, x_{\ep}) \right|_{x = x_{\ep}} + O(d_{\ep}^{-(n-1)} \lambda_{\ep}^{-(n-2)}) \\
&\ge C (\log \Lambda_{\ep}) d_{\ep}^{-(n-1)}  \lambda_{\ep}^{-(n-2)} = C \lambda_{\ep} (\log \Lambda_{\ep}) \Lambda_{\ep}^{-(n-1)}
\end{aligned}
\end{equation}
for some $C > 0$.

\medskip \noindent \textbf{Estimate of $R_j^{\ep}$.}
By \eqref{eq-g1-11}, we have
\[\left| \int_{\pa B^n(x_{\ep}, 2d_{\ep})} v_{\ep}^{p+1} \nu_j\, dS_x \right| \le C \lambda_{\ep} \Lambda_{\ep}^{(n-1)-(n-2)(p+1)} = C \lambda_{\ep} \Lambda_{\ep}^{-(n-1)}.\]
Moreover, it holds that $u_{\ep}(x) \le C \lambda_{\ep}^{\alpha_0} \Lambda_{\ep}^{-(n-2)} \log \Lambda_{\ep}$ for any $x \in \pa B^n(x_{\ep}, 2d_{\ep})$ and that $q_{\ep} > p = \frac{n}{n-2}$, so
\[\left| \int_{\pa B^n(x_{\ep}, 2d_{\ep})} u_{\ep}^{q_{\ep}+1}(x) \nu_j\, dS_x \right|
\le C \lambda_{\ep} \Lambda_{\ep}^{(n-1) - (n-2) (q_{\ep}+1)} (\log \Lambda_{\ep})^{q_{\ep}+1} = o(\lambda_{\ep} \Lambda_{\ep}^{-(n-1)}).\]
Therefore
\begin{equation}\label{eq-g1-6}
|R_j^{\ep}| \le C \lambda_{\ep} \Lambda_{\ep}^{-(n-1)}.
\end{equation}

\medskip
By combining \eqref{eq-g1-15} and \eqref{eq-g1-6}, we obtain
\[C \lambda_{\ep} (\log \Lambda_{\ep}) \Lambda_{\ep}^{-(n-1)} \le - \sum_{j=1}^n a_j L_j^{\ep} \le \sum_{j=1}^n |a_j| |R_j^{\ep}| \le C \lambda_{\ep} \Lambda_{\ep}^{-(n-1)}.\]
Since $\Lambda_{\ep} \to \infty$ as $\ep \to 0$, a contradiction arises and so $d_{\ep}$ must be away from 0. This concludes the proof.
\end{proof}

\section{Proof of Theorem \ref{thm-2}}\label{sec-thm-2}
This section is devoted to the proof of Theorem \ref{thm-2}.
We keep assuming that $d_{\ep} \to 0$ as $\ep \to 0$.
We also recall that $\Lambda_{\ep} = d_{\ep} \lambda_{\ep} \to \infty$ as $\ep \to 0$ which is verified in Lemma \ref{lem-lam}.

\medskip
In the following lemma, we obtain a pointwise estimate for $v_{\ep}$ outside the blow-up point $x_{\ep}$, which can be regarded as an extension of Lemma \ref{lem-g10}.
This estimate is essential in deriving an estimate of $u_{\ep}$ that will be described in Lemma \ref{lem-8-2}.
\begin{lem}\label{lem-8-1}
Suppose that $p \in (\frac{2}{n-2}, \frac{n}{n-2})$. For any $x \in \Omega \setminus B^n(0, \frac{3d_{\ep}}{\sqrt{\Lambda_{\ep}}})$, we have
\begin{equation}\label{eq-v-2}
v_{\ep}(x) = \lambda_{\ep}^{-\alpha_0} A_{U_0} G(x,x_{\ep}) + \lambda_{\ep}^{-\alpha_0} Q_{\ep}(x)
\end{equation}
where $A_{U_0}$ and $\alpha_0$ are the positive constants defined in \eqref{eq-AUV} and \eqref{eq-ab}, respectively, and $Q_{\ep}$ is a remainder term which satisfies
\[\sup_{x \in \Omega \setminus B^n(0, \frac{3d_{\ep}}{\sqrt{\Lambda_{\ep}}})} |x-x_{\ep}|^{n-2}|Q_{\ep}(x)| \to 0 \quad \text{as } \ep \to 0.\]
\end{lem}
\begin{proof}
Assuming that $x \in \Omega \setminus B^n(0, \frac{3d_{\ep}}{\sqrt{\Lambda_{\ep}}})$, we write
\begin{equation}\label{eq-v-1}
\lambda_{\ep}^{\alpha_0} v_{\ep}(x) = \(\int_{B^n(x_{\ep}, {d_{\ep} \over \sqrt{\Lambda_{\ep}}})} + \int_{\Omega \setminus B^n(x_{\ep}, {d_{\ep} \over \sqrt{\Lambda_{\ep}}})}\) G(x,y) \lambda_{\ep}^{\alpha_0} u_{\ep}^{q_{\ep}}(y) dy.
\end{equation}
We will analyze two integrals in the right-hand side.

We consider the first term of \eqref{eq-v-1}. We decompose it by
\begin{align*}
&\ \int_{B^n(x_{\ep}, {d_{\ep} \over \sqrt{\Lambda_{\ep}}})} G(x,y) \lambda_{\ep}^{\alpha_0} u_{\ep}^{q_{\ep}}(y) dy
\\
&= G(x, x_{\ep}) \int_{B^n(x_{\ep}, {d_{\ep} \over \sqrt{\Lambda_{\ep}}})} \lambda_{\ep}^{\alpha_0} u_{\ep}^{q_{\ep}}(y) dy
+ \int_{B^n(x_{\ep}, {d_{\ep} \over \sqrt{\Lambda_{\ep}}})} [G(x,y)-G(x,x_{\ep})] \lambda_{\ep}^{\alpha_0} u_{\ep}^{q_{\ep}}(y) dy.
\end{align*}
By \eqref{eq-tuv}, Lemma \ref{lem-lam}, Proposition \ref{prop-dec} and the dominated convergence theorem,
\[\lim_{\ep \to 0} \int_{B^n(x_{\ep}, {d_{\ep} \over \sqrt{\Lambda_{\ep}}})} \lambda_{\ep}^{\alpha_0} u_{\ep}^{q_{\ep}}(y) dy
= \lim_{\ep \to 0} \int_{B^n(0,\sqrt{\Lambda_{\ep}})} \lambda_{\ep}^{\alpha_0 + \alpha_{\ep}q_{\ep}-n} U_{\ep}^{q_{\ep}}(y) dy  = \int_{\R^n} U_0^{q_0}(y) dy = A_{U_0}.\]
Moreover, for each $x \in \Omega \setminus B^n(0, \frac{3d_{\ep}}{\sqrt{\Lambda_{\ep}}})$ and $y \in B^n(x_{\ep}, \frac{d_{\ep}}{\sqrt{\Lambda_{\ep}}})$, we have
\begin{align*}
\left|\nabla_y G(x,y) \right| = \left|\nabla_y G(y,x) \right|
&\le \frac{C \sqrt{\Lambda_{\ep}}}{d_{\ep}} \sup_{z \in B^n(x_{\ep}, \frac{2d_{\ep}}{\sqrt{\Lambda_{\ep}}})} |G(z,x)| \\
&\le \frac{C \sqrt{\Lambda_{\ep}}}{d_{\ep}} \sup_{z \in B^n(x_{\ep}, \frac{2d_{\ep}}{\sqrt{\Lambda_{\ep}}})} \frac{1}{|x-z|^{n-2}}
\le \frac{C \sqrt{\Lambda_{\ep}}}{d_{\ep}} \frac{1}{|x-x_{\ep}|^{n-2}}.
\end{align*}
Hence the mean value theorem shows that
\begin{align*}
&\ \left| \int_{B^n(x_{\ep}, {d_{\ep} \over \sqrt{\Lambda_{\ep}}})} [G(x,y)-G(x,x_{\ep})] \lambda_{\ep}^{\alpha_0} u_{\ep}^{q_{\ep}}(y) dy \right| \\
&\le \frac{C \sqrt{\Lambda_{\ep}}}{d_{\ep}} \frac{1}{|x-x_{\ep}|^{n-2}} \int_{B^n(x_{\ep}, {d_{\ep} \over \sqrt{\Lambda_{\ep}}})} |y-x_{\ep}|
\lambda_{\ep}^{\alpha_0 + \alpha_{\ep}q_{\ep}} U_{\ep}^{q_{\ep}} (\lambda_{\ep} (y-x_{\ep})) dy \\
&\le \frac{C}{\sqrt{\Lambda_{\ep}}} \frac{1}{|x-x_{\ep}|^{n-2}} \int_{\R^n} |y| U_0^{q_0}(y) dy
\le \frac{C}{\sqrt{\Lambda_{\ep}}} \frac{1}{|x-x_{\ep}|^{n-2}},
\end{align*}
where the relation $((n-2)p-2)q_0 = n+2(p+1) > n+1$ guarantees that the value of the integral on the last line is finite. As a result, in view of \eqref{eq-g1-9}, we discover
\begin{equation}\label{eq-u-51}
\begin{aligned}
\int_{B^n(x_{\ep}, {d_{\ep} \over \sqrt{\Lambda_{\ep}}})} G(x,y) \lambda_{\ep}^{\alpha_0} u_{\ep}^{q_{\ep}}(y) dy
&= (A_{U_0} + o(1)) G(x,x_{\ep}) + O\(\frac{1}{\sqrt{\Lambda_{\ep}}} \frac{1}{|x-x_{\ep}|^{n-2}}\) \\
&= A_{U_0} G(x,x_{\ep}) + o\(\frac{1}{|x-x_{\ep}|^{n-2}}\)
\end{aligned}
\end{equation}
where $O$ and $o$ notations are uniform with respect to $x \in \Omega \setminus B^n(0, \frac{3d_{\ep}}{\sqrt{\Lambda_{\ep}}})$.

We turn to estimating the second integral in the right-hand side of \eqref{eq-v-1}. It holds that
\begin{equation}\label{eq-u-53}
\begin{aligned}
&\ \int_{\Omega \setminus B^n(x_{\ep}, {d_{\ep} \over \sqrt{\Lambda_{\ep}}})} G(x,y) \lambda_{\ep}^{\alpha_0} u_{\ep}^{q_{\ep}}(y) dy \\
&\le C \int_{\Omega \setminus B^n(x_{\ep}, {d_{\ep} \over \sqrt{\Lambda_{\ep}}})} \frac{1}{|x-y|^{n-2}}
\frac{\lambda_{\ep}^n}{(1+\lambda_{\ep} |y-x_{\ep}|)^{((n-2)p-2)q_{\ep}}} dy \\
&\le C \int_{\Omega \setminus B^n(x_{\ep}, {d_{\ep} \over \sqrt{\Lambda_{\ep}}})} \frac{1}{|x-y|^{n-2}}
\frac{\lambda_{\ep}^{n-((n-2)p-2)q_{\ep}}}{|y-x_{\ep}|^{((n-2)p-2)q_{\ep}}} dy.
\end{aligned}
\end{equation}
To calculate the last integral, we split its domain into
\[\Omega \setminus B^n\(x_{\ep}, {d_{\ep} \over \sqrt{\Lambda_{\ep}}}\) = D_4 \cup D_5 \cup D_6\]
where
\[D_4 := \left\{y \in \Omega: |y-x| \le \dfrac{|x-x_{\ep}|}{2} \right\}, \quad
D_5 := \left\{y \in \Omega: \frac{d_{\ep}}{\sqrt{\Lambda_{\ep}}} < |y-x_{\ep}| \le \dfrac{|x-x_{\ep}|}{2}\right\}\]
and
\[D_6 := \left\{y \in \Omega: |y-x| > \dfrac{|x-x_{\ep}|}{2},\ |y-x_{\ep}| > \dfrac{|x-x_{\ep}|}{2} \right\}.\]
We note that
\[\begin{cases}
|y-x_{\ep}| \ge |x-x_{\ep}| - |y-x| \ge \dfrac{|x-x_{\ep}|}{2} &\text{if } y \in D_4,\\
|y-x| \ge |x-x_{\ep}| - |y-x_{\ep}| \ge \dfrac{|x-x_{\ep}|}{2} &\text{if } y \in D_5,\\
|y-x_{\ep}| > \dfrac{|y-x_{\ep}|}{2} + \dfrac{|x-x_{\ep}|}{4} > \dfrac{|y-x|}{4} &\text{if } y \in D_6.
\end{cases}\]
Using this, we can compute the integral over $D_4$ as
\begin{align*}
\int_{D_4} \frac{1}{|x-y|^{n-2}} \frac{\lambda_{\ep}^{n-((n-2)p-2)q_{\ep}}}{|y-x_{\ep}|^{((n-2)p-2)q_{\ep}}} dy
&\le \frac{C \lambda_{\ep}^{n-((n-2)p-2)q_{\ep}}}{|x-x_{\ep}|^{((n-2)p-2)q_{\ep}}} \int_{D_4} \frac{1}{|x-y|^{n-2}} dy
\\
&\le \frac{C}{(\lambda_{\ep}|x-x_{\ep}|)^{((n-2)p-2)q_{\ep}-n}} \frac{1}{|x-x_{\ep}|^{n-2}}
\\
&\le \frac{C}{\Lambda_{\ep}^{[((n-2)p-2)q_{\ep}-n]/2}} \frac{1}{|x-x_{\ep}|^{n-2}}.
\end{align*}
Similarly, we estimate the integrals over $D_5$ and $D_6$ as
\begin{align*}
\int_{D_5} \frac{1}{|x-y|^{n-2}} \frac{\lambda_{\ep}^{n-((n-2)p-2)q_{\ep}}}{|y-x_{\ep}|^{((n-2)p-2)q_{\ep}}} dy
&\le \frac{C \lambda_{\ep}^{n-((n-2)p-2)q_{\ep}}}{|x-x_{\ep}|^{n-2}} \int_{\Omega \setminus B^n(x_{\ep}, {d_{\ep} \over \sqrt{\Lambda_{\ep}}})} \frac{1}{|y-x_{\ep}|^{((n-2)p-2)q_{\ep}}} dy \\
&\le \frac{C \lambda_{\ep}^{n-((n-2)p-2)q_{\ep}}}{|x-x_{\ep}|^{n-2}} \({d_{\ep} \over \sqrt{\Lambda_{\ep}}}\)^{n-((n-2)p-2)q_{\ep}}
\\
&= \frac{C}{\Lambda_{\ep}^{[((n-2)p-2)q_{\ep}-n]/2}} \frac{1}{|x-x_{\ep}|^{n-2}}
\end{align*}
and
\begin{align*}
\int_{D_6} \frac{1}{|x-y|^{n-2}} \frac{\lambda_{\ep}^{n-((n-2)p-2)q_{\ep}}}{|y-x_{\ep}|^{((n-2)p-2)q_{\ep}}} dy
&\le C \int_{\{y \in \Omega:\, |y-x| > \frac{|x-x_{\ep}|}{2}\}} \frac{\lambda_{\ep}^{n-((n-2)p-2)q_{\ep}}}{|x-y|^{n-2+((n-2)p-2)q_{\ep}}} dy \\
&\le \frac{\lambda_{\ep}^{n-((n-2)p-2)q_{\ep}}}{|x-x_{\ep}|^{((n-2)p-2)q_{\ep}-2}}
= \frac{C}{\Lambda_{\ep}^{[((n-2)p-2)q_{\ep}-n]/2}} \frac{1}{|x-x_{\ep}|^{n-2}}.
\end{align*}
Consequently,
\begin{equation}\label{eq-u-52}
\int_{\Omega \setminus B^n(x_{\ep}, {d_{\ep} \over \sqrt{\Lambda_{\ep}}})} G(x,y) \lambda_{\ep}^{\alpha_0} u_{\ep}^{q_{\ep}}(y) dy
\le \frac{C}{\Lambda_{\ep}^{[((n-2)p-2)q_{\ep}-n]/2}} \frac{1}{|x-x_{\ep}|^{n-2}} = o\(\frac{1}{|x-x_{\ep}|^{n-2}}\).
\end{equation}

Estimate \eqref{eq-v-2} now follows from \eqref{eq-u-51} and \eqref{eq-u-52}. The proof is completed.
\end{proof}

We deduce estimates for $u_{\ep}$ near the blow-up point.
\begin{lem}\label{lem-8-2}
Suppose that $p \in (\frac{2}{n-2}, \frac{n}{n-2})$. For each point $x \in \pa B^n(x_{\ep}, 2d_{\ep})$, we have
\begin{equation}\label{eq-u-70}
u_{\ep}(x) = \lambda_{\ep}^{-\alpha_0p} A_{U_0}^p \wtg(x,x_{\ep}) + o(d_{\ep}^{-(n-2)p+2} \lambda_{\ep}^{-\alpha_0p})
\end{equation}
and
\begin{equation}\label{eq-u-71}
\nabla u_{\ep}(x) = \lambda_{\ep}^{-\alpha_0p} A_{U_0}^p \nabla_x \wtg(x,x_{\ep}) + o(d_{\ep}^{-(n-2)p+1} \lambda_{\ep}^{-\alpha_0p}).
\end{equation}
Here, the definition of the numbers $A_{U_0}$ and $\alpha_0$ can be found in \eqref{eq-AUV} and \eqref{eq-ab},
and $o$ notation is uniform with respect to $x \in \pa B^n(x_{\ep}, 2d_{\ep})$ in the sense that \eqref{eq-g1-18} holds.
\end{lem}
\begin{proof}
It holds that
\begin{equation}\label{eq-u-1}
\begin{aligned}
\lambda_{\ep}^{\alpha_0p} u_{\ep}(x)
&= \lambda_{\ep}^{\alpha_0p} \int_{B^n(x_{\ep}, \frac{3d_{\ep}}{\sqrt{\Lambda_{\ep}}})} G(x,y) v_{\ep}^p(y) dy
+ A_{U_0}^p \int_{\Omega \setminus B^n(x_{\ep}, \frac{3d_{\ep}}{\sqrt{\Lambda_{\ep}}})} G(x,y) G^p(y,x_{\ep}) dy \\
&\ + \int_{\Omega \setminus B^n(x_{\ep}, \frac{3d_{\ep}}{\sqrt{\Lambda_{\ep}}})}
G(x,y) \left[ \lambda_{\ep}^{\alpha_0p}v_{\ep}^p(y) - A_{U_0}^p G^p(y,x_{\ep}) \right] dy.
\end{aligned}
\end{equation}
We shall compute each of three integrals in the right-hand side.

By the identity $\alpha_0 + \beta_0 = n-2$, \eqref{eq-g1-9} and \eqref{eq-4-3}, the first integral in the right-hand side of \eqref{eq-u-1} is estimated as
\begin{align*}
\lambda_{\ep}^{\alpha_0p} \int_{B^n(x_{\ep}, \frac{3d_{\ep}}{\sqrt{\Lambda_{\ep}}})} G(x,y) v_{\ep}^p(y) dy
&\le C \lambda_{\ep}^{(n-2)p} \int_{B^n(x_{\ep}, \frac{3d_{\ep}}{\sqrt{\Lambda_{\ep}}})} \frac{1}{|x-y|^{n-2}} V_{\ep}^p(\lambda_{\ep}(y-x_{\ep})) dy \\
&\le C \lambda_{\ep}^{(n-2)p-n} d_{\ep}^{-(n-2)} \int_{B^n(0, 3\sqrt{\Lambda_{\ep}})} V_{\ep}^p(y) dy \\
&\le C \lambda_{\ep}^{(n-2)p-n} d_{\ep}^{-(n-2)} (\sqrt{\Lambda_{\ep}})^{n-(n-2)p} \\
&= C (\sqrt{\Lambda_{\ep}})^{(n-2)p-n} d_{\ep}^{2-(n-2)p} = o(d_{\ep}^{2-(n-2)p}).
\end{align*}

Moreover, from the definition of $\wtg$ determined by \eqref{eq-h2-2}, we get
\[\int_{\Omega \setminus B^n(x_{\ep}, \frac{3d_{\ep}}{\sqrt{\Lambda_{\ep}}})} G(x,y) G^p(y,x_{\ep}) dy
= \wtg(x,x_{\ep}) - \int_{B^n(x_{\ep}, \frac{3d_{\ep}}{\sqrt{\Lambda_{\ep}}})} G(x,y) G^p(y,x_{\ep}) dy.\]
For $x \in \pa B^n(x_{\ep}, 2d_{\ep})$ and $y \in B^n(x_{\ep}, \frac{3d_{\ep}}{\sqrt{\Lambda_{\ep}}})$, we have that $|y-x| \ge d_{\ep}$. Therefore
\begin{align*}
\int_{B^n(x_{\ep}, \frac{3d_{\ep}}{\sqrt{\Lambda_{\ep}}})} G(x,y) G^p(y,x_{\ep}) dy
&\le C d_{\ep}^{-(n-2)} \(\frac{d_{\ep}}{\sqrt{\Lambda_{\ep}}}\)^{-(n-2)p+n} \\
&= C (\sqrt{\Lambda_{\ep}})^{(n-2)p-n} d_{\ep}^{2-(n-2)p} = o(d_{\ep}^{2-(n-2)p}),
\end{align*}
and thus the second integral in the right-hand side of \eqref{eq-u-1} is calculated as
\begin{equation}\label{eq-u-62}
A_{U_0}^p \int_{\Omega \setminus B^n(x_{\ep}, \frac{3d_{\ep}}{\sqrt{\Lambda_{\ep}}})} G(x,y) G^p(y,x_{\ep}) dy = A_{U_0}^p \wtg(x,x_{\ep}) + o(d_{\ep}^{2-(n-2)p}).
\end{equation}

On the other hand, by applying the elementary inequality
\[|a^p - b^p| \le 2^{p-1}p\, |a-b|(a^{p-1}+b^{p-1}) \quad \text{for any } a, b > 0\]
and Lemma \ref{lem-8-1}, we can easily deduce that
\begin{align*}
&\ \left| \int_{\Omega \setminus B^n(x_{\ep}, \frac{3d_{\ep}}{\sqrt{\Lambda_{\ep}}})}
G(x,y) \left[ \lambda_{\ep}^{\alpha_0p}v_{\ep}^p(y) - A_{U_0}^p G^p(y,x_{\ep}) \right] dy \right| \\
&\le C \int_{\Omega \setminus B^n(x_{\ep}, \frac{3d_{\ep}}{\sqrt{\Lambda_{\ep}}})} \frac{1}{|x-y|^{n-2}} \frac{|Q_{\ep}(y)|}{|y-x_{\ep}|^{(n-2)(p-1)}} dy \\
&= o(1) \int_{\Omega \setminus B^n(x_{\ep}, \frac{3d_{\ep}}{\sqrt{\Lambda_{\ep}}})} \frac{1}{|x-y|^{n-2}} \frac{1}{|y-x_{\ep}|^{(n-2)p}} dy.
\end{align*}
To examine the last integral, we divide the domain of integration $\Omega \setminus B^n(x_{\ep}, \frac{3d_{\ep}}{\sqrt{\Lambda_{\ep}}})$ into
\[\Omega \setminus B^n\(x_{\ep}, \frac{3d_{\ep}}{\sqrt{\Lambda_{\ep}}}\)
= B^n(x,d_{\ep}) \cup \left[ B^n(x_{\ep}, d_{\ep}) \setminus B^n\(x_{\ep}, \frac{3d_{\ep}}{\sqrt{\Lambda_{\ep}}}\) \right]
\cup \left[\Omega \setminus \(B^n(x,d_{\ep}) \cup B^n(x_{\ep},d_{\ep})\) \right],\]
and evaluate the integral of $|x-y|^{-(n-2)}|y-x_{\ep}|^{-(n-2)p}$ over each subdomain, as we did for the integral in the rightmost side of \eqref{eq-u-53}.
Then we observe that all integrals are bounded by $o(d_{\ep}^{2-(n-2)p})$.

Having this fact, \eqref{eq-u-1} and \eqref{eq-u-62}, we conclude that \eqref{eq-u-70} is true.

\medskip
In the same manner, we can prove that \eqref{eq-u-71} is valid. In this time, we have to use the gradient estimate of $G(x,y)$ in \eqref{eq-g1-9}.
\end{proof}
\begin{proof}[Proof of Theorem \ref{thm-2}]
To the contrary, we assume that $d_{\ep} = \frac{1}{4} \textrm{dist}(x_{\ep}, \pa \Omega) \to 0$ as $\ep \to 0$ along a subsequence. Let us keep using $\ep$ as the parameter.

For each $1 \le j \le n$, let
\begin{multline*}
\mcl_j^{\ep} = -\int_{\pa B^n(x_{\ep}, 2d_{\ep})} \( \frac{\pa u_{\ep}}{\pa \nu} \frac{\pa v_{\ep}}{\pa x_j} + \frac{\pa v_{\ep}}{\pa \nu} \frac{\pa u_{\ep}}{\pa x_j} \) dS_x \\
+ \int_{\pa B^n(x_{\ep}, 2d_{\ep})} (\nabla u_{\ep} \cdot \nabla v_{\ep}) \nu_j\, dS_x - \frac{1}{p+1} \int_{\pa B^n(x_{\ep}, 2d_{\ep})} v_{\ep}^{p+1} \nu_j\, dS_x
\end{multline*}
and
\[\mcr_j^{\ep} = \frac{1}{q_{\ep}+1} \int_{\pa B^n(x_{\ep}, 2d_{\ep})} u_{\ep}^{q_{\ep}+1} \nu_j\, dS_x.\]
By the Pohozaev identity \eqref{eq-poho-0}, it holds that $\mcl_j^{\ep} = \mcr_j^{\ep}$.
As in the proof of Theorem \ref{thm-1}, we will estimate $\mcl_j^{\ep}$ and $\mcr_j^{\ep}$, respectively, and derive a contradiction by comparing them.

\medskip \noindent \textbf{Estimate of $\mcl_j^{\ep}$.}
The standard gradient estimate of Poisson's equation yields
\begin{align*}
&\ \left\|\nabla_x \wtg(x,x_{\ep}) \right\|_{L^{\infty}(B^n(x_{\ep},3d_{\ep}) \setminus B(x_{\ep},d_{\ep}))} \\
&\le \frac{C}{d_{\ep}} \left\| \wtg(x,x_{\ep}) \right\|_{L^{\infty}(B^n(x_{\ep},4d_{\ep}) \setminus B^n(x_{\ep}, d_{\ep}/2))}
+ C d_{\ep} \left\| G^p(\cdot,x_{\ep}) \right\|_{L^{\infty}(B^n(x_{\ep},4d_{\ep}) \setminus B^n(x_{\ep},d_{\ep}/2))} \\
&\le C d_{\ep}^{1-(n-2)p},
\end{align*}
which indicates
\[\int_{\pa B^n(x_{\ep},2d_{\ep})} |\nabla_x \wtg (x,x_{\ep})|\, dS_x \le C d_{\ep}^{n-(n-2)p}.\]
Accordingly, by \eqref{eq-u-71}, \eqref{eq-g1-0}, \eqref{eq-g1-8}, \eqref{eq-g1-9} and the previous estimate,
\begin{equation}\label{eq-h3-1}
\mcl_j^{\ep} = \lambda_{\ep}^{-\frac{n(p+1)}{q_0+1}} \mci_2(2d_{\ep}) + o\(d_{\ep}^{1-(n-2)p} \lambda_{\ep}^{-\frac{n(p+1)}{q_0+1}}\)
\end{equation}
where
\begin{equation}\label{eq-h3-4}
\begin{aligned}
\mci_2(r) &:= - \int_{\pa B^n(x_{\ep}, r)} \( \frac{\pa \wtg}{\pa \nu} (x,x_{\ep}) \frac{\pa G}{\pa x_j}(x,x_{\ep}) + \frac{\pa G}{\pa \nu}(x,x_{\ep}) \frac{\pa \wtg}{\pa x_j}(x,x_{\ep})\) dS_x \\
&\ + \int_{\pa B^n(x_{\ep}, r)} \( \nabla_x G (x,x_{\ep}) \cdot \nabla_x \wtg(x,x_{\ep}) \) \nu_j\, dS_x \\
&\ - \frac{1}{p+1} A_{U_0}^{p+1} \int_{\pa B^n(x_{\ep}, r)} G^{p+1}(x,x_{\ep}) \nu_j\, dS_x.
\end{aligned}
\end{equation}

To compute the value of $\mci_2(2d_{\ep})$, we first observe that the value of $\mci_2(r)$ is independent of $r>0$.
To this end, we recall from \eqref{eq-h2-2} that
\[-\Delta_x \wtg(x,x_{\ep}) = G^p(x,x_{\ep}) \quad \text{and} \quad -\Delta_x G(x,x_{\ep}) = 0 \quad \text{in } A_r = B^n(x_{\ep},2d_{\ep}) \setminus B^n(x_{\ep},r)\]
with $r \in (0,2d_{\ep})$. Integrating by parts, we get
\begin{multline*}
\frac{1}{p+1} \int_{\pa A_r} G^{p+1}(x,x_{\ep}) \nu_j\, dS_x
=\int_{A_r} G^p(x,x_{\ep}) \frac{\pa G}{\pa x_j}(x,x_{\ep}) dx
= \int_{A_r} -\Delta_x \wtg (x,x_{\ep})\frac{\pa G}{\pa x_j}(x,x_{\ep}) dx \\
= - \int_{\pa A_r} \frac{\pa \wtg}{\pa \nu}(x,x_{\ep}) \frac{\pa G}{\pa x_j} (x,x_{\ep})dS_x + \int_{A_r} \nabla_x \wtg (x,x_{\ep}) \cdot \nabla_x \frac{\pa G}{\pa x_j}(x,x_{\ep}) dx
\end{multline*}
and
\begin{align*}
0 &= \int_{A_r} \frac{\pa \wtg}{\pa x_j} (x,x_{\ep}) (-\Delta_x G) (x,x_{\ep}) dx \\
&= -\int_{\pa A_r}\frac{\pa \wtg}{\pa x_j}(x,x_{\ep}) \frac{\pa G}{\pa \nu} (x,x_{\ep}) dS_x + \int_{A_r} \nabla_x \frac{\pa \wtg}{\pa x_j}(x,x_{\ep}) \cdot \nabla_x G(x,x_{\ep}) dx.
\end{align*}
By summing these two equalities and performing a further integration by parts, we obtain
\begin{align*}
\frac{1}{p+1} \int_{\pa A_r} G^{p+1} (x,x_{\ep}) \nu_j\, dS_x &= \int_{A_r} \( \nabla_x G (x,x_{\ep}) \cdot \nabla_x \wtg(x,x_{\ep}) \) \nu_j\, dS_x \\
&\ - \int_{A_r} \( \frac{\pa \wtg}{\pa \nu} (x,x_{\ep}) \frac{\pa G}{\pa x_j}(x,x_{\ep}) + \frac{\pa G}{\pa \nu}(x,x_{\ep}) \frac{\pa \wtg}{\pa x_j}(x,x_{\ep})\) dS_x,
\end{align*}
which implies that $\mci_2(r)$ is a constant function on $r \in (0,2d_{\ep})$. In particular,
\[\mci_2(2d_{\ep}) = \lim_{r \to 0} \mci_2(r).\]

We now determine this limit. For the moment, we assume that $p \in [\frac{n-1}{n-2}, \frac{n}{n-2})$. Then \eqref{eq-h2-1} implies
\begin{align*}
&\ \lim_{r \to 0} \int_{\pa B^n(x_{\ep}, r)} \frac{\pa \wtg}{\pa \nu} (x,x_{\ep}) \frac{\pa G}{\pa x_j}(x,x_{\ep}) dS_x \\
&= \lim_{r \to 0} \int_{\pa B^n(x_{\ep},r)} \left[ -\frac{\gamma_1 ((n-2)p-2)}{|x-x_{\ep}|^{(n-2)p-1}}
+ \frac{\gamma_2((n-2)p-n)}{|x-x_{\ep}|^{(n-2)p-n+1}} H(x,x_{\ep}) + \frac{\gamma_2}{|x-x_{\ep}|^{(n-2)p-n}} \frac{\pa H}{\pa \nu}(x,x_{\ep}) \right. \\
&\hspace{80pt} \left. - \frac{\pa \wth}{\pa \nu}(x,x_{\ep}) \right] \cdot \left[ -(n-2)c_n\frac{(x-x_{\ep})_j}{|x-x_{\ep}|^n} - \frac{\pa H}{\pa x_j}(x,x_{\ep}) \right] dS_x
\\
&= J_1 + J_2 + J_3 + J_4
\end{align*}
where
\begin{align*}
J_1 &:= \gamma_1c_n (n-2)((n-2)p-2) \lim_{r \to 0} \int_{\pa B^n(x_{\ep},r)} \frac{(x-x_{\ep})_j}{|x-x_{\ep}|^{(n-2)p+n-1}} dS_x,
\\
J_2 &:= -(n-2)c_n \lim_{r \to 0} \int_{\pa B^n(x_{\ep},r)} \left[ \frac{\gamma_2((n-2)p-n)}{|x-x_{\ep}|^{(n-2)p-n+1}} H(x,x_{\ep})
+ \frac{\gamma_2}{|x-x_{\ep}|^{(n-2)p-n}} \frac{\pa H}{\pa \nu}(x,x_{\ep}) \right.
\\
&\hspace{140pt} \left. - \frac{\pa \wth}{\pa \nu}(x,x_{\ep}) \right] \cdot \frac{(x-x_{\ep})_j}{|x-x_{\ep}|^n}\, dS_x,
\\
J_3 &:= \gamma_1 ((n-2)p-2) \lim_{r \to 0} \int_{\pa B^n(x_{\ep},r)} \frac{1}{|x-x_{\ep}|^{(n-2)p-1}} \frac{\pa H}{\pa x_j}(x,x_{\ep}) dS_x,
\\
J_4 &:= - \lim_{r \to 0} \int_{\pa B^n(x_{\ep},r)} \left[ \frac{\gamma_2((n-2)p-n)}{|x-x_{\ep}|^{(n-2)p-n+1}} H(x,x_{\ep})
+ \frac{\gamma_2}{|x-x_{\ep}|^{(n-2)p-n}} \frac{\pa H}{\pa \nu}(x,x_{\ep}) - \frac{\pa \wth}{\pa \nu}(x,x_{\ep}) \right]
\\
&\hspace{95pt} \times \frac{\pa H}{\pa x_j}(x,x_{\ep}) dS_x.
\end{align*}
We immediately observe that $J_1 = 0$ since the integrand is odd. Furthermore, it is easy to see that $J_3 = J_4 =0$
by concerning the order of the singularities in the integrands (refer to Lemma \ref{lem-H-est}) and the condition that $p < \frac{n}{n-2}$.
It is worth to mention that we are conducting the computations for each fixed parameter $\ep > 0$, and in particular, for each fixed number $d_{\ep} > 0$.

We turn to compute $J_2$. Because $p < \frac{n}{n-2}$, it is true that
\begin{align*}
\int_{\pa B^n(x_{\ep},r)} \frac{(x-x_{\ep})_j}{|x-x_{\ep}|^{(n-2)p+1}} H(x,x_{\ep}) dS_x & = \int_{\pa B^n(x_{\ep},r)} \frac{(x-x_{\ep})_j}{|x-x_{\ep}|^{(n-2)p+1}} [H(x,x_{\ep})- H(x_{\ep}, x_{\ep})] dS_x \\
&= O \( \int_{\pa B^n(x_{\ep},r)} \frac{dS_x}{|x-x_{\ep}|^{(n-2)p-1}} \) \\
&= O(r^{n-(n-2)p}) \to 0 \quad \text{as } r \to 0.
\end{align*}
By considering the order of the singularity, we obtain
\[\lim_{r \to 0} \int_{\pa B^n(x_{\ep}, r)} \frac{(x-x_{\ep})_j}{|x-x_{\ep}|^{(n-2)p}} \frac{\pa H}{\pa \nu} (x,x_{\ep}) dS_x = 0\]
and
\[(n-2) c_n \lim_{r \to 0} \int_{\pa B^n(x_{\ep},r)} \frac{(x-x_{\ep})_j}{|x-x_{\ep}|^n} \frac{\pa \wth}{\pa \nu}(x,x_{\ep}) dS_x
= \frac{(n-2)c_n}{n} \left. \frac{\pa \wth}{\pa x_j}(x, x_{\ep}) \right|_{x = x_{\ep}} |\S^{n-1}|.\]
As a result,
\begin{equation}\label{eq-h3-8}
\lim_{r \to 0} \int_{\pa B^n(x_{\ep}, r)} \frac{\pa \wtg}{\pa \nu} (x,x_{\ep}) \frac{\pa G}{\pa x_j}(x,x_{\ep}) dS_x
= J_2 = \frac{(n-2)c_n}{n} \left. \frac{\pa \wth}{\pa x_j}(x, x_{\ep}) \right|_{x = x_{\ep}} |\S^{n-1}|.
\end{equation}

By performing similar calculations, we discover
\begin{equation}\label{eq-h3-9}
\lim_{r \to 0} \int_{\pa B^n(x_{\ep}, r)} \frac{\pa G}{\pa \nu}(x,x_{\ep}) \frac{\pa \wtg}{\pa x_j}(x,x_{\ep}) dS_x
= (n-2) c_n \left. \frac{\pa \wth}{\pa x_j}(x, x_{\ep}) \right|_{x = x_{\ep}} |\S^{n-1}|
\end{equation}
and
\begin{equation}\label{eq-h3-10}
\lim_{r \to 0} \int_{\pa B^n(x_{\ep}, r)} \( \nabla_x G (x,x_{\ep}) \cdot \nabla_x \wtg(x,x_{\ep}) \) \nu_j\, dS_x
= \frac{(n-2)c_n}{n} \left. \frac{\pa \wth}{\pa x_j}(x, x_{\ep}) \right|_{x = x_{\ep}} |\S^{n-1}|.
\end{equation}

Finally, by \eqref{eq-g-decom} and the identity
\[\int_{\pa B^n(x_{\ep},r)} \( \frac{c_n}{|x-x_{\ep}|^{n-2}}\)^{p+1} \nu_j\, dS_x = 0 \quad \text{for } j = 1,\cdots, n,\]
we have
\begin{align*}
&\ \int_{\pa B^n(x_{\ep}, r)} G^{p+1}(x,x_{\ep}) \nu_j\, dS_x
\\
& = \int_{\pa B^n(x_{\ep},r)} \left[ \( \frac{c_n}{|x-x_{\ep}|^{n-2}} - H(x,x_{\ep})\)^{p+1} - \( \frac{c_n}{|x-x_{\ep}|^{n-2}}\)^{p+1} + \frac{(p+1) H(x,x_{\ep})}{|x-x_{\ep}|^{(n-2)p}} \right] \nu_j \, dS_x
\\
&\ - (p+1) \int_{\pa B^n(x_\ep,r)} \frac{H(x,x_{\ep})}{|x-x_{\ep}|^{(n-2)p}} \nu_j\, dS_x.
\end{align*}
Also, Taylor's theorem and the condition $p < \frac{n}{n-2} \le \frac{2n-3}{n-2}$ imply
\begin{align*}
&\left| \int_{\pa B^n(x_\ep,r)} \left[ \( \frac{c_n}{|x-x_{\ep}|^{n-2}} - H(x,x_{\ep})\)^{p+1} - \( \frac{c_n}{|x-x_{\ep}|^{n-2}}\)^{p+1}
+ \frac{(p+1)H(x,x_{\ep})}{|x-x_{\ep}|^{(n-2)p}} \right] \nu_j\, dS_x \right| \\
&\le C\int_{\pa B^n(x_\ep,r)} \frac{1}{|x-x_{\ep}|^{(n-2)(p-1)}} dS_x = O(r^{(n-1)-(n-2)(p-1)}) \to 0 \quad \text{as } r \to 0
\end{align*}
and
\begin{align*}
\left| \int_{\pa B^n(x_\ep,r)} \frac{H(x,x_{\ep})}{|x-x_{\ep}|^{(n-2)p}}\, \nu_j\, dS_x\right|
&= \left| \int_{\pa B^n(x_\ep,r)} \frac{H(x,x_{\ep})-H(x_{\ep},x_{\ep})}{|x-x_{\ep}|^{(n-2)p}}\, \nu_j\, dS_x\right| \\
&= O(r^{n-(n-2)p}) \to 0 \quad \text{as } r \to 0.
\end{align*}
Hence
\begin{equation}\label{eq-h3-5}
\frac{1}{p+1} \lim_{r \to 0} \int_{\pa B^n(x_{\ep}, r)} G^{p+1}(x,x_{\ep}) \nu_j\, dS_x = 0.
\end{equation}

Plugging \eqref{eq-h3-8}-\eqref{eq-h3-5} into \eqref{eq-h3-4}, we conclude that
\begin{equation}\label{eq-h3-2}
\mci_2(2d_{\ep}) = \lim_{r \to 0} \mci_2(r) = \frac{(n-2)(n+2)c_n}{n} |\S^{n-1}| \left. \frac{\pa \wth}{\pa x_j}(x, x_{\ep}) \right|_{x = x_{\ep}}
\end{equation}
for $p \in [\frac{n-1}{n-2}, \frac{n}{n-2})$. If $p \in (\frac{2}{n-2}, \frac{n-1}{n-2})$, we can deduce \eqref{eq-h3-2} in a similar but simpler way.

Now, denoting $\nu_{x_{\ep}} = (a_1, \cdots, a_n) \in \S^{n-1}$ and employing \eqref{eq-h3-1} and \eqref{eq-h3-2}, we derive
\[\sum_{j=1}^n a_j \mcl_j^{\ep} = \frac{(n-2)(n+2)c_n}{n} |\S^{n-1}| \lambda_{\ep}^{-\frac{n(p+1)}{q_0+1}}
\left. \frac{\pa \wth}{\pa \nu_{x_{\ep}}}(x, x_{\ep}) \right|_{x = x_{\ep}}
+ o\(d_{\ep}^{1-(n-2)p} \lambda_{\ep}^{- \frac{n(p+1)}{q_0+1}}\).\]
Therefore, by applying \eqref{eq-h2} and \eqref{eq-cr-hy} to the above inequality, we obtain
\begin{equation}\label{eq-h3-3}
\sum_{j=1}^n a_j \mcl_j^{\ep} \ge C \lambda_{\ep}^{-\frac{n(p+1)}{q_0+1}} \lambda_{\ep}^{-1 + (n-2)p} \Lambda_{\ep}^{1-(n-2)p} =C \lambda_{\ep} \Lambda_{\ep}^{1-(n-2)p}.
\end{equation}

\medskip \noindent \textbf{Estimate of $\mcr_j^{\ep}$.}
Using \eqref{eq-tuv} and Proposition \ref{prop-dec}, we find
\begin{equation}\label{eq-h3-7}
\begin{aligned}
|\mcr_j^{\ep}| = \frac{1}{q_{\ep}+1} \left|\int_{\pa B^n(x_{\ep}, 2d_{\ep})} u_{\ep}^{q_{\ep}+1} \nu_j\, dS_x\right|
&\le C \lambda_{\ep} \int_{\pa B^n(x_{\ep}, 2d_{\ep})} U_{\ep}^{q_{\ep}+1}(x)dS_x \\
&= \lambda_{\ep} \Lambda_{\ep}^{(n-1)- ((n-2)p-2)(q_{\ep}+1)} \\
&= \lambda_{\ep} \Lambda_{\ep}^{-np-1+\ep(p+1)(q_{\ep}+1)}
\end{aligned}
\end{equation}
where \eqref{eq-pq-e} was used for the last equality.

\medskip
From \eqref{eq-h3-3} and \eqref{eq-h3-7}, we observe
\[C\lambda_{\ep} \Lambda_{\ep}^{1-(n-2)p} \le \sum_{j=1}^n a_j \mcl_j^{\ep} \le \sum_{j=1}^n |a_j| |\mcr_j^{\ep}| \le C \lambda_{\ep} \Lambda_{\ep}^{-np-1+\ep(p+1)(q_{\ep}+1)}.\]
Because $\Lambda_{\ep} \to \infty$ as $\ep \to 0$, the above estimate implies that $2p +2 \le 0$, which is nonsense in that $p > \frac{2}{n-2}$.
Therefore $d_{\ep} \nrightarrow 0$ as $\ep \to 0$. In other words, the maximum point $x_{\ep}$ of $u_{\ep}$ should be away from $\pa \Omega$, and the proof is done.
\end{proof}

\appendix
\section{Interior Regularity of $\wth$ (Proof of Lemma \ref{lem-wth-reg})}\label{sec-app-1}
In this appendix, we prove $C^1$ interior regularity of the regular part $\wth$ of the function $\wtg$ defined by \eqref{eq-h2-2}.
For this aim, we need an elementary lemma which we state now.
\begin{lem}\label{lem-cal-1}
\textnormal{(1)} If $p \in (0,1]$, it holds that $0 \le a^p - (a-b)^p \le b^p$ for $0 \le b \le a$.

\medskip \noindent \textnormal{(2)} If $p \in [1,2]$, it holds that $-\min\{(p-1) a^{p-2}b^2, b^p\} \le a^p - (a-b)^p - pa^{p-1}b \le 0$ for $0 \le b \le a$.

\medskip \noindent \textnormal{(3)} If $p \ge 2$, it holds that $-\frac{p(p-1)}{2} a^{p-2}b^2 \le a^p - (a-b)^p - pa^{p-1}b \le 0$ for $0 \le b \le a$.
\end{lem}

\begin{proof}[Proof of Lemma \ref{lem-wth-reg}]
Fix $y \in \Omega$. By virtue of elliptic regularity theory, it is enough to check that the function $x \in \Omega \mapsto -\Delta_x \wth(x,y)$ is contained in $L^{n+\eta}_{\text{loc}}(\Omega)$ for some $\eta > 0$.
By \eqref{eq-h2-1} and \eqref{eq-h2-2}, we have
\begin{equation}\label{eq-tilh}
-\Delta_x \wth(x,y) = \begin{cases}
\dfrac{c_n^p}{|x-y|^{(n-2)p}} - G^p (x,y) &\text{for } p \in (\frac{2}{n-2}, \frac{n-1}{n-2}), \\
\begin{aligned}
& \frac{c_n^p}{|x-y|^{(n-2)p}} - G^p (x,y) - \frac{pc_n^{p-1} H(x,y)}{|x-y|^{(n-2)(p-1)}}
\\
&\quad + \frac{2pc_n^{p-1} \nabla_x H(x,y)}{(n-2)p-2(n-1)} \cdot \frac{x-y}{|x-y|^{(n-2)(p-1)}}
\end{aligned}
&\text{for } p \in [\frac{n-1}{n-2}, \frac{n}{n-2})
\end{cases}
\end{equation}
whenever $x \in \Omega$ and $x \ne y$.

\medskip \noindent \textbf{Case 1}.
Assume that $p \in (\frac{2}{n-2}, \frac{n-1}{n-2})$.
Plugging \eqref{eq-g-decom} into \eqref{eq-tilh} and using Lemma \ref{lem-cal-1} (1) and (2), we find a constant $C > 0$ (dependent on $y \in \Omega$) such that
\[\left| - \Delta_x \wth(x,y) \right| = \left| \frac{c_n^p}{|x-y|^{(n-2)p}} - \(\frac{c_n}{|x-y|^{n-2}} - H(x,y)\)^p \right| \le C \max \left\{1, \frac{1}{|x-y|^{(n-2)(p-1)}}\right\}\]
for all $x \in \Omega \setminus \{y\}$. Since $(n-2)(p-1) < 1$, there exists a number $\eta > 0$ such that $-\Delta_x \wth(x,y) \in L_{\text{loc}}^{n+\eta}(\Omega)$.

\medskip \noindent \textbf{Case 2}.
Assume that $p \in [\frac{n-1}{n-2}, \frac{n}{n-2})$. By Lemma \ref{lem-cal-1} (2) and (3), there exists a constant $C > 0$ (dependent on $y \in \Omega$) such that
\[\left| \frac{c_n^p}{|x-y|^{(n-2)p}} - \(\frac{c_n}{|x-y|^{n-2}} - H(x,y)\)^p - \frac{pc_n^{p-1} H(x,y)}{|x-y|^{(n-2)(p-1)}} \right|
\le \begin{cases}
C &\text{if } n \ge 4,\\
\dfrac{C}{|x-y|^{p-2}} &\text{if } n = 3
\end{cases}\]
for all $x \in \Omega \setminus \{y\}$. On the other hand,
\[\left| {\nabla_x H(x,y)} \cdot \frac{x-y}{|x-y|^{(n-2)(p-1)}} \right| \le \frac{C|x-y|}{|x-y|^{(n-2)(p-1)}} = \frac{C}{|x-y|^{(n-2)p-(n-1)}}\]
for all $x \in \Omega \setminus \{y\}$.
Since $(n-2)p-(n-1) < 1$, there is a number $\eta > 0$ such that $-\Delta_x \wth(x,y) \in L_{\text{loc}}^{n+\eta}(\Omega)$. The lemma is proved.
\end{proof}

\section{Boundary Behavior of $\wth$ (Proof of Proposition \ref{prop-h2})}\label{sec-app-2}
This section is devoted to the proof of Proposition \ref{prop-h2}, which requires a delicate quantitative analysis of the Green's function $G$.
It is decomposed into two parts: In Subsection \ref{subsec-red},
we show that checking \eqref{eq-h2} can be reduced to verifying some strict inequalities involving several integrals over the half-space $\R^n_+$ and its boundary $\R^{n-1} = \pa \R^n_+$.
In Subsection \ref{subsec-ver}, we prove the validity of such inequalities.

\subsection{Reduction of \eqref{eq-h2}} \label{subsec-red}
Firstly, we derive the following result.
\begin{lem}\label{lem-h2-r}
Denote $x = (\bx, x_n) \in \R^n_+$ and $e_n = (0, \cdots, 0, 1) \in \R^n_+$.
Let also $c_n > 0$, $\gamma_1 > 0$ and $\gamma_2 < 0$ be the constants introduced in Subsection \ref{subsec-Green} and \eqref{eq-gamma}, respectively.

\medskip \noindent \textnormal{(1)} If $p \in (\frac{2}{n-2}, \frac{n-1}{n-2})$ and it holds that
\begin{equation}\label{eq-as-0}
\begin{aligned}
&c_n^p \int_{\R^n_+} \left[ \frac{1-x_n}{|x-e_n|^n} - \frac{1+x_n}{|x+e_n|^n}\right]
\left[ \frac{1}{|x-e_n|^{(n-2)p}} - \(\frac{1}{|x-e_n|^{n-2}} - \frac{1}{|x+e_n|^{n-2}}\)^p \right] dx
\\
&> 2 \gamma_1 \int_{\R^{n-1}} \left[\frac{1}{|(\bx-e_n)|^{(n-2)(p+1)}} - \frac{n}{|(\bx-e_n)|^{(n-2)p+n}} \right] d\bx,
\end{aligned}
\end{equation}
then there exist small numbers $C > 0$ and $\delta > 0$ such that \eqref{eq-h2} holds for all $x \in \Omega$ with $\mfd(x) = \textnormal{dist}(x,\pa \Omega) < \delta$.

\medskip \noindent \textnormal{(2)} If $p \in [\frac{n-1}{n-2}, \frac{n}{n-2})$ and it holds that
\begin{align}
&c_n^p \int_{\R^n_+} \left[ \frac{(1-x_n)}{|x-e_n|^{n}} - \frac{(1+x_n)}{|x+e_n|^n}\right]
\left[ \frac{1}{|x-e_n|^{(n-2)p}} - \(\frac{1}{|x-e_n|^{n-2}} - \frac{1}{|x+e_n|^{n-2}}\)^p \right. \nonumber
\\
&\quad \left. - \frac{p}{|x-e_n|^{(n-2)(p-1)}} \frac{1}{|x+e_n|^{n-2}} - \frac{2(n-2)p}{(n-2)p-2(n-1)} \frac{x-e_n}{|x-e_n|^{(n-2)(p-1)}} \cdot \frac{x+e_n}{|x+e_n|^n} \right] dx \nonumber
\\
&> 2(\gamma_1 - c_n \gamma_2) \int_{\R^{n-1}} \left[\frac{1}{|(\bx-e_n)|^{(n-2)(p+1)}} - \frac{n}{|(\bx-e_n)|^{(n-2)p+n}} \right] d\bx, \label{eq-as-1}
\end{align}
then there exist small numbers $C > 0$ and $\delta > 0$ such that \eqref{eq-h2} holds for all $x \in \Omega$ with $\mfd(x) < \delta$.
\end{lem}

To obtain this lemma, we need a preliminary result, that is, Lemma \ref{lem-hd}.

Let $x_0 \in \Omega$ be a fixed point sufficiently close to $\pa \Omega$.
By virtue of the translational and rotational invariance of the problem,
we can assume that $0 \in \pa \Omega$, $x_0 = (0, \cdots, 0, \kappa) = \kappa e_n$ and $x_0^* = x_0 + 2\kappa\nu_{x_0} = - \kappa e_n$
where $\kappa = \mfd(x_0) > 0$ and $\nu_{x_0} \in \S^{n-1}$ is the unique vector such that $x_0 + \kappa\nu_{x_0} \in \pa \Omega$.
Then \eqref{eq-h2} can be written as
\begin{equation}\label{eq-h2-0}
\left. \frac{\pa}{\pa x_n} \wth(x,\kappa e_n) \right|_{x=\kappa e_n} \ge C \kappa^{1-(n-2)p} \quad \text{for } \kappa > 0 \text{ small}.
\end{equation}
Moreover, since $\pa \Omega$ is smooth, it can be locally parameterized by a smooth function $f: \R^{n-1} \to \R^n$ satisfying $f(0) = \nabla f(0) = 0$
so that $\pa \Omega \cap B^n (0,r_1) \subset \{ (x,f(x)): x \in B^{n-1}(0,r_2)\}$ for small fixed values $r_1, r_2 > 0$.

To show \eqref{eq-h2-0}, we look at the behavior of the left-hand side as $\kappa \to 0$, considering the following rescaling:
Letting $\Omega_{\kappa}$ be a rescaled domain $\frac{1}{\kappa} \Omega$, we also define $\mch_k: \Omega_k \to \R$ for each small $\kappa > 0$ by
\begin{equation}\label{eq-g-0}
\mch_{\kappa}(z) = \kappa^{(n-2)p-2} \wth(\kappa z, \kappa e_n) \quad \text{for } z \in \Omega_{\kappa}.
\end{equation}
Then it satisfies that
\begin{equation}\label{eq-g-1}
\mch_{\kappa}(z) = \begin{cases}
\dfrac{\gamma_1}{|z-e_n|^{(n-2)p-2}} &\text{if } p \in (\frac{2}{n-2}, \frac{n-1}{n-2}), \\
\dfrac{\gamma_1 - c_n\gamma_2}{|z-e_n|^{(n-2)p-2}} &\text{if } p \in [\frac{n-1}{n-2}, \frac{n}{n-2})
\end{cases}
\end{equation}
for $z \in \pa \Omega_{\kappa}$.
Since we are dealing with the situation when points are near the boundary,
we expect that $\Omega_{\kappa}$ converges to the half-space $\R^n_+$ as $\kappa \to 0$.
Hence it is natural to introduce the function $\wth_0: \R^n_+ \times \R^n_+ \to \R$ satisfying
\begin{equation}\label{eq-h0-1}
-\Delta_z \wth_0 (z,y)
= \begin{cases}
\dfrac{c_n^p}{|z-y|^{(n-2)p}} - \( \dfrac{c_n}{|z-y|^{n-2}} - \dfrac{c_n}{|z-\ty|^{n-2}} \)^p &\text{if } p \in (\frac{2}{n-2}, \frac{n-1}{n-2}), \\
\begin{aligned}
& \dfrac{c_n^p}{|z-y|^{(n-2)p}} - \( \dfrac{c_n}{|z-y|^{n-2}} - \dfrac{c_n}{|z-\ty|^{n-2}}\)^p \\
&\quad - \frac{pc_n^p}{|z-y|^{(n-2)(p-1)}} \frac{1}{|z-\ty|^{n-2}}
\\
&\quad - \frac{2(n-2)p c_n^p}{(n-2)p-2(n-1)} \frac{z-y}{|z-y|^{(n-2)(p-1)}} \cdot \frac{z-\ty}{|z-\ty|^n}
\end{aligned}
&\text{if } p \in [\frac{n-1}{n-2}, \frac{n}{n-2})
\end{cases}
\end{equation}
for $z \ne y \in \R^n_+$ and
\begin{equation}\label{eq-h0-2}
\wth_0(z,y) = \begin{cases}
\dfrac{\gamma_1}{|z-y|^{(n-2)p-2}} &\text{if } p \in (\frac{2}{n-2}, \frac{n-1}{n-2}), \\
\dfrac{\gamma_1 - c_n\gamma_2}{|z-y|^{(n-2)p-2}} &\text{if } p \in [\frac{n-1}{n-2}, \frac{n}{n-2})
\end{cases}
\end{equation}
for $z \in \R^{n-1}$ and $y \in \R^n_+$ where $\ty := (y_1, \cdots, y_{n-1}, -y_n)$ is the reflection of $y = (y_1, \cdots, y_{n-1}, y_n)$ with respect to the space $\R^{n-1}$.
Indeed, if we define $\mch_0: \R^n_+ \to \R$ as
\[\mch_0(z) = \wth_0(z,e_n) \quad \text{for } z \in \R^n_+,\]
then $\mch_{\kappa} \to \mch_0$ as $\kappa \to 0$ in local $C^1$-sense, as we shall see in
\begin{lem}\label{lem-hd}
The function $\mch_{\kappa}$ converges to $\mch_0$ in $C^1(B^n(e_n, \frac{1}{4}))$ as $\kappa \to 0$.
\end{lem}
\noindent In its proof, we use the following elementary lemma which is a slight variant of Lemma \ref{lem-cal-1}.
\begin{lem}\label{lem-cal-2}
Fix $p \in \R$ and $\eta \in (0,1)$.
For any pair $(a,b)$ such that $0 \le |b| < \eta a$, we have
\[(a-b)^p = a^p - pa^{p-1}b + O(a^{p-2}b^2) \quad \text{and} \quad
(a-b)^{p-1} = a^{p-1} + O(a^{p-2}b).\]
\end{lem}

\begin{proof}[Proof of Lemma \ref{lem-hd}]
Suppose that $p \in [\frac{n-1}{n-2}, \frac{n}{n-2})$.
Let $\mce_{\kappa}: B^n(e_n, \frac{1}{4}) \to \R$ be the function defined by $\mce_{\kappa} = \mch_0 - \mch_{\kappa}$,
and $G_{\kappa}$ and $G_0$ be the Green's functions of the Dirichlet Laplacians $-\Delta$ in $\Omega_{\kappa}$ and $\R^n_+$, respectively. Then
\[\begin{cases}
G_{\kappa}(z, w) = \kappa^{n-2} G(\kappa z, \kappa w) &\text{for } z \ne w \in \Omega_{\kappa},\\
G_0(z, w) = \dfrac{c_n}{|z-w|^{n-2}} - \dfrac{c_n}{|z-\tilde{w}|^{n-2}} &\text{for } z \ne w \in \R^n_+
\end{cases}\]
where $\tilde{w}$ is the reflection of $w$ with respect to $\R^{n-1}$. By Green's representation formula, we have
\begin{align*}
\mce_{\kappa}(z) &= \left[ \int_{\R^n_+} G_0(z,w) (- \Delta_w \wth_0)(w,e_n) dw
- \kappa^{(n-2)(p+1)} \int_{\Omega_{\kappa}} G(\kappa z, \kappa w) (- \Delta_x \wth)(\kappa w, \kappa e_n) dw \right] \\
&\ + \left[ \int_{\R^{n-1}} \frac{2(n-2)c_n z_n}{|z-\bw|^n} \mch_0(\bw) d\bw
+ \kappa^{n-1} \int_{\pa \Omega_{\kappa}} \frac{\pa G}{\pa \nu_x}(\kappa z,\kappa w) \mch_{\kappa}(w) dS_w \right] \\
&=: \mce_{\kappa 1}(z) + \mce_{\kappa 2}(z)
\end{align*}
for $z \in B^n(e_n, \frac{1}{4})$ where $w = (\bw, w_n) \in \R^n$, $x = \kappa w \in \R^n$ and $\nu_x$ is the outward unit normal vector on $\pa \Omega_{\kappa}$.
In order to deduce the lemma, it suffices to show that $\mce_{\kappa 1},\, \mce_{\kappa 2} \to 0$ in $C^1(B^n(e_n, \frac{1}{4}))$ as $\kappa \to 0$.

\medskip \noindent \textbf{Estimate of $\mce_{\kappa 1}$.}
 Pick any $z \in B^n(e_n, \frac{1}{4})$ and rewrite the function $\mce_{\kappa 1}(z)$ as
\begin{align*}
&\ \mce_{\kappa 1}(z) \\
&= \mce_{\kappa 11}(z) + \mce_{\kappa 12}(z) + \mce_{\kappa 13}(z) + \mce_{\kappa 14}(z) + \mce_{\kappa 15}(z) \\
&:= \(\int_{B^n(e_n, \frac{1}{2})} + \int_{(\Omega_{\kappa} \cap \R^n_+) \setminus B^n(e_n, \frac{1}{2})}\) G_0(z,w) \left[ (-\Delta_w \wth_0)(w,e_n) - \kappa^{(n-2)p} (-\Delta_x \wth)(\kappa w, \kappa e_n) \right] dw \\
&\ + \int_{\Omega_{\kappa} \cap \R^n_+} \left[G_0(z,w) - \kappa^{n-2} G(\kappa z, \kappa w)\right] \kappa^{(n-2)p} (- \Delta_x \wth)(\kappa w, \kappa e_n) dw \\
&\ + \int_{\R^n_+ \setminus \Omega_{\kappa}} G_0(z,w) (- \Delta_w \wth_0)(w,e_n) dw
- \kappa^{(n-2)(p+1)} \int_{\Omega_{\kappa} \setminus \R^n_+} G(\kappa z, \kappa w) (- \Delta_x \wth)(\kappa w, \kappa e_n) dw.
\end{align*}
Here, the integrands of $\mce_{\kappa 11}$ and $\mce_{\kappa 12}$ are both $G_0(z,\cdot) [(-\Delta \wth_0)(\cdot,e_n) - \kappa^{(n-2)p} (-\Delta_x \wth)(\kappa \cdot, \kappa e_n)]$,
but they have different domains $B^n(e_n, \frac{1}{2})$ and $(\Omega_{\kappa} \cap \R^n_+) \setminus B^n(e_n, \frac{1}{2})$, respectively.

Firstly, let us estimate the term $\mce_{\kappa 11}$ and its derivative.
By \eqref{eq-h0-1} and \eqref{eq-tilh}, we have
\[(-\Delta_w \wth_0)(w,e_n) - \kappa^{(n-2)p} (-\Delta_x \wth)(\kappa w, \kappa e_n) = c_n^p \(\mcf_{\kappa 1}(w) + \mcf_{\kappa 2}(w)\) \quad \text{in } \Omega_{\kappa} \cap \R^n_+\]
where
\begin{equation}\label{eq-hd-11}
\begin{aligned}
\mcf_{\kappa 1}(w) &:= c_n^{-p} {\kappa}^{(n-2)p} G^p({\kappa}w, {\kappa} e_n) - \( \frac{1}{|w-e_n|^{n-2}} - \frac{1}{|w+e_n|^{n-2}}\)^p \\
&\ - \frac{p}{|w-e_n|^{(n-2)(p-1)}} \(\frac{1}{|w+e_n|^{n-2}} - c_n^{-1} {\kappa}^{n-2} H({\kappa}w, {\kappa} e_n) \)
\end{aligned}
\end{equation}
and
\begin{equation}\label{eq-hd-12}
\begin{aligned}
\mcf_{\kappa 2}(w) &:= - \frac{2p}{(n-2)p-2(n-1)} \frac{w-e_n}{|w-e_n|^{(n-2)(p-1)}} \\
&\quad \times \((n-2) \frac{w+e_n}{|w+e_n|^n} + c_n^{-1} {\kappa}^{n-1} \nabla_x H({\kappa}w, {\kappa} e_n)\)
\end{aligned}
\end{equation}
in $\Omega_{\kappa} \cap \R^n_+$.
On the other hand, if we define functions $T_{\kappa 1}$ and $T_{\kappa 2}$ in $\Omega_{\kappa}$ by
\begin{equation}\label{eq-hd-2}
\begin{cases}
{\kappa}^{n-2} H(\kappa w, \kappa e_n) = c_n \(\dfrac{1}{|w+e_n|^{n-2}} + T_{\kappa 1}(w)\), \\
{\kappa}^{n-1} \nabla_x H({\kappa}w, {\kappa} e_n) = c_n \(-(n-2) \dfrac{w+e_n}{|w+e_n|^n} + T_{\kappa 2}(w)\),
\end{cases}
\end{equation}
then \eqref{eq-h1-1} and \eqref{eq-h1-3} give
\begin{equation}\label{eq-hd-5}
T_{\kappa 1}(w) = O\(\frac{\kappa}{|w+e_n|^{n-2}}\) \quad \text{in } \Omega_{\kappa}
\end{equation}
and
\begin{equation}\label{eq-hd-51}
T_{\kappa 2}(w) = O\(\frac{\kappa}{|w+e_n|^{n-2}}\) \quad \text{on } \Omega_{\kappa}^{\delta_0} := \{w \in \Omega_{\kappa}: \mfd(\kappa w) \ge \kappa \delta_0\}
\end{equation}
for any fixed small number $\delta_0 \in (0, \frac{1}{2})$.
Putting \eqref{eq-hd-2} into \eqref{eq-hd-11} and \eqref{eq-hd-12}, we find that
\begin{multline}\label{eq-hd-3}
\mcf_{\kappa 1}(w) = \(\frac{1}{|w-e_n|^{n-2}} - \frac{1}{|w+e_n|^{n-2}} - T_{\kappa 1}(w)\)^p \\
- \(\frac{1}{|w-e_n|^{n-2}} - \frac{1}{|w+e_n|^{n-2}}\)^p + \frac{p}{|w-e_n|^{(n-2)(p-1)}} T_{\kappa 1}(w)
\end{multline}
and
\begin{equation}\label{eq-hd-4}
\mcf_{\kappa 2}(w) = - \frac{2p}{(n-2)p-2(n-1)} \frac{w-e_n}{|w-e_n|^{(n-2)(p-1)}} T_{\kappa 2}(w)
\end{equation}
in $\Omega_{\kappa} \cap \R^n_+$. An application of \eqref{eq-hd-5}, \eqref{eq-hd-51} and Lemma \ref{lem-cal-2} shows
\[|\mcf_{\kappa 1}(w)| = O\(\kappa |w-e_n|^{(n-2)(2-p)}\)
\quad \text{and} \quad
|\mcf_{\kappa 2}(w)| = O\(\kappa |w-e_n|^{1-(n-2)(p-1)}\)\]
in $\Omega_{\kappa}^{\delta_0} \cap \R^n_+$.
For $\kappa > 0$ small enough, it holds that $B^n(e_n, \frac{1}{2}) \subset \Omega_{\kappa}^{\delta_0} \cap \R^n_+$.
Therefore, for $k = 0$ or $1$,
\[\left| \int_{B^n(e_n, \frac{1}{2})} \nabla_z^k G_0(z,w) \mcf_{\kappa 1}(w) dw \right| \le C \kappa \int_{B^n(e_n, \frac{1}{2})} \frac{|w-e_n|^{(n-2)(2-p)}}{|z-w|^{n-2+k}} dw.\]
Since $n+(n-2)(2-p)-(n-2+k) > 0$ holds for any $p < \frac{n}{n-2} \le \frac{2n-3}{n-2}$, the right-hand side goes to $0$ as $\kappa \to 0$ uniformly in $B^n(e_n, \frac{1}{2})$.
Furthermore, for $k = 0$ or $1$,
\[\left| \int_{B^n(e_n, \frac{1}{2})} \nabla_z^k G_0(z,w) \mcf_{\kappa 2}(w) dw \right| \le C\kappa \int_{B^n(e_n, \frac{1}{2})} \frac{1}{|z-w|^{n-2+k}} \frac{|w-e_n|}{|w-e_n|^{(n-2)(p-1)}} dw,\]
which also goes to $0$ as $\kappa \to 0$ uniformly in $B^n(e_n, \frac{1}{2})$, because $n+1-(n-2+k)-(n-2)(p-1) > 0$ holds for any $p < \frac{n}{n-2}$.
Consequently, $\mce_{\kappa 11}$ and its derivative tend to 0 as $\kappa \to 0$ uniformly in $B^n(e_n, \frac{1}{4})$.

Secondly, we estimate the term $\mce_{\kappa 12}$ and its derivative.
Because $\pa \Omega$ is smooth and compact, there is a constant $C > 1$ (independent of the choice of the point $x_0 \in \Omega$) such that
\begin{equation}\label{eq-hd-90}
\frac{1}{C} |w+e_n| \le |w-e_n| \le C|w+e_n| \quad \text{for all } w \in \Omega_{\kappa} \setminus B^n\(e_n, \frac{1}{2}\).
\end{equation}
Hence it follows from \eqref{eq-hd-3} and \eqref{eq-hd-5} that
\begin{equation}\label{eq-hd-91}
|\mcf_{\kappa 1}(w)| = O\(\dfrac{T_{\kappa 1}(w)}{|w+e_n|^{(n-2)(p-1)}}\) = O\(\dfrac{\kappa}{|w+e_n|^{(n-2)p}}\) \quad \text{in } (\Omega_{\kappa} \cap \R^n_+) \setminus B^n\(e_n, \dfrac{1}{2}\),
\end{equation}
and from \eqref{eq-hd-4} and \eqref{eq-hd-51} that
\begin{equation}\label{eq-hd-92}
\begin{aligned}
|\mcf_{\kappa 2}(w)| &= O\(\dfrac{|w-e_n|}{|w-e_n|^{(n-2)(p-1)}} T_{\kappa 2}(w)\) \\
&= O\(\dfrac{\kappa |w+e_n|}{|w+e_n|^{(n-2)p}}\)
\end{aligned}
\quad \text{in } (\Omega_{\kappa}^{\delta_0} \cap \R^n_+) \setminus B^n\(e_n, \dfrac{1}{2}\).
\end{equation}
In order to estimate $\mcf_{\kappa 2}$ in $\Omega_{\kappa} \setminus \overline{\Omega_{\kappa}^{\delta_0}}$, we take a small number $\vep_1 > 0$ (independent of the choice of the point $x_0 \in \Omega$)
such that
\begin{equation}\label{eq-hd-9}
\Omega \cap Q^n(\vep_1) = \left\{ (\bx, x_n) \in Q^n(\vep_1): x_n > f(\bx) \right\}
\end{equation}
where $Q^n(\vep_1) := B^{n-1}(0,\vep_1) \times (-\vep_1,\vep_1)$ and $f: B^{n-1}(0,\vep_1) \to (-\vep_1,\vep_1)$ is a smooth function satisfying $|f(\bx)| \le C|\bx|^2$ for all $\bx \in B^{n-1}(0,\vep_1)$.
Then, as computed at the end of the proof, we have
\begin{equation}\label{eq-hd-93}
|\mcf_{\kappa 2}(w)| = O\({\kappa (1+|w|^2) \over |w+e_n|^{(n-2)p+1}}\) \quad \text{in } \(\Omega_{\kappa} \setminus \overline{\Omega_{\kappa}^{\delta_0}}\) \cap Q^n\(\frac{\vep_1}{\kappa}\)
\end{equation}
and
\begin{equation}\label{eq-hd-94}
|\mcf_{\kappa 2}(w)| = O(\kappa^{(n-2)p}) \quad \text{in } \(\Omega_{\kappa} \setminus \overline{\Omega_{\kappa}^{\delta_0}}\) \setminus Q^n\(\frac{\vep_1}{\kappa}\).
\end{equation}
Therefore \eqref{eq-hd-91} implies that
\[\left| \int_{(\Omega_{\kappa} \cap \R^n_+) \setminus B^n(e_n, \frac{1}{2})} \nabla_z^k G_0(z,w) \mcf_{\kappa 1}(w) dw \right|
\le C\kappa \int_{\R^n_+ \setminus B^n(e_n, \frac{1}{2})} \frac{1}{|z-w|^{n-2+k}} \frac{1}{|w+e_n|^{(n-2)p}} dw\]
for $k = 0$ or $1$, and the right-hand side goes to 0 as ${\kappa} \to 0$ since $n-(n-2+k)-(n-2)p < 0$ for any $p > \frac{2}{n-2}$.
Moreover, we infer from \eqref{eq-hd-92}, \eqref{eq-hd-93} and \eqref{eq-hd-94} that
\begin{multline*}
\left| \int_{(\Omega_{\kappa} \cap \R^n_+) \setminus B^n(e_n, \frac{1}{2})} \nabla_z^k G_0(z,w) \mcf_{\kappa 2}(w) dw \right| \\
\le C\kappa \int_{\R^n_+ \setminus B^n(e_n, \frac{1}{2})} \frac{1}{|z-w|^{n-2+k}} \frac{1+|w|^2}{|w+e_n|^{(n-2)p+1}} dw + C \kappa^{-n+(n-2+k)+(n-2)p}
\end{multline*}
for $k = 0$ or $1$, and the right-hand side again goes to 0 as ${\kappa} \to 0$ since $n+2-(n-2+k)-((n-2)p+1) < 0$ for any $n \ge 5$ and $p \ge \frac{n-1}{n-2} > \frac{3}{n-2}$.
Consequently, $\mce_{\kappa 12}$ and its derivative tend to 0 as $\kappa \to 0$ uniformly in $B^n(e_n, \frac{1}{4})$.

Thirdly, we compute $\mce_{\kappa 13}$ and its derivative. By virtue of \eqref{eq-h1-1}, we have that
\[G_0(z,w) - \kappa^{n-2} G(\kappa z, \kappa w) = c_n \left[{1 \over |w-\tilde{z}|^{n-2}} - {1 + O(\kappa) \over |w-\kappa^{-1}(\kappa z)^*|^{n-2}} \right]\]
for $z \in B^n(e_n, \frac{1}{4})$ and $w \in \Omega_{\kappa} \cap \R^n_+$
where $\tilde{z}$ is the reflection of $z$ with respect to $\R^{n-1}$, and that
\[\kappa^{-1}(\kappa z)^* \to \tilde{z} \quad \text{as } \kappa \to 0 \text{ uniformly in } B^n(e_n, \frac{1}{4}).\]
Also, inspecting the proof of Lemma \ref{lem-wth-reg} and employing \eqref{eq-g1-9}, we see that
\[\left| \kappa^{(n-2)p} (- \Delta_x \wth)(\kappa w, \kappa e_n) \right|
\le C {1 \over |w+e_n|^{n-1}}{1 \over |w-e_n|^{(n-2)p-(n-1)}}
+ \begin{cases}
\dfrac{C}{|w+e_n|^{(n-2)p}} &\text{if } n \ge 4,\\
0 &\text{if } n = 3.
\end{cases}\]
Hence it follows that $\mce_{\kappa 13} \to 0$ as $\kappa \to 0$ uniformly in $B^n(e_n, \frac{1}{4})$.
A similar argument with \eqref{eq-h1-2} shows that the derivative of $\mce_{\kappa 13}$ tends to 0 uniformly in $B^n(e_n, \frac{1}{4})$ as well.

For the estimate of $\mce_{\kappa 14}$, $\mce_{\kappa 15}$ and their derivatives,
we use the fact that there exists a small number $\vep_2 > 0$ (independent of the choice of the point $x_0 \in \Omega$) such that
\[B^n(\vep_2 e_n, \vep_2) \subset \Omega \quad \text{and} \quad B^n(-\vep_2 e_n, \vep_2) \subset \R^n \setminus \Omega\]
and so
\[\R^n_+ \setminus \Omega_{\kappa} \subset \R^n_+ \setminus B^n\({\vep_2 e_n \over \kappa}, {\vep_2 \over \kappa}\)
\quad \text{and} \quad
\Omega_{\kappa} \setminus \R^n_+ \subset \R^n_- \setminus B^n\(-{\vep_2 e_n \over \kappa}, {\vep_2 \over \kappa}\).\]
Then, for $k = 0$ or $1$,
\begin{align*}
\left| \nabla^k_z \mce_{\kappa 14}(z) \right| &\le C \int_{\R^n_+ \setminus B^n\({\vep_2 e_n \over \kappa}, {\vep_2 \over \kappa}\)} {1 \over |z-w|^{n-2+k}} {1 \over |w-e_n|^{(n-2)p}} dw \\
&\le C \( \int_{\R^n_+ \setminus (B^{n-1}(0, {\vep_2 \over \kappa}) \times (0, {\vep_2 \over \kappa}))}
+ \int_{\left\{ 0 < w_n < {\vep_2 \over \kappa} - \sqrt{({\vep_2 \over \kappa})^2 - |\bw|^2} \right\}} \) {1 \over |w-e_n|^{(n-2)(p+1)+k}} dw \\
&\le C \({\kappa \over \vep_2}\)^{(n-2)p-2+k} \to 0,
\end{align*}
and similarly, $\left| \nabla^k_z \mce_{\kappa 15}(z) \right| \to 0$ as $\kappa \to 0$ uniformly in $B^n(e_n, \frac{1}{4})$.

In conclusion, $\mce_{\kappa 1} \to 0$ in $C^1(B^n(e_n,\frac{1}{4}))$ as $\kappa \to 0$.

\medskip \noindent \textbf{Estimate of $\mce_{\kappa 2}$.}
Fix $z \in B^n(e_n, \frac{1}{4})$.
In this step, we will separately deal with the cases when $w \in \pa \Omega_{\kappa}$ is close to the origin and when it is not.

For the former case, we compute employing \eqref{eq-g-1}, \eqref{eq-h0-2}, \eqref{eq-h1-2} and the mean value theorem that \begin{align}
&\ \int_{B^{n-1}(0, {\vep_1 \over \kappa})} \left[\frac{2(n-2)c_n z_n}{|z-\bw|^n} \mch_0(\bw) + \kappa^{n-1} \mch_{\kappa}(w) \right. \nonumber\\
&\hspace{125pt} \left. \times \left\{\nabla_{\bx} f(\kappa \bw) \cdot \nabla_{\bx} G(\kappa z, \kappa w)
- {\pa G \over \pa x_n}(\kappa z, \kappa w) \right\} \sqrt{1+|(\nabla f)(\kappa \bw)|^2} \right] d\bw \nonumber \\
&= \int_{B^{n-1}(0, {\vep_1 \over \kappa})} \left[\frac{2(n-2)c_n z_n}{|z-\bw|^n}
- \kappa^{n-1} {\pa G \over \pa x_n}(\kappa z, \kappa w)\right] \dfrac{\gamma_1 - c_n\gamma_2}{|\bw-e_n|^{(n-2)p-2}} d\bw \nonumber \\
&\ + (\gamma_1 - c_n\gamma_2) \int_{B^{n-1}(0, {\vep_1 \over \kappa})} \kappa^{n-1} {\pa G \over \pa x_n}(\kappa z, \kappa w)
\(\dfrac{1}{|\bw-e_n|^{(n-2)p-2}} - \dfrac{1}{|w-e_n|^{(n-2)p-2}}\) d\bw + O(\kappa^{\eta}) \nonumber \\
&= O(\kappa^{\eta}) \label{eq-hd-6}
\end{align}
where $\vep_1 > 0$ and $f: B^{n-1}(0,\vep_1) \to (-\vep_1, \vep_1)$ were defined in the sentence containing \eqref{eq-hd-9},
$w = (\bw, \kappa^{-1} f(\kappa \bw)) \in \pa \Omega_{\kappa}$, $x = (\bx, x_n) = \kappa w \in \R^n$ and $\eta > 0$ is a sufficiently small number.

For the latter case, we observe from \eqref{eq-g-1}, \eqref{eq-h0-2} and \eqref{eq-g1-9} that
\begin{equation}\label{eq-hd-7}
\begin{aligned}
\left| \int_{\R^{n-1} \setminus B^{n-1}(0, {\vep_1 \over \kappa})} \frac{2(n-2)c_n z_n}{|z-\bw|^n} \mch_0(\bw) d\bw \right|
&\le C \int_{\R^{n-1} \setminus B^{n-1}(0, {\vep_1 \over \kappa})} \frac{1}{|z-\bw|^{(n-2)(p+1)}} d\bw \\
&= O(\kappa^{(n-2)p-1})
\end{aligned}
\end{equation}
and
\begin{multline}\label{eq-hd-8}
\left| \kappa^{n-1} \int_{\pa \Omega_{\kappa} \setminus \left\{ (\bw, w_n): |\bw| < {\vep_1 \over \kappa},\, w_n
= {f(\kappa \bw) \over \kappa} \right\}} \frac{\pa G}{\pa \nu_x}(\kappa z,\kappa w) \mch_{\kappa}(w) dS_w \right| \\
\le C \int_{\pa \Omega_{\kappa} \setminus \left\{ (\bw, w_n): |\bw| < {\vep_1 \over \kappa},\, w_n
= {f(\kappa \bw) \over \kappa} \right\}} \frac{|z-w|}{|z-w|^{(n-2)(p+1)}} dS_w = O(\kappa^{(n-2)p-2}).
\end{multline}
Estimates \eqref{eq-hd-6}-\eqref{eq-hd-8} together show that $\mce_{\kappa 2} \to 0$ uniformly in $B^n(e_n,\frac{1}{4})$ as $\kappa \to 0$.
With \eqref{eq-h1-4} in hand, one can also check that the derivative of $\mce_{\kappa 2}$ converges to 0 uniformly in $B^n(e_n,\frac{1}{4})$.

\medskip
From the above estimates on $\mce_{\kappa 1}$ and $\mce_{\kappa 2}$, we conclude that $\mch_{\kappa} \to \mch_0$ in $C^1(B^n(e_n, \frac{1}{4}))$ as $\kappa \to 0$ for all $p \in [\frac{n-1}{n-2}, \frac{n}{n-2})$.
Treating the remaining range $p \in (\frac{2}{n-2}, \frac{n-1}{n-2})$ is easier, so we leave it to the reader.
The proof of Lemma \ref{lem-hd} will be completed once we justify \eqref{eq-hd-93} and \eqref{eq-hd-94}, which we now do.

\medskip \noindent
\textbf{Derivation of \eqref{eq-hd-93}.}
Given $\delta > 0$ small enough, let $x \in \Omega$ be a point such that $\mfd(x) < \delta$.
We know that there exists a unique element $x' = (\bx', f(\bx')) \in \pa \Omega$ such that
\begin{equation}\label{eq-hd-930}
x = x' + \mfd(x) \nu_x
\quad \text{where }
\nu_x = (\nu_{x1}, \cdots, \nu_{xn}) = {(-\nabla_{\bx} f(\bx'), 1) \over \sqrt{1+ |\nabla_{\bx} f(\bx')|^2}} \in \S^{n-1}.
\end{equation}
Set $x^*(x) = (x^*_1(x), \cdots, x^*_n(x)) = x + 2\mfd(x)\nu_x$.
For any fixed point $w_1 = (\bw_1, w_{1n}) \in (\Omega_{\kappa} \setminus \overline{\Omega_{\kappa}^{\delta_0}}) \cap Q^n(\frac{\vep_1}{\kappa})$ and $j = 1, \cdots, n$, we see from \eqref{eq-h1-2} that
\begin{multline}\label{eq-hd-931}
c_n^{-1} {\kappa}^{n-1} \pa_{x_j} H(\kappa w_1, \kappa e_n) \\
= -(n-2) \frac{\kappa^{-1}(\kappa w_1)^*-e_n}{|\kappa^{-1}(\kappa w_1)^*-e_n|^n} \cdot \pa_{x_j} x^*(\kappa w_1) + O\(\frac{\kappa}{|\kappa^{-1}(\kappa w_1)^*-e_n|^{n-2}}\).
\end{multline}

We expand $\kappa^{-1}(\kappa w_1)^*$. If we write $x = \kappa w_1$, using the shape of the function $f: B^{n-1}(0,\vep_1) \to (-\vep_1,\vep_1)$ and \eqref{eq-hd-930}, we can infer that
\[\kappa w_1 = (\kappa \bw_1, \kappa w_{1n}) = (\bx', f(\bx')) + \mfd(\kappa w_1) \nu_{\kappa w_1} = (\bx' + O(\kappa \delta_0 |\bx'|), \mfd(\kappa w_1) + O(|\bx'|^2)).\]
Comparing each component and applying the implicit function theorem, we observe
\[\bx' = \kappa \bw_1 + O(\kappa^2 \delta_0 |\bw_1|)
\quad \text{and} \quad
w_{1n} 
= \kappa^{-1} \mfd(\kappa w_1) + O(\kappa|\bw_1|^2).\]
Inserting this into
\[(\kappa w_1)^* = (\bx', f(\bx')) - \mfd(\kappa w_1) \nu_{\kappa w_1} = (\bx' + O(\kappa \delta_0 |\bx'|), -\mfd(\kappa w_1) + O(|\bx'|^2)),\]
we get
\begin{equation}\label{eq-hd-932}
\kappa^{-1} (\kappa w_1)^* = (\bw_1 + O(\kappa \delta_0 |\bw_1|), - w_{1n} + O(\kappa |\bw_1|^2)).
\end{equation}

We next calculate $\pa_{x_j} x^*(\kappa w_1)$. For $i, j = 1, \cdots, n$, there holds that
\[\pa_{x_j} x^*_i(x) = \delta_{ij} - 2\nu_{xi}\nu_{xj} + 2\mfd(x)\pa_{x_j}\nu_{xi}\]
where $\delta_{ij}$ is the Kronecker delta.
By estimating $\pa_{x_j}\nu_{xi}$ with \eqref{eq-hd-930} and  the implicit function theorem, we deduce
\begin{equation}\label{eq-hd-933}
\pa_{x_j} x^*_i(\kappa w_1) = \begin{cases}
\delta_{ij} + O(\kappa \delta_0) + O(\kappa^2 |\bw_1|^2) &\text{if } 1 \le i, j \le n-1,\\
O(\kappa |\bw_1|) &\text{if } i = n \text{ and } 1 \le j \le n-1,\\
O(\kappa |\bw_1|) &\text{if } 1 \le i \le n-1 \text{ and } j = n,\\
- 1 + O(\kappa^2 |\bw_1|^2) &\text{if } i = j = n.
\end{cases}
\end{equation}

Putting \eqref{eq-hd-932} and \eqref{eq-hd-933} into \eqref{eq-hd-931} and using the mean value theorem, we find
\[c_n^{-1} {\kappa}^{n-1} \pa_{x_j} H(\kappa w_1, \kappa e_n) = -(n-2) \frac{w_1+e_n}{|w_1+e_n|^n}
+ O\(\frac{\kappa(1+|w_1|^2)}{|w_1+e_n|^n}\).\]
Hence we obtain \eqref{eq-hd-93} from \eqref{eq-hd-12} and \eqref{eq-hd-90}.

\medskip \noindent
\textbf{Derivation of \eqref{eq-hd-94}.}
By \eqref{eq-hd-12}, \eqref{eq-hd-90} and \eqref{eq-g1-9},
\[\mcf_{\kappa 2}(w) = O\(\frac{|w-e_n|}{|w-e_n|^{(n-2)(p-1)}} \cdot \frac{|w+e_n|}{|w+e_n|^n}\) = O\(\frac{1}{|w+e_n|^{(n-2)p}}\) = O(\kappa^{(n-2)p})\]
for $w \in (\Omega_{\kappa} \setminus \overline{\Omega_{\kappa}^{\delta_0}}) \setminus Q^n(\frac{\vep_1}{\kappa})$.
\end{proof}
\noindent The function $\mch_{\kappa}$ and the domain $\Omega_{\kappa}$ depend on the point $x_0 \in \Omega$ implicitly.
Nonetheless, the previous proof shows that $\mch_{\kappa} \to \mch_0$ (or equivalently, $\mce_{\kappa} \to 0$)
in $C^1(B^n(e_n, \frac{1}{4}))$-uniformly in $x_0$ provided that $\pa \Omega$ is of class $C^2$;
notice that the principal curvatures of $\pa \Omega$ are well-defined and uniformly bounded if $\pa \Omega \in C^2$.

\begin{proof}[Proof of Lemma \ref{lem-h2-r}]
As before, let $x_0$ be an arbitrary point near $\pa \Omega$ identified with $\kappa e_n = \mfd(x_0) e_n \in \R^n_+$. Then
\begin{align*}
\nu_{x_0} \cdot \nabla_x \wth(x,x_0)|_{x = x_0} &= - \left. \frac{\pa \wth}{\pa x_n}(x,x_0) \right|_{x = x_0}
= - \left. \frac{\pa \wth}{\pa x_n}(\kappa z, \kappa e_n) \right|_{z = e_n}\\
&= - \kappa^{1-(n-2)p} \frac{\pa \mch_{\kappa}}{\pa z_n}(e_n) \quad \text{(by \eqref{eq-g-0})}\\
&= \mfd(x_0)^{1-(n-2)p} \(-\frac{\pa \mch_0}{\pa z_n}(e_n) + o(1)\) \quad \text{(by Lemma \ref{lem-hd})}.
\end{align*}
On the other hand, Green's representation formula gives us that
\[\mch_0(z) = \int_{\R^n_+} \(\frac{c_n}{|z-x|^{n-2}} - \frac{c_n}{|z-\tilde{x}|^{n-2}}\) (-\Delta \mch_0)(x) dx
+ 2(n-2)c_n \int_{\R^{n-1}} \frac{z_n}{|z-\bx|^n} \mch_0(\bx) d\bx\]
where $x = (\bx, x_n) \in \R^n_+$ and $\tilde{x}$ is the reflection of $x$ with respect to $\R^{n-1}$.
Differentiating it with respect to the $z_n$-variable and putting \eqref{eq-h0-1}, \eqref{eq-h0-2} and $z = e_n$ into the result,
we find that either \eqref{eq-as-0} or \eqref{eq-as-1} is equivalent to $-\frac{\pa \mch_0}{\pa z_n}(e_n) > 0$ according to the value of $p \in (\frac{2}{n-2}, \frac{n}{n-2})$.
\end{proof}

\subsection{Verification of \eqref{eq-as-0} and \eqref{eq-as-1}} \label{subsec-ver}
To establish Proposition \ref{prop-h2}, it remains to check \eqref{eq-as-0} and \eqref{eq-as-1}.
We consider three mutually exclusive cases in order.

\medskip \noindent
\textbf{Case 1.} Assume that $n \ge 5$ and $p \in (\frac{2}{n-2}, 1]$.

\begin{proof}[Proof of Proposition \ref{prop-h2} (Case 1)]
Let $\mcv_L$ and $\mcv_R$ be the left-hand and right-hand sides of \eqref{eq-as-0}, respectively.

We have
\begin{equation}\label{eq-mcv-1}
\mcv_L = c_n^p |\S^{n-2}| \int_0^{\infty} \int_0^{\infty} r^{n-2} \left[ \frac{1-t}{(r^2+(t-1)^2)^{n/2}} - \frac{1+t}{(r^2+(t+1)^2)^{n/2}} \right] \mck_1(r,t) dr dt
\end{equation}
where
\begin{equation}\label{eq-mcv-5}
\mck_1(r,t) := \frac{1}{(r^2+(t-1)^2)^{(n-2)p/2}} - \left[ \frac{1}{(r^2+(t-1)^2)^{(n-2)/2}} - \frac{1}{(r^2+(t+1)^2)^{(n-2)/2}} \right]^p.
\end{equation}
Owing to Lemma \ref{lem-cal-1} (1),
\begin{equation}\label{eq-mcv-4}
0 \le \mck_1(r,t) \le \frac{1}{(r^2+(t+1)^2)^{(n-2)p/2}} \quad \text{for } r, t \in (0, \infty).
\end{equation}
By making substitutions $t \mapsto t^{-1}$ for $t \in (1, \infty)$ and then $rt \mapsto r$, we conclude that
\begin{multline}\label{eq-mcv-3}
\(\int_0^1 + \int_1^{\infty}\) \int_0^{\infty} r^{n-2} \frac{1-t}{(r^2+(t-1)^2)^{n/2}} \mck_1(r,t) dr dt \\
= \int_0^1 \int_0^{\infty} r^{n-2} \frac{(1-t) (1-t^{(n-2)p-2})}{(r^2+(t-1)^2)^{n/2}} \mck_1(r,t) dr dt \ge 0.
\end{multline}
From \eqref{eq-mcv-1}, \eqref{eq-mcv-3} and \eqref{eq-mcv-4}, we obtain
\begin{align*}
\mcv_L &\ge - c_n^p |\S^{n-2}| \int_0^{\infty} \int_0^{\infty} r^{n-2} \frac{1+t}{(r^2+(t+1)^2)^{((n-2)p+n)/2}} dr dt \\
&= - c_n^p |\S^{n-2}| \int_0^{\infty} {dt \over (t+1)^{(n-2)p}} \int_0^{\infty} \frac{r^{n-2}}{(r^2+1)^{((n-2)p+n)/2}} dr \quad \text{(substitute $(t+1)r$ for $r$)} \\
&= - \frac{c_n^p |\S^{n-2}|}{2[(n-2)p-1]} \cdot \text{B}\({n-1 \over 2}, {(n-2)p+1 \over 2}\).
\end{align*}

On the other hand,
\begin{align}
\mcv_R &= \frac{2 c_n^p |\S^{n-2}|}{[(n-2)p-2][n-(n-2)p]} \int_0^{\infty} r^{n-2} \left[\frac{1}{(r^2+1)^{(n-2)(p+1)/2}} - \frac{n}{(r^2+1)^{((n-2)p+n)/2}} \right] dr \nonumber \\
&= - \frac{c_n^p |\S^{n-2}|}{n-(n-2)p} \cdot \text{B}\({n+1 \over 2}, {(n-2)p-1 \over 2}\). \label{eq-mcv-2}
\end{align}

As a consequence,
\[\mcv_L - \mcv_R \ge \frac{(n-2)(p+1) c_n^p |\S^{n-2}|}{2(n-1)[n-(n-2)p]} \cdot
\text{B}\({n+1 \over 2}, {(n-2)p-1 \over 2}\) > 0\]
as desired.
\end{proof}

\medskip \noindent
\textbf{Case 2.} Assume that $n \ge 5$ and $p \in (1, \frac{n-1}{n-2})$.

\begin{proof}[Proof of Proposition \ref{prop-h2} (Case 2)]
As in the previous case, let $\mcv_L$ and $\mcv_R$ be the left-hand and right-hand sides of \eqref{eq-as-0} so that they are the same as \eqref{eq-mcv-1} and \eqref{eq-mcv-2}, respectively.

By Lemma \ref{lem-cal-1} (2), we have that $\mck_1(r,t) \ge 0$ and
\begin{align}
-(p-1) \frac{(r^2+(t-1)^2)^{(n-2)(2-p)/2}}{(r^2+(t+1)^2)^{n-2}}
&\le \mck_1(r,t) - \frac{p}{(r^2+(t-1)^2)^{(n-2)(p-1)/2}} \frac{1}{(r^2+(t+1)^2)^{(n-2)/2}} \nonumber \\
&\le 0 \label{eq-mcv-6}
\end{align}
for $r, t \in (0, \infty)$. Hence \eqref{eq-mcv-3} shows
\begin{align*}
\mcv_L &\ge - p c_n^p |\S^{n-2}| \int_0^{\infty} \int_0^{\infty} r^{n-2} \frac{1+t}{(r^2+(t+1)^2)^{n-1}} \frac{1}{(r^2+(t-1)^2)^{(n-2)(p-1)/2}} dr dt \\
&> - p c_n^p |\S^{n-2}| \int_0^{\infty} \int_0^{\infty} r^{(n-2)(2-p)} \frac{1+t}{(r^2+(t+1)^2)^{n-1}} dr dt \\
&= - \frac{p c_n^p |\S^{n-2}|}{4(n-2)} \cdot \text{B}\({2n-3-(n-2)p \over 2}, {(n-2)p-1 \over 2}\).
\end{align*}

Using the condition that $p \in (1, \frac{n-1}{n-2})$, we deduce
\begin{align*}
\mcv_L - \mcv_R &> c_n^p |\S^{n-2}| \Gamma\({(n-2)p-1 \over 2}\)
\left[\frac{1}{n-(n-2)p} \cdot \frac{\Gamma\({n+1 \over 2}\)}{\Gamma\({(n-2)p+n \over 2}\)}
- \frac{p}{4} \cdot \frac{\Gamma\({2n-3-(n-2)p \over 2}\)}{\Gamma(n-1)}\right] \\
&> c_n^p |\S^{n-2}| \Gamma\({(n-2)p-1 \over 2}\)
\left[\frac{1}{2} \cdot \frac{\Gamma\({n+1 \over 2}\)}{\Gamma\({2n-1 \over 2}\)}
- \frac{n-1}{4(n-2)} \cdot \frac{\Gamma\({n-1 \over 2}\)}{\Gamma(n-1)}\right].
\end{align*}
By applying induction on $n \in \N$, we observe that the quantity surrounded by parentheses in the rightmost side is positive for all $n \ge 5$. Thus $\mcv_L - \mcv_R > 0$.
\end{proof}

\medskip \noindent
\textbf{Case 3.} Assume that $n \ge 5$ and $p \in [\frac{n-1}{n-2}, \frac{n}{n-2})$.

Let $\mcw_L$ and $\mcw_R$ be the left-hand and right-hand sides of \eqref{eq-as-1}, respectively, and $\mck_1(r,t)$ the function defined in \eqref{eq-mcv-5}. It holds that
\[\mcw_L = c_n^p |\S^{n-2}| \int_0^{\infty} \int_0^{\infty} r^{n-2} \left[ \frac{1-t}{(r^2+(t-1)^2)^{n/2}} - \frac{1+t}{(r^2+(t+1)^2)^{n/2}} \right] \mck_2(r,t) drdt\]
where
\begin{align*}
\mck_2(r,t) &:= \mck_1(r,t) - \frac{p}{(r^2+(t-1)^2)^{(n-2)(p-1)/2}} \frac{1}{(r^2+(t+1)^2)^{(n-2)/2}} \\
&\ + \frac{2(n-2)p}{2(n-1)-(n-2)p} \frac{1}{(r^2+(t-1)^2)^{(n-2)(p-1)/2}}
\frac{r^2+t^2-1}{(r^2+(t+1)^2)^{n/2}}.
\end{align*}
Arguing as for \eqref{eq-mcv-3} and using the nonnegativity of $\mck_1$ and \eqref{eq-mcv-6}, we discover
\begin{align*}
\mcw_L &\ge c_n^p |\S^{n-2}| \left[-(p-1) \int_0^1 \int_0^{\infty} r^{n-2} \frac{(1-t) (1-t^{(n-2)p-2})}{(r^2+(t-1)^2)^{((n-2)(p-1)+2)/2}} \frac{1}{(r^2+(t+1)^2)^{n-2}} drdt \right.\\
&\ + \frac{2(n-2)p}{2(n-1)-(n-2)p} \int_0^{\infty} \int_0^{\infty} r^{n-2} \frac{1-t}{(r^2+(t-1)^2)^{((n-2)(p-1)+n)/2}} \frac{r^2+t^2-1}{(r^2+(t+1)^2)^{n/2}} dr dt \\
&\ \left. - \frac{2(n-2)p}{2(n-1)-(n-2)p} \int_0^{\infty} \int_0^{\infty} r^{n-2} \frac{1+t}{(r^2+(t-1)^2)^{(n-2)(p-1)/2}} \frac{r^2+t^2-1}{(r^2+(t+1)^2)^n} drdt \right] \\
&=: c_n^p |\S^{n-2}| \left[ -\mcx_1 + \frac{2(n-2)p}{2(n-1)-(n-2)p} \mcx_2 - \frac{2(n-2)p}{2(n-1)-(n-2)p} \mcx_3 \right].
\end{align*}
Also,
\[\mcw_R = - \frac{c_n^p |\S^{n-2}|}{n-(n-2)p} \left[ 1 + \frac{[(n-2)p-2]p}{2(n-1)-(n-2)p} \right] \mcy\]
where
\[\mcy := \text{B}\({n+1 \over 2}, {(n-2)p-1 \over 2}\).\]
Therefore, to deduce the desired inequality $\mcw_L - \mcw_R > 0$, it suffices to verify
\begin{equation}\label{eq-mcv-7}
\frac{2(n-1)-(n-2)p}{2(n-2)p} \mcx_1 - \mcx_2 + \mcx_3 < \frac{1}{2(n-2)p} \left[ \frac{2(n-1)-np + (n-2)p^2}{n-(n-2)p} \right] \mcy.
\end{equation}

In Lemmas \ref{lem-mcx-1}-\ref{lem-mcy}, we obtain estimates for $\mcx_1$, $\mcx_2$, $\mcx_3$ and $\mcy$.
\begin{lem}\label{lem-mcx-1}
Suppose that $p \in [\frac{n-1}{n-2}, \frac{n}{n-2})$. Then
\begin{equation}\label{eq-b30}
\mcx_1 < {p-1 \over 2[(n-2)p-2]} \cdot \textnormal{B}\({2n-3-(n-2)p \over 2}, {(n-2)p-1 \over 2}\)
\end{equation}
for all $n \ge 5$. In particular,
\begin{equation}\label{eq-b35}
\frac{2(n-1)-(n-2)p}{2(n-2)p} \mcx_1 < \frac{17}{16n(n-2)} \cdot \frac{1}{2^n}
\end{equation}
for all $n \ge 100$.
\end{lem}
\begin{proof}
Arguing as \eqref{eq-mcv-3}, we get
\begin{align*}
\mcx_1 &= (p-1) \int_0^{\infty} \int_0^{\infty} r^{n-2} \frac{1-t}{(r^2+(t-1)^2)^{((n-2)(p-1)+2)/2}} \frac{1}{(r^2+(t+1)^2)^{n-2}} drdt \\
&< (p-1) \int_0^{\infty} \int_0^{\infty} r^{(n-2)(2-p)} \frac{1}{(r^2+(t+1)^2)^{n-2}} drdt.
\end{align*}
The value of the rightmost integral equals to the right-hand side of \eqref{eq-b30}.

Besides, if we set $\zeta = (n-2)p-n+1 \in [0,1)$, then
\begin{equation}\label{eq-b20}
\begin{aligned}
\text{B}\(\frac{n-2-\zeta}{2}, \frac{n-2+\zeta}{2}\)
&= \int_0^1 t^{\frac{n-4-\zeta}{2}} (1-t)^{\frac{n-4+\zeta}{2}} dt \\
&= \frac{1}{2^{n-3}} \int_{-1}^1 (1+t)^{\frac{n-4-\zeta}{2}} (1-t)^{\frac{n-4+\zeta}{2}} dt \\
&=\frac{1}{2^{n-3}} \int_0^1 (1-t^2)^{\frac{n-4-\zeta}{2}} \left[(1-t)^{\zeta} + (1+t)^{\zeta}\right] dt \\
&\le \frac{1}{2^{n-3}} \int_0^1 (1-t^2)^{\frac{n-5}{2}} (2+t) dt
\end{aligned}
\end{equation}
where the substitution $t \mapsto \frac{1}{2}(1+t)$ was made to derive the second equality. Therefore
\begin{equation}\label{eq-b10}
\text{B}\(\frac{n-2-\zeta}{2}, \frac{n-2+\zeta}{2}\) \le \frac{1}{2^{n-3}} \int_0^1 (1-t^2)^{\frac{95}{2}} (2+t) dt \le \frac{17}{2^{n+3}}
\end{equation}
for all $n \ge 100$. This together with \eqref{eq-b30} and the inequalities
\[0 < \frac{2(n-1)-(n-2)p}{2(n-2)p} \cdot {p-1 \over 2[(n-2)p-2]} < \frac{1}{2n(n-2)}\]
implies \eqref{eq-b35}.
\end{proof}

Estimating the integral $\mcx_2$ is exceptionally difficult, because one needs to balance two factors in the denominator of the integrand: $(r^2 + (t-1)^2)^{((n-2)(p-1)+n)/2}$ and $(r^2 + (t+1)^2)^{n/2}$.
Coordinate changes such as \eqref{eq-R} turn out to be effective.
\begin{lem}
Suppose that $p \in [\frac{n-1}{n-2}, \frac{n}{n-2})$. Then
\begin{align}
- \mcx_2 &< \frac{1}{2^n} \left[ \frac{4}{n-2} + \frac{4}{(n-1)^2} + \frac{9391}{360(n-1)} + \frac{1}{2n} + \frac{981}{200} \(\frac{8}{9}\)^{n \over 2} \right. \nonumber \\
&\hspace{30pt} \left. - \frac{1}{n-1} \left\{\frac{64}{9} \({23 \over 24}\)^{n-1 \over 2} + \frac{589}{40} \({19 \over 20}\)^{n-1 \over 2} \right\}
- \frac{9}{n-2} \(\frac{2}{3}\)^n + \frac{18}{n-3} \(\frac{2}{3}\)^n \right] \nonumber \\
&+ \frac{1}{2^n} \frac{1}{n-(n-2)p} \left[\frac{1}{n-1} \(3+ \frac{64}{3\sqrt{3}}\)
+ \frac{223}{200} \(\frac{8}{9}\)^{n \over 2} - \frac{128}{3\sqrt{11}(n-1)} \({11 \over 12}\)^{n \over 2} \right. \nonumber \\
&\hspace{90pt} \left. - \frac{2\sqrt{2}}{n-1} \(\frac{1}{2}\)^{n \over 2} \right] \label{eq-mcv-9}
\end{align}
for all $n \ge 5$. In particular,
\begin{equation}\label{eq-b40}
-\mcx_2 < \frac{1}{n-2} \(\frac{983}{32} + \frac{16}{n-(n-2)p}\) \cdot \frac{1}{2^n}
\end{equation}
for all $n \ge 100$.
\end{lem}
\begin{proof}
We split the integral $\mcx_2$ as
\begin{align}
\mcx_2 &:= \mcx_{2,[0,1)} + \mcx_{2,[1,2)} + \mcx_{2,[2,\infty)} \label{eq-mcx_2} \\
&:= \( \int_0^1 + \int_1^2 + \int_2^{\infty} \) \int_0^{\infty} r^{n-2} \frac{1-t}{(r^2+(t-1)^2)^{((n-2)(p-1)+n)/2}} \frac{r^2+t^2-1}{(r^2+(t+1)^2)^{n/2}} dr dt. \nonumber
\end{align}

\medskip
We estimate the term $\mcx_{2,[0,1)}$ that is the most delicate.
Substituting $t$ with $1-t$ and then $r$ with $rt$, we get
\begin{align}
\mcx_{2,[0,1)} &= \int_0^1 \int_0^{\infty} r^{n-2} \frac{1-t}{(r^2+(t-1)^2)^{((n-2)(p-1)+n)/2}} \frac{r^2+t^2-1}{(r^2+(t+1)^2)^{n/2}} dr dt \nonumber \\
&= \int_0^1 \(\int_0^{2\sqrt{2} \over \sqrt{t}} + \int_{2\sqrt{2} \over \sqrt{t}}^{\infty}\) \frac{r^{n-2}}{t^{(n-2)(p-1)}}
\frac{1}{(r^2+1)^{((n-2)(p-1)+n)/2}} \frac{r^2t^2 + t^2-2t}{(r^2t^2+(t-2)^2)^{n/2}} dr dt \nonumber \\
&=: \mcj_1 + \mcj_2. \label{eq-mcj-0}
\end{align}

If we set
\begin{equation}\label{eq-R}
R = t(r^2+1)-2 \quad \text{for } r \in \(0, \frac{2\sqrt{2}}{\sqrt{t}}\),
\end{equation}
then $R \in (t-2, t+6)$ and
\begin{align}
\mcj_1 &= \frac{1}{2^{n+1}} \int_0^1 \int_{t-2}^{t+6} \frac{t^{(3-(n-2)(p-1))/2} R} {(R+2)^{((n-2)(p-1)+3)/2}}
\left[ \frac{1- \frac{t}{R+2}}{1 + \frac{t(R-2)}{4}} \right]^{\frac{n-3}{2}} \frac{1}{[1 + \frac{t(R-2)}{4}]^{3/2}} dR dt \label{eq-mcj-1} \\
&> - \frac{1}{2^{n+1}} \int_0^1 \int_t^2 t^{\frac{3-(n-2)(p-1)}{2}} \frac{2-R} {R^{((n-2)(p-1)+3)/2}}
\left[ \frac{1- \frac{t}{R}}{1 + \frac{t(R-4)}{4}} \right]^{\frac{n-3}{2}} \frac{1}{[1 + \frac{t(R-4)}{4}]^{3/2}} dR dt. \nonumber
\end{align}
Write the rightmost term in \eqref{eq-mcj-1} as $\mcj_1'$. We have
\begin{align}
- \mcj_1 < - \mcj_1' &\le \frac{1}{2^{n+1}} \left[ \int_0^{1 \over 3} \int_t^{3t}
t^{\frac{3-(n-2)(p-1)}{2}} \frac{2}{R^{((n-2)(p-1)+3)/2}} \(\frac{8}{9}\)^{\frac{n-3}{2}} \(\frac{6}{5}\)^3 dR dt \right. \nonumber \\
&\hspace{40pt} + \int_{1 \over 3}^1 \int_t^{\frac{3}{4}(t+1)}
t^{\frac{3-(n-2)(p-1)}{2}}
\frac{2}{R^{((n-2)(p-1)+3)/2}} \(\frac{8}{9}\)^{\frac{n-3}{2}} 8\, dR dt \nonumber \\
&\hspace{40pt} + \int_0^{1 \over 3} \int_{3t}^1 t^{\frac{3-(n-2)(p-1)}{2}}
\frac{2}{R^{((n-2)(p-1)+3)/2}} \(1-\frac{t}{4R}\)^{\frac{n-3}{2}} \(\frac{4}{3}\)^{3 \over 2} dR dt \label{eq-mcj-21} \\
&\hspace{40pt} + \int_0^{1 \over 3} \int_1^2 t^{\frac{3-(n-2)(p-1)}{2}}
(2-R) \left\{ 1 - \frac{t(2-R)^2}{8} \right\}^{\frac{n-3}{2}} \(\frac{4}{3}\)^{3 \over 2} dR dt \nonumber \\
&\hspace{40pt} \left. + \int_{1 \over 3}^1 \int_{\frac{3}{4}(t+1)}^2 t^{\frac{3-(n-2)(p-1)}{2}}
(2-R) \left\{ 1 - \frac{t(2-R)^2}{8} \right\}^{\frac{n-3}{2}} \(\frac{8}{3}\)^{3 \over 2} dR dt \right]. \nonumber
\end{align}
Computing each term, we obtain
\begin{equation}\label{eq-mcj-2}
\begin{aligned}
- \mcj_1 < - \mcj_1' &\le \frac{1}{2^{n+1}} \left[ \frac{1}{n-(n-2)p} \cdot \frac{223}{100} \(\frac{8}{9}\)^{n \over 2}
+ \frac{981}{100} \(\frac{8}{9}\)^{n \over 2} \right.\\
&\hspace{40pt} + \frac{1}{n-(n-2)p} \cdot \frac{128}{3\sqrt{3}(n-1)} \left\{1 - \({11 \over 12}\)^{n-1 \over 2} \right\} \\
&\hspace{40pt} \left. + \frac{128}{9(n-1)} \left\{1 - \({23 \over 24}\)^{n-1 \over 2} \right\}
+ \frac{589}{20(n-1)} \left\{1 - \({19 \over 20}\)^{n-1 \over 2} \right\} \right].
\end{aligned}
\end{equation}
In dealing with the third integral in \eqref{eq-mcj-21}, we substituted ${t \over R}$ with $R$ and extended the domain of the resulting integral to be $(0,\frac{1}{3}) \times (0,\frac{1}{3})$.
For the fourth and fifth integrals, we substituted $t(2-R)^2$ with $R$.

On the other hand, because
\[r^2 t^2 + (t-2)^2 > \begin{cases}
(t+2)^2 &\text{for } r > \frac{2\sqrt{2}}{\sqrt{t}},\\
\frac{r^2t^2}{2}+(t-2)^2+4 &\text{for } r > \frac{2\sqrt{2}}{t},
\end{cases}\]
we have
\begin{align*}
|\mcj_2| &\le \int_0^1 \frac{1}{t^{(n-2)(p-1)}}
\left[ \frac{t^2}{(t+2)^{n-1}} \int_{2\sqrt{2} \over \sqrt{t}}^{2\sqrt{2} \over t} \frac{1}{r^{(n-2)(p-1)}} dr \right. \\
&\hspace{80pt} + \frac{t^{(n-2)(p-1)+1}}{((t-2)^2+4)^{((n-2)(p-1)+n-1)/2}}
\int_{2 \over \sqrt{(t-2)^2+4}}^{\infty} \frac{1}{r^{(n-2)(p-1)}} \frac{1}{(r^2+1)^{n/2}} dr \\
&\hspace{80pt} \left. + \frac{t}{(t+2)^{n-1}} \int_{2\sqrt{2} \over \sqrt{t}}^{\infty} \frac{1}{r^{(n-2)(p-1)+2}} dr \right] dt.
\end{align*}
Estimating each term in the right-hand side, we deduce
\begin{equation}\label{eq-mcj-3}
|\mcj_2| \le \frac{1}{2^n} \left[ \frac{4}{n-1} + \frac{1}{2n} + \frac{1}{4(n-1)} \right] = \frac{1}{2^n} \cdot \frac{19n-2}{4n(n-1)}.
\end{equation}

\medskip
Secondly, we handle the term $\mcx_{2,[1,2)}$. As in \eqref{eq-mcj-0}, we write
\begin{align*}
\mcx_{2,[1,2)} &= \int_1^2 \int_0^{\infty} r^{n-2} \frac{1-t}{(r^2+(t-1)^2)^{((n-2)(p-1)+n)/2}} \frac{r^2+t^2-1}{(r^2+(t+1)^2)^{n/2}} dr dt \\
&= \int_0^1 \(\int_0^1 + \int_1^{\infty}\) \frac{r^{n-2}}{t^{(n-2)(p-1)}}
\frac{-1}{(r^2+1)^{((n-2)(p-1)+n)/2}} \frac{r^2t^2 + t^2+2t}{(r^2t^2+(t+2)^2)^{n/2}} dr dt \\
&=: \mcj_3 + \mcj_4.
\end{align*}
The term $\mcj_3$ can be computed as
\begin{equation}\label{eq-mcj-4}
\begin{aligned}
0 < -\mcj_3 &\le \int_0^1 \int_0^1 \frac{r^{n-2}(r^2+1)}{(t+2)^n} dr dt + 2 \int_0^1 \int_0^1 \frac{r^{n-2} t^{1-(n-2)(p-1)}}{(t+2)^n} dr dt \\
&\le \frac{1}{2^n (n-1)} \left[ \frac{4}{n-1} + \frac{1}{n-(n-2)p} \right].
\end{aligned}
\end{equation}
Let $\mcj_{41}$ and $\mcj_{42}$ be the integrals obtained by replacing the term $r^2t^2 + t^2 + 2t$ with $(r^2+1)t^2$ and $2t$, respectively.
Then clearly $\mcj_4 = \mcj_{41} + \mcj_{42}$ and there holds that
\begin{equation}\label{eq-mcj-51}
\begin{aligned}
0 < -\mcj_{41} &\le \int_0^1 \int_1^{\infty} \frac{r^{n-4} (r^2 + 1)}{(r^2+1)^{n/2}} \frac{(rt)^{2-(n-2)(p-1)}}{(r^2 t^2 + (t+2)^2 )^{n/2}} drdt \\
&\le \int_0^1 \int_1^{\infty} \frac{1}{r^2} \frac{1}{(t+2)^{n-1}} dr dt = \frac{1}{2^n} \frac{4}{n-2} \left\{1 - \(\frac{2}{3}\)^{n-2} \right\}
\end{aligned}
\end{equation}
and
\begin{equation}\label{eq-mcj-52}
\begin{aligned}
0 < -\mcj_{42} &\le 2 \int_0^1 \int_1^{\infty} r^{n-2} t^{1-(n-2)(p-1)} \frac{1}{(r^2+1)^{((n-2)(p-1)+n)/2}} \frac{1}{2^n} dr dt \\
&\le \frac{1}{2^{n-1}} \frac{1}{n-(n-2)p} \int_1^{\infty} \frac{r^{n-2}}{(r^2+1)^{(n+1)/2}} dr \\
&= \frac{1}{2^{n-1}} \frac{1}{n-(n-2)p} \cdot \frac{1}{n-1} \left\{1 - \(\frac{1}{2}\)^{n-1 \over 2} \right\}.
\end{aligned}
\end{equation}

\medskip
Finally,
\begin{equation}\label{eq-mcj-6}
\begin{aligned}
\mcx_{2,[2,\infty)} &:= \int_2^{\infty} \int_0^{\infty} r^{n-2} \frac{1-t}{(r^2+(t-1)^2)^{((n-2)(p-1)+n)/2}} \frac{r^2+t^2-1}{(r^2+(t+1)^2)^{n/2}} dr dt \\
&\ge - \int_2^{\infty}\int_0^{\infty} r^{n-2} \frac{t-1}{(r^2 + (t-1)^2)^{((n-2)(p-1)+n)/2}} \frac{1}{(r^2 + (t+1)^2)^{(n-2)/2}} drdt \\
&\ge - \int_2^{\infty} \int_0^{\infty} \frac{r^{n-2}}{(r^2+1)^{n/2}} \frac{1}{(t+1)^{n-2}} dr dt
\ge - \frac{1}{3^n} \frac{18}{n-3}.
\end{aligned}
\end{equation}

\medskip
Inequality \eqref{eq-mcv-9} now follows from \eqref{eq-mcx_2}-\eqref{eq-mcj-1} and \eqref{eq-mcj-2}-\eqref{eq-mcj-6}.
The values of the quantities surrounded by the parentheses in \eqref{eq-mcv-9} are obviously of order $O(\frac{1}{n})$ as $n \to \infty$.
In fact, they are bounded by $\frac{983}{32(n-2)}$ and $\frac{16}{n-2}$ provided $n \ge 100$, respectively. Hence \eqref{eq-b40} holds.
\end{proof}

\begin{lem}\label{lemma-mcx_3}
Suppose that $p \in [\frac{n-1}{n-2}, \frac{n}{n-2})$. Then
\begin{equation}\label{eq-mcx_3-1}
\mcx_3 < \frac{1}{4(n-2)} \cdot \textnormal{B}\({2n-3-(n-2)p \over 2}, {(n-2)p-1 \over 2}\)
\end{equation}
for all $n \ge 5$. In particular,
\begin{equation}\label{eq-mcx_3-2}
\mcx_3 < \frac{17}{32(n-2)} \cdot \frac{1}{2^n}
\end{equation}
for all $n \ge 100$.
\end{lem}
\begin{proof}
Since $r^2+t^2-1 = r^2 + (t+1)^2 - 2(t+1)$, we have
\begin{align*}
\mcx_3 &\le \int_0^{\infty} \int_0^{\infty} r^{n-2} \frac{1+t}{(r^2+(t-1)^2)^{(n-2)(p-1)/2}} \frac{1}{(r^2+(t+1)^2)^{n-1}} drdt \\
&< \int_0^{\infty} \int_0^{\infty} r^{(n-2)(2-p)} \frac{1+t}{(r^2+(t+1)^2)^{n-1}} drdt
= \frac{1}{4(n-2)} \cdot \text{B}\(\frac{n-2-\zeta}{2}, \frac{n-2+\zeta}{2}\)
\end{align*}
where $\zeta = (n-2)p-n+1 \in [0,1)$. Hence \eqref{eq-mcx_3-1} is valid.
Estimate \eqref{eq-mcx_3-2} follows from \eqref{eq-mcx_3-1} and \eqref{eq-b10}.
\end{proof}

\begin{lem}\label{lem-mcy}
Suppose that $p \in [\frac{n-1}{n-2}, \frac{n}{n-2})$. Then
\begin{equation}\label{eq-mcy}
\frac{1}{2(n-2)p} \left[ \frac{2(n-1)-np + (n-2)p^2}{n-(n-2)p} \right] \mcy > \frac{99}{20\sqrt{n}\, [n-(n-2)p]} \cdot \frac{1}{2^n}
\end{equation}
for all $n \ge 100$.
\end{lem}
\begin{proof}
Arguing as in \eqref{eq-b20}, we find
\[\mcy = \int_0^1 t^{\frac{n-1}{2}} (1-t)^{\frac{n-4+\zeta}{2}} dt = \frac{1}{2^{n-3/2+\zeta/2}} \int_0^1 (1-t^2)^{\frac{n-1}{2}} \left[(1-t)^{\frac{-3+\zeta}{2}} + (1+t)^{\frac{-3+\zeta}{2}}\right] dt\]
where $\zeta = (n-2)p-n+1 \in [0,1)$.
It is easy to check that the map $f(t) = (1-t)^{\frac{-3+\zeta}{2}} + (1+t)^{\frac{-3+\zeta}{2}}$ is increasing in $[0,1)$ so that $f(t) \ge 2$ in $[0,1)$.
Thus an induction argument shows
\begin{equation}\label{eq-mcy-1}
\begin{aligned}
\mcy \ge \frac{1}{2^{n-2}} \int_0^1 (1-t^2)^{\frac{n-1}{2}}dt = \frac{\sqrt{\pi}}{2^{n-1}} \frac{\Gamma\({n+1 \over 2}\)}{\Gamma\({n+2 \over 2}\)} \ge \frac{5}{\sqrt{n}} \cdot \frac{1}{2^n}
\end{aligned}
\end{equation}
for all $n \ge 100$. On the other hand, we have
\begin{equation}\label{eq-mcy-2}
\frac{2(n-1)-np + (n-2)p^2}{2p} > \frac{(n-1)(n-2)}{n} \ge \frac{99(n-2)}{100}
\end{equation}
for all $n \ge 100$. Consequently, we deduce \eqref{eq-mcy} from \eqref{eq-mcy-1} and \eqref{eq-mcy-2}.
\end{proof}

We are now in position to finish the proof of Proposition \ref{prop-h2}. \begin{proof}[Completion of the proof of Proposition \ref{prop-h2} (Case 3)]
It is enough to show that \eqref{eq-mcv-7} holds. We divide the cases according to the magnitude of $n \in \N$.

\medskip
\noindent \textbf{Case 3 (i).} Suppose that $n \ge 100$. In light of \eqref{eq-b35}, \eqref{eq-b40}, \eqref{eq-mcx_3-2} and \eqref{eq-mcy}, inequality \eqref{eq-mcv-7} is reduced to
\begin{equation}\label{eq-b50}
\frac{125}{4} + \frac{17}{16n} + \frac{16}{n-(n-2)p} \le \frac{99(n-2)}{20\sqrt{n}} \cdot \frac{1}{n-(n-2)p}.
\end{equation}
Because it holds that $n-(n-2)p \ge 1$ and
\[\frac{125}{4} + \frac{17}{16n} \le \frac{99(n-2)}{20\sqrt{n}} - 16,\]
\eqref{eq-b50} must be valid for any $p \in [\frac{n-1}{n-2}, \frac{n}{n-2})$.

\medskip
\noindent \textbf{Case 3 (ii).} Suppose that $5 \le n \le 99$.
Although estimates \eqref{eq-b30}, \eqref{eq-mcv-9} and \eqref{eq-mcx_3-1} capture the asymptotic behavior of the left-hand side of \eqref{eq-mcv-7} as $n \to \infty$ fairly well, it is possible to improve them considerably.
For example, the third integral of \eqref{eq-mcj-21} can be computed precisely.
As a result, one can further reduce \eqref{eq-mcv-7} into an inequality involving the Gauss hypergeometric function $_2F_1$.
However, the resulting inequality is too complex to check by hand, so we verify it using a computer software.
See the supplement \cite{CK2} for more details.
\end{proof}

{\footnotesize \medskip \noindent
\textbf{Acknowledgements.} This work is a full generalization of the unpublished paper `The Lane-Emden system near the critical hyperbola on nonconvex domains' (arXiv:1505.06978)  written by W. Choi.
In that paper, Theorem \ref{thm-main} in our paper was verified under the additional assumptions that $p \ge 1$ and \eqref{eq-h2} holds.
Also, the validity of \eqref{eq-h2} was shown provided that $p$ is sufficiently near 1; the proof is based on our Lemma \ref{lem-h2-r} and a perturbation argument.
W. Choi was supported by Basic Science Research Program through the National Research Foundation of Korea(NRF) funded by the Ministry of Education (NRF2017R1C1B5076348).
S. Kim was partially supported by Basic Science Research Program through the National Research Foundation of Korea(NRF) funded by the Ministry of Education (NRF2017R1C1B5076384).}


\begin{thebibliography}{20}
\bibitem{ACP}
N. Ackermann,  M. Clapp, and A. Pistoia, \emph{Boundary clustered layers near the higher critical exponents}, J. Differential Equations \textbf{254} (2013), 4168--4193.

\bibitem{BLR}
A. Bahri, Y. Y. Li, and O. Rey, \emph{On a variational problem with lack of compactness: the topological effect of the critical points at infinity}, Calc. Var. Partial Differential Equations \textbf{3} (1995), 67--93.

\bibitem{BE}
M. Ben Ayed and K. El Mehdi, \emph{On a biharmonic equation involving nearly critical exponent}, Nonlinear Differ. Equ. Appl. \textbf{13} (2006), 485--509.

\bibitem{BMR}
D. Bonheure, E. M. dos Santos, and M. Ramos, \emph{Ground state and non-ground state solutions of some strongly coupled elliptic systems}, Trans. Amer. Math. Soc. \textbf{364} (2012), 447--491.

\bibitem{CLO}
X. Chen, C. Li, and  B. Ou, \emph{Classification of solutions for a system of integral equations}, Comm. Partial Differential Equations \textbf{30} (2005) 59--65.

\bibitem{CK}
W. Choi and S. Kim, \emph{Minimal energy solutions to the fractional Lane-Emden system: Existence and singularity formation}, to appear in Rev. Mat. Iberoam.

\bibitem{CK2}
\bysame, \emph{Elliptic systems with nearly critical exponents on general domains (supplement)}, available at https://sites.google.com/site/wchoiam/publication.

\bibitem{CG}
K.S. Chou and D. Geng, \emph{Asymptotics of positive solutions for a biharmonic equation involving critical exponent}, Differential Integral Equations \textbf{13} (2000), 921--940.

\bibitem{CFM}
Ph. Cl\'ement, D. G. de Figueiredo, and E. Mitidieri, \emph{Positive solutions of semilinear elliptic systems}, Comm. Partial Differential Equations \textbf{17} (1992), 923--940.

\bibitem{FF}
D. G. de Figueiredo and P. Felmer, \emph{On superquadratic elliptic systems}, Trans. Amer. Math. Soc. \textbf{343} (1994), 99--116.

\bibitem{FLN}
D. G. de Figueiredo, P. L. Lions, and R. D. Nussbaum, \emph{A priori estimates and existence of positive solutions of semilinear elliptic equations}, J. Math. Pures Appl. \textbf{61} (1982), 41--63.

\bibitem{El}
K. El Mehdi, \emph{Single blow-up solutions for a slightly subcritical biharmonic equation}, Abstr. Appl. Anal. \textbf{2006} (2006) Article ID 18387, 20 page.

\bibitem{Ge}
D. Geng, \emph{Location of the blow up point for positive solutions of a biharmonic equation involving nearly critical exponent}, Acta Math. Sci. Ser. B Engl. Ed. \textbf{25} (2005), 283--295.

\bibitem{G}
I. A. Guerra, \emph{Solutions of an elliptic system with a nearly critical exponent}, Ann. Inst. H. Poincar\'e Anal. Non Lin\'eaire \textbf{25} (2008), 181-200.

\bibitem{H}
Z. C. Han, \emph{Asymptotic approach to singular solutions for nonlinear elliptic equations involving critical Sobolev exponent}, Ann. Inst. H. Poincar\'e Anal. Non Lin\'eaire \textbf{8} (1991), 159--174.

\bibitem{HMV}
J. Hulshof, E. Mitidieri, and R. C. A. M. Van der Vorst, \emph{Strongly indefinite systems with critical Sobolev exponents}, Trans. Amer. Math. Soc. \textbf{350} (1998), 2349--2365.

\bibitem{HV}
J. Hulshof and R. C. A. M. Van der Vorst, \emph{Differential systems with strongly indefinite structure}, J. Funct. Anal. \textbf{114} (1993), 32--58.

\bibitem{HV2}
\bysame, \emph{Asymptotic behaviour of ground states}, Proc. Amer. Math. Soc. \textbf{124} (1996), 2423--2431.

\bibitem{M1}
E. Mitidieri, \emph{A Rellich type identity and applications: Identity and applications}, Comm. Partial Differential Equations \textbf{18} (1993), 125-151.

\bibitem{MP}
M. Musso and A. Pistoia, \emph{Multispike solutions for a nonlinear elliptic problem involving the critical {S}obolev exponent}, Indiana Univ. Math. J. \textbf{51} (2002), 541--579.

\bibitem{Sou}
P. Quittner and P. Souplet, \emph{Superlinear parabolic problems: blow-up, global existence and steady states}, in: Birkh\"auser Advanced Texts: Basler Lehrb\"ucher, Birkh\"auser Verlag, Basel, 2007.

\bibitem{R}
O. Rey, \emph{The role of the {G}reen's function in a nonlinear elliptic equation involving the critical Sobolev exponent}, J. Funct. Anal. \textbf{89} (1990), 1--52.

\bibitem{R3}
\bysame, \emph{The topological impact of critical points at infinity in a variational problem with lack of compactness: The dimension 3}, Adv. Differential Equations \textbf{4} (1999), 581--616.

\bibitem{SW}
E. M. Stein and G. Weiss, \emph{Fractional integrals in $n$-dimensional Euclidean space}, J. Math. Mech., \textbf{7} (1958), 503--514.

\bibitem{V}
R. C. A. M. Van der Vorst, \emph{Variational identities and applications to differential systems}, Arch. Rational Mech. Anal. \textbf{116} (1992), 375--398.

\end{thebibliography}
\end{document}